\documentclass[10pt]{amsart}

\usepackage[usenames,dvipsnames]{xcolor}
\usepackage[bookmarks,colorlinks=true,citecolor=OliveGreen,linkcolor=RoyalBlue]{hyperref}
\usepackage{amssymb,amsthm}
\usepackage{mathtools}
\usepackage{mathabx}
\usepackage{xparse}
\usepackage{enumitem}
\usepackage{thmtools}
\usepackage[capitalise]{cleveref}		
\usepackage{bm}
\usepackage{tikz, soul}
\usetikzlibrary{calc,arrows}
\usepackage{cancel}

\numberwithin{equation}{section}

\declaretheorem[numberwithin=section,refname={Theorem,Theorems}]{theorem}
\declaretheorem[sibling=theorem,refname={Proposition,Propositions}]{proposition}
\declaretheorem[sibling=theorem,refname={Corollary,Corollaries}]{corollary}

\declaretheorem[sibling=theorem,style=definition,refname={Definition,Definitions}]{definition}

\declaretheorem[sibling=theorem,style=definition,refname={Lemma,Lemmas}]{lemma}
\declaretheorem[sibling=theorem,style=remark,refname={Remark,Remarks}]{remark}

\declaretheorem[sibling=theorem,style=remark]{notation}

\newcounter{claimCounter}[theorem]

\theoremstyle{remark}            
\newtheorem{claim}[claimCounter]{Claim}

\newcounter{LRClaim}

\newlist{equivalent}{enumerate}{1}
\setlist[equivalent,1]{label=\textup{(\arabic*)}}

\newlist{sublemma}{enumerate}{1}
\setlist[sublemma,1]{label=\textup{(\alph*)}}

\newlist{orderedlist}{enumerate}{1}
\setlist[orderedlist,1]{label=(\roman*)}

\newlist{sublemma*}{enumerate*}{1}
\setlist[sublemma*,1]{label=\textup{(\alph*)},afterlabel=\hspace{5pt}}

\newlist{orderedlist*}{enumerate*}{1}
\setlist[orderedlist*,1]{label=\textup{(\roman*)},afterlabel=\hspace{3pt}}

\newlist{TSPlist}{enumerate}{1}
\setlist[TSPlist,1]{labelindent=0em,labelwidth=4em,labelsep=0.5em,leftmargin=5em,itemindent=!,
label=\textup{TSP(\arabic*):},ref=\textup{TSP(\arabic*)}}

\newcommand{\seq}[1]{{\left\langle{#1}\right\rangle}}

\newcommand{\smallseq}[1]{\langle{#1}\rangle}

\newcommand{\conc}{\hat{\,\,}}
\newcommand{\rest}[1]{\! \upharpoonright\!{#1}} 
\newcommand{\cat}{\widehat{\phantom{\alpha}}}

\newcommand{\andd}{\,\,\,\&\,\,\,}

\newcommand{\Then}{\,\Longrightarrow\,}
\newcommand{\Iff}{\,\,\Longleftrightarrow\,\,}
\renewcommand{\iff}{\leftrightarrow}
\newcommand{\then}{\,\,\rightarrow\,\,}

\DeclareMathOperator{\dom}{dom}
\DeclareMathOperator{\range}{range}

\DeclareMathOperator{\otp}{otp}

\newcommand{\Tur}{\textup{\scriptsize T}}

\newcommand{\Nat}{\mathbb N}

\newcommand{\w}{\omega}
\newcommand{\s}{\sigma}

\renewcommand{\phi}{\varphi}
\renewcommand{\epsilon}{\varepsilon}

\renewcommand{\le}{\leqslant}
\renewcommand{\ge}{\geqslant}

\newcommand{\nle}{\nleqslant}

\renewcommand{\preceq}{\preccurlyeq}
\renewcommand{\succeq}{\succcurlyeq}
\renewcommand{\npreceq}{\npreccurlyeq}
\renewcommand{\nsucceq}{\nsucccurlyeq}
\newcommand{\tleq}{\trianglelefteq}
\newcommand{\ntleq}{\ntrianglelefteq}
\newcommand{\tle}{\triangleleft}

\newcommand{\treq}[2]{\trianglerighteq^{#1}_{#2}}

\renewcommand{\Diamond}{\diamondsuit}
\newcommand{\Club}{\clubsuit}

\newcommand{\bSigma}{\boldsymbol{\Sigma}}
\newcommand{\bDelta}{\boldsymbol{\Delta}}
\newcommand{\bGamma}{\boldsymbol{\Gamma}}
\newcommand{\bLambda}{\boldsymbol{\Lambda}}
\newcommand{\bTheta}{\boldsymbol{\Theta}}
\newcommand{\bUpsilon}{\boldsymbol{\Upsilon}}
\newcommand{\bPi}{\boldsymbol{\Pi}}

\newcommand{\RCA}{\mathsf{RCA}}       
\newcommand{\ATR}{\mathsf{ATR}}       
\newcommand{\Pind}{\Pi^1_1\text{-}\mathsf{IND}}    

\newcommand{\Baire}{\mathcal{N}}

\newcommand{\emptystring}{\seq{}}

\newcommand{\RefClub}{{\hyperref[TSP:Club]{($\Club$)}}}

\newcommand{\ClassName}[1]{\textup{#1}}
\newcommand{\bClassName}[1]{{\textbf{#1}}}

\newcommand{\dual}[1]{\check{#1}}

\newcommand{\cs}[2]{\cramped{#1^{#2}}}

\DeclareDocumentCommand{\SU}{ m m g}
{
\IfNoValueT{#3}{\ClassName{SU}_{#1}(#2)}
\IfNoValueF{#3}{\ClassName{SU}_{#1}^{#2}(#3)}
}

\DeclareDocumentCommand{\bSU}{ m m g}
{
\IfNoValueT{#3}{\bClassName{SU}_{#1}(#2)}
\IfNoValueF{#3}{\bClassName{SU}_{#1}^{#2}(#3)}
}

\newcommand{\KB}{\textup{\scriptsize{KB}}}

\begin{document}

\title[Effective Wadge determinacy]{An Effective classification of Borel Wadge classes}

\author[A.\:Day]{Adam Day}

\author[N.\:Greenberg]{Noam Greenberg}
\address{School of Mathematics and Statistics\\ Victoria University of Wellington\\ Wellington, New Zealand}
\email{greenberg@msor.vuw.ac.nz}

\author[M.\:Harrison-Trainor]{Matthew Harrison-Trainor}
\address{Department of Mathematics\\ University of Michigan}
\email{matthhar@umich.edu}

\author[D.\:Turetsky]{Dan Turetsky}
\address{School of Mathematics and Statistics\\ Victoria University of Wellington\\ Wellington, New Zealand}
\email{dan.turetsky@vuw.ac.nz}

\begin{abstract}
We give a new and effective classification of all Borel Wadge classes of subsets of Baire space~$\Baire$. This relies on the true stage machinery originally developed by Montalb\'an. We use this machinery to give a new proof of Louveau and Saint-Raymond's separation theorem for Borel Wadge classes. This gives a proof of Borel Wadge determinacy in the subsystem $\ATR_0+\Pind$ of second-order arithmetic.
\end{abstract}

\maketitle

\setcounter{tocdepth}{1}
\tableofcontents

\section{Introduction}

Let $\Baire$ denote Baire space, $\omega^\omega$.  A set $A\subseteq \Baire$ is \emph{Wadge reducible} to a set $B\subseteq \Baire$ if~$A$ is a continuous pre-image of~$B$; we write $A\le_W B$. The \emph{Wadge class} of~$B$ is the collection of all sets which are Wadge reducible to~$B$. What are all Wadge classes of Borel sets? This question was essentially answered by Wadge (\cite{Wadge:phd}), with a full description appearing in \cite{Louveau:83} and another one in \cite{LouveauSR:Strength}. This analysis was then extended by Duparc \cite{Duparc1} to spaces $\kappa^\w$, Selivanov \cite{Selivanov:2007,Selivanov:2017} to $k$-partitions rather than sets, and Kihara and Montalb\'an \cite{KiharaMontalban:BQO} to BQO-valued functions. 

Once a system of descriptions for Borel Wadge classes is suggested, how does one prove that indeed every such class is described? The standard route, as suggested by Wadge, relies on \emph{Borel Wadge determinacy}. For any two sets $A,B\subseteq \Baire$, players in a game $G_L(A,B)$ alternate choosing natural numbers, resulting in a real $x\in \Baire$ built by player~I and another $y\in \Baire$ built by player~II; player~II wins if and only if $x\in A \Iff y\in B$. A winning strategy for player~II gives $A\le_W B$; for player~I, it gives $B\le_W A^\complement$. Borel Wadge determinacy is the statement that every such game, for Borel sets~$A$ and~$B$, is determined. Borel Wadge determinacy, thus, implies Wadge's \emph{semi-linear ordering principle} for Borel Wadge degrees: for all Borel~$A$ and~$B$, either $A\le_W B$ or $B\le_W A^\complement$. 

Further, using Borel Wadge determinacy, Martin showed (see \cite{KechrisMartin:Pi13}) that the Wadge ordering of Borel sets is well-founded. Thus, theorems about Borel Wadge classes can be proved inductively. One such theorem is that every Borel Wadge class has a description. 

In \cite{LouveauSR:WH,LouveauSR:Strength}, Louveau and Saint-Raymond follow a different route. They give two systems of descriptions (one is from \cite{Louveau:83}, based on work of Wadge's). They then give a new proof of Borel Wadge determinacy, but restricted to described classes. Using this, and a detailed analysis from \cite{Louveau:83} of the ambiguous classes of described classes, they then show that all Borel Wadge degrees are described, and so their determinacy result gives full Borel Wadge determinacy. 

One motivation for taking this different route is that their argument gives a proof of Borel Wadge determinacy within second-order arithmetic. This was quite surprising, since the most straightforward proof relies on Borel determinacy, which is known to require strong axioms \cite{FriedmanH:BorelDeterminacy}; and since~$\bPi^1_1$ Wadge determinacy was known to be equivalent to full $\bPi^1_1$ determinacy (Harrington \cite{Harrington:AnalyticDeterminacy}). The argument also gives an extension of Louveau's separation theorem \cite{Louveau:80:Separation} to all non-self-dual Borel Wadge classes.

It is then natural to ask for the weakest subsystem of second-order arithmetic in which Borel Wadge determinacy, or the semi-linear ordering principle, are provable. Cord\'on-Franco, Lara-Mart\'in and Loureiro \cite{Loureiro:phd,LoureiroEtAl:1} established some lower and upper bounds for lower levels of the Borel hierarchy. 

In the current paper we give a new analysis of all Borel Wadge classes. The main feature of our descriptions of Wadge classes is their \emph{dynamic} nature. Rather than a ``static'' description of sets in these classes, as constructed from simpler sets by using a variety of Boolean operations on sets, we describe approximation procedures that determine whether a given real is in the set described, based on its ({finite}) initial segments. 

To do this, we rely on the \emph{true stage machinery} developed by Montalb\'an \cite{Montalban:TrueStages:paper} to organise iterated priority arguments in computable structure theory. Applications of this technique to descriptive set theory were discovered by Day, Downey and Westrick \cite{DayDowneyWestrick} and Day and Marks \cite{DayMarks}. Our previous paper \cite{BSL_paper} is an expository paper, in which we give a description of the true stage machinery (as re-developed in \cite{GreenbergTuretsky:Pi11}), and showed how to use it to prove effective versions of the Hausdorff-Kuratowski theorem, Wadge's theorem on the structure of~$\bDelta^0_\lambda$ for limit~$\lambda$, and Louveau's separation theorem~\cite{Louveau:80:Separation}. In the current paper, we use our descriptions of Borel Wadge classes to prove Louveau and Saint-Raymond's separation theorem mentioned above, extending our argument from \cite{BSL_paper}, where it was restricted to the classes~$\bSigma^0_\xi$. This argument is more direct compared to the original one from \cite{LouveauSR:WH,LouveauSR:Strength}, which relied on the ramification method of unravelling games. We also provide the full details of the development of the true stage machinery, which was delayed from \cite{BSL_paper}.

Our reliance on the true stage method means that our analysis of the ambiguous classes described classes is inherently effective: computability of class descriptions and of names for sets in these classes are a necessary feature in this analysis, rather than an afterthought. 

Finally, we will observe that our methods are sufficiently effective so that they can be proved within the system $\ATR_0$ of second-order arithmetic, which is, in a sense, the natural base system to use when analysing Borel sets. We can then complete the argument sketched in \cite{LouveauSR:Strength} to prove Borel Wadge determinacy within that system (with the added requirement of some induction).

\section{Preliminaries}
\label{sec:preliminaries}

\subsection{An informal description of class descriptions} 
\label{sub:an_informal_description_of_class_descriptions}

In some sense, our descriptions of Borel Wadge classes are a common generalisation of the two systems of descriptions presented in \cite{LouveauSR:Strength}. A \emph{class description}~$\Gamma$ will include a well-founded tree~$T_\Gamma$, whose leaves will be labelled by 0's and 1's. For such a class description~$\Gamma$, a \emph{$\Gamma$-name}~$N$ for a subset of~$\Baire$ will determine functions $f_s$ for \emph{non-leaf} $s\in T_\Gamma$. Such a function~$f_s$ will choose, for each real $x\in \Baire$,  a child~$t$ of~$s$ on $T_\Gamma$. The subset of~$\Baire$ named by such a name~$N$ is then defined as follows: for a real~$x$, we determine a path $s_0,s_1, \dots, s_k$ through the tree~$T_\Gamma$, starting at the root $s_0= \seq{}$, and recursively choosing the child $s_{i+1} = f_{s_i}(x)$. When we get to a leaf~$s_k$, we decide that~$x$ is in the set or not, depending on the label (0 or 1) of~$s_k$ given by~$\Gamma$. The class description~$\Gamma$ will provide information stating which functions~$f_s$ can be used by $\Gamma$-names. The class~$\bGamma$ will consist of all sets that are named by a $\Gamma$-name. 

More specifically, for a \emph{non-leaf} $s\in T_\Gamma$, the class description~$\Gamma$ will specify two ordinals~$\xi_s$ and~$\eta_s$. A function $f_s$ used at the node~$s$ by a $\Gamma$-name will be the limit of an approximation at level~$\xi_s$, with bound $\eta_s+1$, and with a ``default outcome'' being the leftmost child of~$s$ on~$T_\Gamma$. Set-theoretically, this means that for each non-default child~$t$ of~$s$, the collection of~$x$ mapped by~$f_s$ to~$t$ form a $D_{\eta_s}(\bSigma^0_{1+\xi_s})$ set. 

By induction on the rank of~$T_\Gamma$, we can view our class descriptions as being built recursively, starting with the very simplest class descriptions: those which consists only of the root, which name the classes $\{\emptyset\}$ and $\{\Baire\}$ (depending on the label 0 or 1 of the root). In the typical application, given class descriptions $\Gamma_0, \Gamma_1, \dots$, we construct a more complicated class description~$\Gamma$ by combining the trees $T_{\Gamma_i}$ to a single tree $T_\Gamma$ by adding a root below them all and specifying the ordinals~$\xi$ and~$\eta$ for the new root. A $\Gamma$-name~$N$ will consist of a list $N_0,N_1,\dots$ with $N_i$ a $\Gamma_i$-name, and a function $f =f_{\seq{}}$ which for each real~$x$ tells us which name~$N_i$ we should apply to~$x$, in order to determine whether~$x$ is in the set named by~$N$. That is, if $A_i$ is the set named by~$N_i$, then the set named by~$N$ is $\{x\in \Baire \,:\,  x\in A_{f(x)} \}$. 

There are a couple of comments we should make now. The first is that even though $\Gamma$-names can be described ``statically'', using $D_{\eta_s}(\bSigma^0_{1+\xi_s})$ sets rather than functions~$f_s$, the way we will reason about these classes will be \emph{effective}. Thus, it will be important to view the functions~$f_s$ as the results of approximation procedures using true ${\xi_s}$-initial segments of a real~$x$, as in our analysis of the classes $D_\eta(\bSigma^0_{1+\xi})$ in \cite{BSL_paper}. Combining these approximation procedures, a $\Gamma$-name provides a ``dynamic'' approximation procedure for membership in the set it names. 

Further, the effective nature of our analysis, and our reliance on the true stage machinery, mean that we need to say what it means for an oracle~$y$ to compute a description of a class, and we need to choose specific computable copies of computable ordinals. For this reason, by ``ordinals'' we will mean well-orderings of subsets of~$\Nat$, and we will specify what it means to compute such orderings. 

The second point is that the nature of our approximation procedures indicates that the objects being approximated are \emph{functions} $F\colon \Baire \to \{0,1\}$, the characteristic functions of Borel sets. Our class descriptions can therefore be generalised to discuss Wadge classes of functions with ranges other than~$\{0,1\}$, as was done by Selivanov and Kihara-Montalb\'an. Our classification of these classes, though, is restricted to $\{0,1\}$-valued functions. 

Throughout our classification, we will be decomposing functions into simpler pieces.  \Cref{prop:main_classification_lemma} is illustrative -- here we begin with an $F$ belonging to a described class $\Delta(\bGamma)$, and we find $Y \subseteq \Baire$ such that $F\rest{Y}$ and $F\rest{\Baire\setminus Y}$ are from the simpler classes $\Delta(\bGamma_0)$ and $\bGamma^*$, respectively (the particulars of what this means are not important now; the reader should just accept that these are simpler classes). In the past, authors usually extended partial functions such as $F\rest{Y}$ to total functions by using retraction maps. While these are not effective, there are other, effective methods of extending partial functions to total functions. Thus, \cref{prop:main_classification_lemma} could be stated as $F\rest{\Baire\setminus Y}$ being the restriction of some function from~$\bGamma^*$. This, however, masks the fact that we are really only interested in the partial functions $F\rest{Y}$ and $F\rest{\Baire\setminus Y}$. Our arguments indicate that the more natural approach is to say that each description describes a class of \emph{partial} functions from $\Baire$ to $\{0,1\}$.  Then we can simply write $F\rest{\Baire\setminus Y} \in \bGamma^*$.

In particular, at each level~$\xi$, we will need to consider the restrictions of functions to both $\bSigma^0_{1+\xi}$ and $\bPi^0_{1+\xi}$ subsets of their domains. To do this, we will work with functions defined on subsets of~$\Baire$ in the second level of the Hausdorff hierarchy starting with each~$\bSigma^0_{1+\xi}$: the intersections of~$\bSigma^0_{1+\xi}$ and~$\bPi^0_{1+\xi}$ sets. To be able to properly discuss (and compute) names for such functions, we will represent their domains using the true stage relations, via the notion of $(z,\alpha)$-forests.

In this section we introduce the tools required for our definitions. We discuss computable ordinals. We detail the properties of the true stage machinery that we will use; but we relegate the development and verification of these properties to an appendix-like \cref{sec:true_stage_machinery}. We then introduce $(z,\alpha)$-forests and approximations of functions on such forests, and state a corresponding ``limit lemma''. We define our class descriptions in \cref{sec:described_classes}.

\subsection{Computable ordinals} 
\label{sub:computable_ordinals}

Henceforth, by an \emph{ordinal} we mean a well-ordering of a subset of~$\Nat$. We will nonetheless use set-theoretic terminology, and for example, talk about limit ordinals or successor ordinals; we will write~0 to denote the empty ordinal. When we want to refer to the von-Neumann (set-theoretic) ordinal isomorphic to a well-ordering~$\alpha$, we write $\otp(\alpha)$. 

We say that an oracle~$z$ computes an ordinal~$\alpha$ if~$z$ computes the field of~$\alpha$, the ordering relation, the successor function on~$\alpha$, and the set of limit points of~$\alpha$. 

We also require~$z$ to convey the information of whether~$\alpha$ is 0, a successor, or a limit; if a successor, what its greatest element is; if nonzero, what its initial element is. This means that when we say that~$z$ uniformly computes a sequence $\seq{\alpha_n}$ of ordinals, that $\{n\in \Nat\,:\, \alpha_n\text{ is limit}  \}$ is $z$-computable, and similarly for the other properties. Note that this information allows us to tell whether~$\alpha$ is finite or not, and if so, what its size is. If~$\alpha$ is a limit ordinal, then it allows us to (uniformly) compute an increasing and cofinal~$\w$-sequence in~$\alpha$. 

For two ordinals~$\alpha$ and~$\beta$, we write $\alpha< \beta$ if~$\alpha$ is a proper initial segment of~$\beta$ (this is a much stricter requirement than $\otp(\alpha)< \otp(\beta)$). Continuing set-theoretic notation, when $\alpha<\beta$, we often use~$\alpha$ to denote~$n_\alpha$, the element of the field of~$\beta$ which bounds~$\alpha$. For example, If $f\colon \Nat \to \beta$ is a function, we will often write $f(k)=\alpha$ instead of $f(k) = n_\alpha$. 

If an oracle~$z$ computes~$\beta$ then it computes all $\alpha<\beta$, and uniformly so, in the sense that given~$n_\alpha$ in the field of~$\beta$, $z$ computes all the required information about~$\alpha$.

\subsection{True stages} 
\label{sub:prelime:true_stages}

For each computable ordinal~$\alpha$, the true stage machinery provides a binary relation $\preceq_\alpha$ on  $\w^{\le \w}$, the collection of all finite and infinite sequences of natural numbers (which we call \emph{strings} or \emph{sequences}). The relation $\s\preceq_\alpha \tau$ reads ``$\s$ appears $\alpha$-true to~$\tau$.'' Intuitively, each finite~$\s$ makes a guess about $\Sigma^0_{1+\alpha}$ facts about infinite strings extending~$\s$. For $x\in \Baire$, $\s\prec_\alpha x$ will mean that $\s$'s guesses are correct about~$x$. If~$\tau$ is finite, then $\s\preceq_\alpha \tau$ will mean that~$\tau$ agrees with all of~$\s$'s guesses. This will not be precisely true, but it serves as a useful intuition for this machinery. As mentioned, the details are given in \cref{sec:true_stage_machinery}. 

We list some properties of the true stage relations $\preceq_\alpha$ that we will use.\footnote{The numbering of the properties differs from that in \cite{BSL_paper}.}
Fix a computable ordinal~$\alpha$. 

\begin{TSPlist}
    \item \label{TSP:nested}
    The relations are nested: 
    If $\alpha\le \beta$ then $\s\preceq_\beta \tau$ implies $\s\preceq_\alpha \tau$. We have $\s\preceq_0 \tau\Iff \s\preceq\tau$. Hence, for all~$\alpha$, $\s\preceq_\alpha \tau$ implies $\s\preceq \tau$. 

    \item \label{TSP:trees}
    The relation $\preceq_\alpha$ is a partial ordering of $\w^{\le \w}$ which is a tree (for all $\tau\in \w^{\le \w}$, $\left\{ \s \,:\,  \s\preceq_\alpha \tau \right\}$ is linearly ordered). The root of the tree is $\emptystring$ (the empty string). 

     \item \label{TSP:unique_path}
    For every $x\in \Baire$, 
     \[
      \left\{ \s\in \w^{<\w} \,:\,  \s\prec_\alpha x  \right\}
     \]
     is the unique infinite path of the restriction of $\preceq_\alpha$ to $\left\{ \s\in \w^{<\w} \,:\,  \s\prec x \right\}$. 

     \item \label{TSP:computable}
     The restriction of the relation $\preceq_\alpha$ to pairs of finite strings is computable. 

     \item \label{TSP:limit}
    If $\lambda$ is a limit, then for all~$\s,\tau\in \w^{\le \w}$, 
    \[
        \s\preceq_\lambda \tau \,\,\Iff\,\, (\forall \alpha<\lambda)\, \s \preceq_\alpha \tau.
    \]
    Furthermore, there is a computable and increasing sequence $\seq{\lambda_k}$, cofinal in~$\lambda$, such that for all $\s\in \w^{<\w}$, if $k = |\s|_\lambda$ is the height of~$\s$ in the tree $(\w^{<\w},\preceq_\lambda)$, then $|\s|_{\lambda_k} = k$, and for all $\tau\in \w^{\le \w}$, 
    \[
        \s\preceq_\lambda \tau \,\Iff\,  \s\preceq_{\lambda_k} \tau. 
    \]
\end{TSPlist}

For \ref{TSP:limit}, note that in light of \ref{TSP:nested}, the condition $|\s|_\lambda = |\s|_{\lambda_k}$ says that for all $\rho \preceq \s$, $\rho \preceq_\lambda \s \,\Iff\, \rho\preceq_{\lambda_k} \s$. 

The next property deals with successor ordinals. For an ordinal~$\alpha$, we write $\alpha+1$ to denote any ordinal~$\beta$ such that $\alpha<\beta$ and $\otp(\beta)= \otp(\alpha)+1$. 

\begin{TSPlist}[resume]
    \item \label{TSP:successor}
    There is a computable function $p_{\alpha+1}\colon \w^{<\w}\to \Nat\cup \{-1\}$  such that for all $\s,\tau\in \w^{\le \w}$, $\s\preceq_{\alpha+1} \tau$ if and only if $\s\preceq_\alpha \tau$, and for all finite~$\rho$ with $\s\preceq_\alpha \rho\preceq_\alpha \tau$, $p_{\alpha+1}(\rho) \ge p_{\alpha+1}(\s)$.
\end{TSPlist}

We note that \ref{TSP:successor} implies the following:

\begin{TSPlist}
     \item[($\Club$)] \label{TSP:Club}
     For all $\s\preceq_\alpha \rho \preceq_\alpha \tau \in \w^{\le \omega}$, if $\s\preceq_{\alpha+1} \tau$ then $\s\preceq_{\alpha+1} \rho$. 
\end{TSPlist}

The connection with the (effective) Borel hierarchy is the following:

\begin{TSPlist}[start=7]
     \item \label{TSP:Sigma_alpha_sets}
        A set $A\subseteq \Baire$ is $\Sigma^0_{1+\alpha}$ if and only if there is a c.e.\ set $U\subseteq \w^{<\w}$ such that 
        \[
            A =  \left\{ x\in \Baire \,:\,  (\exists \s \prec_\alpha x)\,\,\s\in U \right\}.
        \]
\end{TSPlist}

\begin{remark} \label{rmk:regarding_uniform_computability}
    We will not repeatedly state this, but every computable procedure in this paper is uniform, in every conceivable way. This means that in each instance, there is a single computable operator which takes the input and produces the desired output. In particular, for example, if $\seq{\alpha_n}$ is a uniformly computable sequence of ordinals, then the relations $\preceq_{\alpha_n}$ are uniformly computable. Similarly, for each~$\alpha$ (uniformly in~$\alpha$) there are computable functions translating between c.e.\ indices of sets $U\subseteq \w^{<\w}$ and $\Sigma^0_{1+\alpha}$-indices of the related subsets of~$\Baire$ as in \ref{TSP:Sigma_alpha_sets}. 
\end{remark}

\subsubsection{Relativised true stages} 

The machinery above was stated for computable ordinals~$\alpha$, to keep notation simple. In fact, we can relativise this machinery to any oracle. For any ordinal~$\alpha$, and for any oracle~$z$ that computes~$\alpha$, we obtain a relation $\preceq^z_\alpha$, which satisfies all of the properties above, with the following changes: ``computable'' (in \ref{TSP:computable}, \ref{TSP:limit} and \ref{TSP:successor}) is replaced by ``$z$-computable''; and in \ref{TSP:Sigma_alpha_sets}, ``c.e.'' is replaced by ``$z$-c.e.'' and $\Sigma^0_{1+\alpha}$ is replaced by $\Sigma^0_{1+\alpha}(z)$.  Note that even when $\alpha$ is computable, the relation $\preceq_\alpha^z$ is not the same as $\preceq_\alpha$, though we can think of $\preceq_\alpha$ as $\preceq_\alpha^\emptyset$.

As expected, this relativisation process is uniform. Suppose that~$\alpha$ is a computable ordinal. Then the relations $\preceq^z_\alpha$ on pairs of finite strings are $z$-computable, uniformly in~$z$. In fact, they naturally turn out to be Lipschitz-computable (computable with constant use): there is a ternary computable relation $\s\preceq^{\rho}_\alpha\tau$, defined when $|\rho|\ge |\tau|$, such that for all~$z$, $\s\preceq^z_\alpha \tau \Iff  \s\preceq^{z\rest{|\tau|}}_\alpha \tau$.

\subsection{$\alpha$-forests} 
\label{sub:_alpha_forests}

Fix a computable ordinal~$\alpha$.

\begin{definition} \label{def:alpha_open_and_closed}
    Let $S\subseteq \w^{<\w}$. 
    \begin{sublemma}
        \item A set $R\subseteq S$ is \emph{$\alpha$-open in~$S$} if for all $\s\in R$ and $\tau\in S$, if $\s\preceq_\alpha \tau$ then $\tau \in R$.
        \item A set $R\subseteq S$ is \emph{$\alpha$-closed in~$S$} if for all $\s\in R$ and $\tau\in S$, if $\tau\preceq_\alpha \s$ then $\tau \in R$.
    \end{sublemma}
\end{definition}

\ref{TSP:Sigma_alpha_sets} can be modified as follows.

\begin{lemma} \label{lem:alpha_open_alpha_closed_sigma_and_pi}
    Let $A\subseteq \Baire$. 
    \begin{sublemma}
        \item $A$ is $\Sigma^0_{1+\alpha}$ if and only if there is a computable $W\subseteq \w^{<\w}$, $\alpha$-open in $\w^{<\w}$, such that 
        \[
                A =  \left\{ x\in \Baire \,:\,  (\exists \s\prec_\alpha x)\,\,\s \in W \right\} = 
                \left\{ x\in \Baire \,:\,  (\forall^\infty \s\prec_\alpha x)\,\,\s \in W \right\}.
            \]
        \item $A$ is $\Pi^0_{1+\alpha}$ if and only if there is a computable  $T\subseteq \w^{<\w}$, $\alpha$-closed in $\w^{<\w}$, such that 
        \[
                A =  
                \left\{ x\in \Baire \,:\,  (\forall \s\prec_\alpha x)\,\,\s \in T \right\}.
            \]
    \end{sublemma}
\end{lemma}

\begin{proof}
    (a): this is a standard trick. If $\seq{U_s}$ is a computable enumeration of the set~$U$ given by \ref{TSP:Sigma_alpha_sets}, then we let $\s\in W$ if there is some $\rho\preceq_\alpha \s$ in $U_{|\s|_\alpha}$, where recall that $|\s|_{\alpha}$ is the height of~$\s$ in the tree $(\w^{< \w},\preceq_\alpha)$. (b) is obtained from~(a) by taking the complement. 
\end{proof}

As mentioned above, we need to generalise these two notions. 

\begin{definition} \label{def:alpha-forest}
    An \emph{$\alpha$-forest} is a subset~$S$ of $\w^{<\w}$ which is $\preceq_\alpha$-convex: if $\s\preceq_\alpha \rho\preceq_\alpha \tau$ and $\s,\tau\in S$ then $\rho\in S$. 

    For an $\alpha$-forest~$S$ we let
    \[
        [S]_\alpha =   \left\{ x\in \Baire \,:\,  (\exists^\infty \s\prec_\alpha x)\,\,\s \in S \right\} = 
                \left\{ x\in \Baire \,:\,  (\forall^\infty \s\prec_\alpha x)\,\,\s \in S \right\}.
    \]
\end{definition}

We can visualise an $\alpha$-forest as the disjoint union of trees: the collection of $\preceq_\alpha$-minimal elements of~$S$ forms an antichain in $(\w^{<\w},\preceq_\alpha)$, and for each~$\s$ in that antichain, the restriction of $(S,\preceq_\alpha)$ to strings $\tau\succeq_\alpha \s$ is an $\alpha$-tree with root~$\s$. 

Every $\alpha$-open subset~$W$ of $\w^{<\w}$ is an $\alpha$-forest, as is every $\alpha$-closed set~$T$; and $[W]_\alpha$ and~$[T]_\alpha$ coincide with the sets described in \cref{lem:alpha_open_alpha_closed_sigma_and_pi}. Further:

\begin{lemma} \label{lem:open_and_closed_within_forest_is_a_forest}
    Let $S$ be an $\alpha$-forest. If $R\subseteq S$ is either $\alpha$-open in~$S$ or $\alpha$-closed in~$S$, then~$R$ is an $\alpha$-forest. 
\end{lemma}

For the following lemma, let $\Sigma^0_{\beta}\wedge \Pi^0_{\beta}$ denote the class of sets which are intersections $A\cap B$ where~$A$ is $\Sigma^0_\beta$ and $B$ is $\Pi^0_\beta$ (this is the class $D_2(\Sigma^0_\beta)$ in the effective Hausdorff hierarchy). 

\begin{lemma} \label{lem:alpha_forests_define_which_sets}
    A set $A\subseteq \Baire$ is $\Sigma^0_{1+\alpha}\wedge \Pi^0_{1+\alpha}$ if and only if there is a computable $\alpha$-forest~$S$ such that $A = [S]_\alpha$. 
\end{lemma}

\begin{proof}
    Suppose that $A$ is $\Sigma^0_{1+\alpha}\wedge \Pi^0_{1+\alpha}$. By \cref{lem:alpha_open_alpha_closed_sigma_and_pi}, let $U,T\subseteq \w^{<\w}$ be computable with~$U$ $\alpha$-open and~$T$ $\alpha$-closed such that $A = [U]_\alpha\cap [T]_\alpha$. Let $S = U\cap T$; then~$S$ is an $\alpha$-forest and $A = [S]_\alpha$ (note that the intersection of any number of $\alpha$-forests is an $\alpha$-forest, though the class $\Sigma^0_{1+\alpha}\wedge \Pi^0_{1+\alpha}$ is not closed under taking infinite intersections, even if effective). 

    In the other direction, suppose that~$S$ is a computable $\alpha$-forest. Let~$W$ be the upwards closure of~$S$ in $\preceq_\alpha$; let $V = W\setminus S$. Then both~$W$ and~$V$ are computable and $\alpha$-open, and $A = [U]_\alpha\setminus [V]_\alpha$, so is $\Sigma^0_{1+\alpha}\wedge \Pi^0_{1+\alpha}$.
\end{proof}

Suppose that~$X$ is $\Sigma^0_{\beta}\wedge \Pi^0_{\beta}$. A set $Y\subseteq X$ is \emph{$\Sigma^0_{\beta}$ within~$X$} if $Y = X\cap A$ for some $\Sigma^0_{\beta}$ set; similarly define $\Pi^0_\beta$ within~$X$. Both sets which are $\Sigma^0_{\beta}$ within~$X$ or $\Pi^0_\beta$ within~$X$ are $\Sigma^0_{\beta}\wedge \Pi^0_{\beta}$. A set $Y\subseteq X$ is $\Sigma^0_{\beta}\wedge \Pi^0_{\beta}$ within~$X$ if and only if it is $\Sigma^0_{\beta}\wedge \Pi^0_{\beta}$. 

\begin{lemma} \label{lem:relatively_effectively_alpha_closed_and_open_by_forests}
    Let $S$ be a computable $\alpha$-forest. 
    \begin{sublemma}
        \item A set $Y\subseteq [S]_\alpha$ is $\Sigma^0_{1+\alpha}$ within~$[S]_\alpha$ if and only if $Y=[W]_\alpha$ for some computable $W\subseteq S$ which is $\alpha$-open in~$S$. 
        \item A set $Y\subseteq [S]_\alpha$ is $\Pi^0_{1+\alpha}$ within~$[S]_\alpha$ if and only if $Y=[T]_\alpha$ for some computable $T\subseteq S$ which is $\alpha$-closed in~$S$.
        \item A set $Y\subseteq [S]_\alpha$ is $\Sigma^0_{1+\alpha}\wedge \Pi^0_{1+\alpha}$ if and only if $Y=[R]_\alpha$ for some computable $\alpha$-forest $R\subseteq S$.
    \end{sublemma}
\end{lemma}

\begin{remark} \label{rmk:alpha_and_beta_forests}
    \ref{TSP:nested} implies that if $\alpha<\beta$ and~$S$ is an $\alpha$-forest, then $S$ is also a $\beta$-forest, and $[S]_\alpha = [S]_\beta$. Thus, when~$\alpha$ is clear, we will just write~$[S]$. 
\end{remark}

\begin{remark} \label{rmk:relativised_forests}
    As above, all notions can be relativised to an oracle~$z$, giving the notion of a $(z,\alpha)$-forest. We write $[S]^z_\alpha$ for the set of ``infinite paths'' through~$S$. 
\end{remark}

\subsubsection{Orthogonal forests} 

\begin{definition} \label{def:alpha-orthogonal}
    Two sets $A,B\subseteq \w^{<\w}$ are \emph{$\alpha$-orthogonal} if every $\s\in A$ and $\tau\in B$ are $\preceq_\alpha$-incomparable. 
\end{definition}

If $A$ and~$B$ are $\alpha$-orthogonal $\alpha$-forests then $[A]_\alpha$ and $[B]_\alpha$ are disjoint, indeed they are contained in disjoint $\Sigma^0_{1+\alpha}$ sets (take the $\preceq_\alpha$-upward closures of~$A$ and~$B$). On the other hand it is possible that~$U$ and~$W$ are $\alpha$-open and not orthogonal, and so not disjoint, while $[U]_\alpha$ and~$[W]_\alpha$ are disjoint. The reason is that when $\alpha > 0$, there are $\s\in \w^{<\w}$ such that~$[\s]_\alpha$ is empty, that is, the tree $(\w^{<\w},\preceq_\alpha)$ is well-founded above~$\s$. The set of such~$\s$ is $\Pi^1_1$-complete, in particular, we cannot identify such~$\s$ computably and exclude them from $\alpha$-forests. However, if we know that $[U]_\alpha\cap [W]_\alpha = \emptyset$ but $\s\in U\cap W$, this gives us a proof that $[\s]_\alpha = \emptyset$, and so we can ignore such~$\s$. 

\begin{lemma} \label{lem:producing_pairwise_orthogonal}
    Let~$S$ be an $\alpha$-forest. Suppose that $Z_0,Z_1,\dots$ are pairwise disjoint sets which are uniformly $\Sigma^0_{1+\alpha}$ in~$[S]_\alpha$. Then there are pairwise orthogonal, uniformly computable $\alpha$-forests $W_0,W_1,\dots$, each $\alpha$-open in~$S$, such that $Z_i = [W_i]_\alpha$ for all~$i$, and such that $\bigcup_i W_i$ also is computable. 
\end{lemma}

\begin{proof}
    There are uniformly computable~$U_i$, $\alpha$-open in~$S$, with $Z_i = [U_i]_\alpha$. To make them pairwise orthogonal, we may first assume that $|\s|_\alpha > i$ for all $\s\in U_i$. Then define $W_i\subseteq U_i$ as follows. For each~$\s$ and~$i$, we determine if $\s\in W_i$ by induction on $|\s|_\alpha$; we declare that $\s\in W_i$ if $\s\in U_i$, and there is no $\rho\precneq_\alpha \s$ in~$W_j$ for any $j\ne i$ with $j<|\s|_\alpha$. 
\end{proof}

\subsection{Approximations and translations} 

\subsubsection{$\alpha$-approximations} 

Some of the following was discussed in \cite{BSL_paper}. We include it in this paper for completeness. The main difference is that we work with $\alpha$-forests rather than all of $\w^{<\w}$. 

\begin{definition} \label{def:alpha_approximation}
    Let $S$ be an $\alpha$-forest. An \emph{$\alpha$-approximation} of a function $F\colon [S]\to \Nat$
    is a function $f\colon S\to \Nat$ such that for all $x\in [S]$, for all but finitely many $\s\prec_\alpha x$, $f(\s) = F(x)$. 
\end{definition}

\begin{remark} \label{rmk:alpha_and_beta_approximations}
    Continuing \cref{rmk:alpha_and_beta_forests}, if $\alpha<\beta$, $S$ is an $\alpha$-forest and $f\colon S\to \Nat$ is an $\alpha$-approximation of some function on $[S]$, then~$f$ is also a $\beta$-approximation of the same function on~$[S]$. 
\end{remark}

Let~$X$ be $\Sigma^0_\beta \wedge \Pi^0_\beta$. A function $F\colon X\to \Nat$ is \emph{$\Sigma^0_\beta$-measurable} if the sets $F^{-1}\{n\}$ for $n\in \Nat$ are uniformly $\Sigma^0_\beta$ within~$X$. 

\begin{proposition} \label{prop:approximations_and_alpha_plus_one_computability}
    Let~$S$ be an $\alpha$-forest. A function $F\colon [S]\to \Nat$ is $\Sigma^0_{1+\alpha+1}$-measurable if and only if it has a computable $\alpha$-approximation $f\colon S\to \Nat$. 
\end{proposition}

We prove this proposition, and the rest of the results of this section, in \cref{sec:true_stage_machinery}. 

\begin{definition} \label{def:witness_for_convergence_of_approximation}
    Let $S$ be an $\alpha$-forest and let $f\colon S\to \Nat$ be an $\alpha$-approximation of a function~$F$. An \emph{ordinal witness for the convergence of~$f$} is a function $\s\mapsto \beta(\s)$ from~$S$ to some ordinal~$\gamma$ satisfying:
    \begin{orderedlist}
        \item For all $\s\preceq_\alpha \tau$ from~$S$, $\beta(\tau)\le \beta(\s)$; 
        \item If $\s\preceq_\alpha \tau$ are from~$S$ and $f(\s)\ne f(\tau)$ then $\beta(\tau)<\beta(\s)$. 
    \end{orderedlist}
\end{definition}

Thus, the well-foundedness of~$\gamma$ ``proves'' that for all $x\in [S]$, the sequence $\seq{f(\s)}_{\s\prec_\alpha x}$ stabilises to the value $F(x)$. 

\begin{lemma} \label{lem:existence_of_a_computable_ordinal_witness}
    For every computable $\alpha$-approximation $f\colon S\to \Nat$ there is a computable ordinal~$\gamma$ and a computable witness $\beta\colon S\to \gamma$ for the convergence of~$f$. 
\end{lemma}

The proof of \cref{lem:existence_of_a_computable_ordinal_witness} will use the following, which we will also need on its own. 

\begin{lemma} \label{lem:computable_rank_function}
    If~$T$ is a well-founded tree, then there is a $T$-computable ordinal~$\eta$ and a $T$-computable rank function $r\colon T\to \eta$. 
\end{lemma}

\subsubsection{Translating between oracles, ordinals, and spaces} 

We will use four ``translation propositions''. For the first, let~$\alpha$ be an ordinal and suppose that an oracle~$z$ computes~$\alpha$. Every $w\ge_\Tur z$ also computes~$\alpha$, and every $\Sigma^0_{1+\alpha}(z)$ set is also $\Sigma^0_{1+\alpha}(w)$, uniformly in the indices by a $w$-computable function (as usual for functions on indices, it can be made computable rather than just $w$-computable, but that is not important). This can be translated to true stage relations.

\begin{proposition} \label{prop:translating_true_stages_between_oracles}
    Suppose that $w\ge_\Tur z$ and that $\alpha$ is a $z$-computable ordinal. Then there is a $w$-computable function $h\colon \w^{<\w}\to \w^{<\w}$ such that:
\begin{orderedlist}
    \item if $\s\preceq^w_\alpha \tau$ then $h(\s) \preceq^z_\alpha h(\tau)$; 
    \item for all $x\in \Baire$, $\left\{ h(\s) \,:\,  \s\prec^w_\alpha x \right\}=\{ \rho \,:\,  \rho\prec^z_\alpha x \}$.
\end{orderedlist}
\end{proposition}

We can also pull back true stage relations by computable functions. 

\begin{proposition} \label{prop:translating_true_stages_across_computable_functions}
    Let~$\alpha$ be a computable ordinal; let $\Phi\colon \Baire\to \Baire$ be a computable function. 
    Then there is a computable function $h\colon \w^{<\w}\to \w^{<\w}$ such that:
\begin{orderedlist}
    \item if $\s\preceq_\alpha \tau$ then $h(\s) \preceq_\alpha h(\tau)$; 
    \item for all $x\in \Baire$, $\left\{ h(\s) \,:\,  \s\prec_\alpha x \right\}=\{ \rho \,:\,  \rho\prec_\alpha \Phi(x) \}$.
\end{orderedlist}
\end{proposition}

Computably isomorphic ordinals give rise to  equivalent true stage relations.

\begin{proposition} \label{prop:translating_true_stages_between_copies_of_the_ordinal}
    Suppose that $\alpha$ and~$\beta$ are two computably isomorphic ordinals. Then there is a computable function $h\colon \w^{<\w}\to \w^{<\w}$ such that:
\begin{orderedlist}
    \item if $\s\preceq_\beta \tau$ then $h(\s) \preceq_\alpha h(\tau)$; 
    \item for all $x\in \Baire$, $\left\{ h(\s) \,:\,  \s\prec_\beta x \right\}=\{ \rho \,:\,  \rho\prec_\alpha x \}$.
\end{orderedlist}
\end{proposition}

We can ``unpair'' true stage relations. Suppose that $(z,x)\mapsto \seq{z,x}$ is a computable bijection between~$\Baire^2$ and~$\Baire$ (a computable ``pairing function'' for Baire space). Let $\s\mapsto ((\s)_0,(\s)_1)$ be a corresponding computable ``unpairing'' of finite strings (so for all~$(x_0,x_1)\in \Baire^2$, $x_i = \bigcup_{\s\prec\seq{x_0,x_1}}(\s)_i$). Recall that for a computable ordinal~$\alpha$, the relation $\s\preceq^\rho_\alpha \tau$ is defined when $|\rho|\ge |\tau|$. 

\begin{proposition} \label{prop:translating_true_stage_relations_by_unpairing}
    For any computable ordinal~$\alpha$, there is a computable function $h\colon \w^{<\w}\to \w^{<\w}$ such that:
    \begin{orderedlist}
    \item For all~$\s$, $|h(\s)|\le |(\s)_0|$;
    \item if $\s\preceq_\alpha \tau$ then $h(\s)  \preceq_\alpha^{(\tau)_0} h(\tau) \preceq_0 (\tau)_1$; 
    \item for all $(z,x)\in \Baire^2$, $\left\{ h(\s) \,:\,  \s\prec_\alpha \seq{z,x} \right\}=\{ \rho \,:\,  \rho\prec_\alpha^z x \}$.
\end{orderedlist}
\end{proposition}

\Cref{prop:translating_true_stages_across_computable_functions,prop:translating_true_stage_relations_by_unpairing,prop:translating_true_stages_between_copies_of_the_ordinal} relativise to any oracle.

\section{Described classes} 
\label{sec:described_classes}

We introduce our class descriptions. We will start with very general descriptions. These are sufficiently general to include all classes described by Wadge \cite{Wadge:phd}, Louveau \cite{Louveau:83} and by Louveau and Saint-Raymond \cite{LouveauSR:Strength}. We will then consider a restricted collection of descriptions, the \emph{acceptable} ones. In \cref{sec:classification_of_acceptable_classes} we will be able to classify all acceptable classes in terms of an analysis of their ambiguous class.

\subsection{General class descriptions} 
\label{sub:general_class_descriptions}

As in \cite{Louveau:83} and \cite{LouveauSR:Strength}, we think of classes as defined recursively. The simplest classes contain the constant partial functions. Suppose that $\{\Upsilon_i\,:\, i\in I\}$ is a finite or countable collection of class descriptions. Among these, we choose a special ``initial'' class~$\Upsilon_{i^*}$. We then combine these classes to a more complicated class~$\Gamma$ by specifying two ordinals~$\xi$ and~$\eta$. The ordinal~$\xi$ (which we will also denote by $o(\Gamma)$) denotes the \emph{level} of~$\Gamma$: $\Gamma$-functions will be defined on $\bSigma^0_{1+\xi}\wedge \bPi^0_{1+\xi}$ sets. A $\Gamma$-function~$F$ on such a set $X = [S]_\xi$ will be defined by choosing a function $f: X \to I$ and a function~$F_i$ on~$X$ in the class~$\Upsilon_i$ for each $i \in I$.  We will then set $F(x) = F_{f(x)}(x)$.  Thus $f$'s role is to decide, for each $x \in X$, which $F_i$ to apply to $x$. The function~$f$ will be approximated on~$S$ in the sense of \cref{def:alpha_approximation}. The ordinal~$\eta+1$ bounds the ``mind-changes'' permitted along this approximation, in the sense of \cref{def:witness_for_convergence_of_approximation}. An approximation of the function~$f$ is supplied with a witness $\beta\colon S\to \eta+1$. We require that the chosen class~$\Upsilon_{i^*}$ is the initial value, meaning: if $\beta(\s)= \eta$ then $f(\s) = i^*$. This requirement to have an initial choice is what makes the class~$\Gamma$ non-self-dual.

\begin{definition} \label{def:description_of_a_class}
    A \emph{class description} consists of an oracle~$y$, and a labelled tree~$\Gamma$ satisfying:
    \begin{orderedlist}
        \item  The domain $T_\Gamma$ of~$\Gamma$ is a well-founded subtree of $\w^{<\w}$;
        \item If $s\in T_\Gamma$ is not a leaf, then $\Gamma(s)$ is a pair of ordinals $(\xi,\eta)$, with~$\eta$ nonzero, which we denote by $(\xi_s, \eta_s) = (\cs{\xi_s}{\Gamma},\cs{\eta_s}{\Gamma})$;
        \item If $s$ is a leaf of $T_\Gamma$ then $\Gamma(s)\in \{0,1\}$; 
        \item If $s,t\in T_\Gamma$ are not leaves and $s\preceq t$ then $\xi_s\le \xi_t$;
        \item $y$ computes~$\Gamma$.
    \end{orderedlist}
    We use~$\Gamma$ to denote the pair $(\Gamma,y)$. If~$T_\Gamma$ consists only of the root $\emptystring$, we declare $o(\Gamma) = \w_1$. Otherwise, we set $o(\Gamma)= \xi_{\emptystring}$. We write~$\cs{\eta}{\Gamma}$ for~$\cs{\eta_{\emptystring}}{\Gamma}$. We write~$\cs{y}{\Gamma}$ for~$y$. 
\end{definition}

Here computing~$\Gamma$ means computing~$T_\Gamma$, the set of leaves of~$T_\Gamma$, the function $s\mapsto \Gamma(s)$ for leaves~$s$ of~$T_\Gamma$, and the ordinals $\xi_s$ and~$\eta_s$, uniformly in non-leaf $s\in T_\Gamma$. There is a Turing-least oracle computing~$\Gamma$, but at times we will want to choose $\cs{y}{\Gamma}$ which is strictly Turing above that real. This will restrict the class of oracles that are allowed to compute a name of a $\Gamma$-function.

As our ordinals are well-orderings of~$\Nat$ (and thus countable), the definition $o(\Gamma) = \omega_1$ when the tree is trivial requires some discussion.  We consider~$\w_1$ as a symbol and declare that $\alpha <\w_1$ for every ordinal~$\alpha$.  Since $\omega_1$ is a limit, if, hypothetically, \ref{TSP:limit} were to extend to $\omega_1$, it would imply (via a regularity argument) that any $\omega_1$-forest is an $\alpha$-forest for some countable $\alpha$ (and the converse would hold by \ref{TSP:nested}).  So we will take this as our definition: for any oracle~$z$, $S$ is a $(z,\w_1)$-forest if it is a $(z,\alpha)$-forest for some $z$-computable ordinal~$\alpha$.

\begin{definition} \label{def:code_of_a_Gamma_function}
    Suppose that~$\Gamma$ is a class description. A \emph{$\Gamma$-name} consists of an oracle $z\ge_\Tur \cs{y}{\Gamma}$, a $(z,o(\Gamma))$-forest~$S$, and a $z$-computable labelled tree~$N$, whose domain $T_N$ is the collection of $s\in  T_\Gamma$ which are not leaves of~$T_\Gamma$, and such that for any $s\in T_N$, $N(s)$ is a pair $(f_s,\beta_s) = (\cs{f_s}{N},\cs{\beta_s}{N})$, satisfying:
    \begin{orderedlist}
        \item $f_s$ is a $(z,\xi_s)$-approximation on~$S$ (where $\xi_s= \cs{\xi_s}{\Gamma}$) of a function on~$[S]$, which we also denote by~$f_s$;
        \item $\beta_s\colon S\to \eta_s+1$ is a witness for the convergence of~$f_s$ (where $\eta_s = \cs{\eta_s}{\Gamma}$); 
        \item For all $\s\in S$, $f_s(\s)$ is a child of~$s$ on $T_\Gamma$;
        \item For all $\s\in S$, if $\beta_s(\s) = \eta_s$ then $f_s(\s)$ is the leftmost child of~$s$ on $T_\Gamma$. 
    \end{orderedlist}
    We use~$N$ to denote the name. We write $\cs{z}{N}$ for~$z$ and $S^N$ for~$S$; we write~$f^N$ for $f^N_{\emptystring}$ and $\beta^N$ for $\beta^N_{\emptystring}$. 
\end{definition}


As with class description, $N$ being $z$-computable requires that~$S$ be $z$-computable, and that $f_s\rest{S}$ and~$\beta_s$ are  $z$-computable functions, uniformly in $s\in T_N$. Note that since $z\ge_\Tur \cs{y}{\Gamma}$, the ordinals $\xi_s$ for $s\in T_N$ are uniformly $z$-computable, so \ref{TSP:computable} implies that the relations $\preceq^z_{\xi_s}$ are uniformly $z$-computable. The fact that $\xi_s\ge o(\Gamma)$ for all $s\in T_N$ implies that $S = S^N$ is a $(z,\xi_s)$-forest (\cref{rmk:alpha_and_beta_forests}), so the notion of a $(z,\xi_s)$-approximation on~$S$ makes sense. 

As we don't use the ordinals~$\eta_s$ for true stage relations, only as bounds, we do not require~$\Gamma$ to choose a particular version of $\eta_s+1$. Rather, we can think of a function into~$\eta_s+1$ as a function into~$\eta_s$ which also takes an extra special value ``$\infty$'', which we also denote by~$\eta_s$. 

\begin{definition} \label{def:function_defined_by_name}
    Suppose that~$\Gamma$ is a class description, $N$ is a $\Gamma$-name, and $w\in S^N\cup [S^N]$. We define a leaf $\ell(w) = \ell^N(w)$ of $T_\Gamma$ by recursion as follows: $\emptystring$ is a predecessor of $\ell(w)$; if $s\in T_N$ has been decided to be a predecessor of $\ell(w)$, then we declare that $f_s(w)$ is also a predecessor of $\ell(w)$.  (Recall that for $w \in S^N$, $f_s(w)$ refers to the approximation function, and for $w \in [S^N]$, $f_s(w)$ refers to the function being approximated.)

    We define $F^N(w) = \Gamma(\ell(w))$. (Recall that for a leaf~$s$, $\Gamma(s)\in \{0,1\}$). 
\end{definition}

When we refer to the function~$F^N$ we mean $F^N\rest{[S^N]}$ (that is, the function on infinite sequences, excluding the finite ones), unless we mention otherwise.

Note that if $T_\Gamma$ consists of only the root ($o(\Gamma) = \omega_1$), then $T_N$ is empty.  So a $\Gamma$-name is determined by only a $z$ and an $S$; in this case $F^N$ is the constant $\Gamma(\emptystring)$-value function on $[S]$.
\begin{definition} \label{def:Gamma_s}
    Let~$\Gamma$ be a class description. For any $z\ge_\Tur \cs{y}{\Gamma}$, we let
    \[
    \Gamma(z) = \left\{ F^N \,:\,  N\text{ is a $\Gamma$-name}\andd \cs{z}{N} = z \right\}. 
    \]
    We let 
    \[
    \bGamma = \bigcup \left\{ \Gamma(z) \,:\,  z \ge_\Tur \cs{y}{\Gamma}  \right\}. 
\]
\end{definition}

Informally, a \emph{$\Gamma(z)$-name} is a $\Gamma$-name~$N$ with $\cs{z}{N} = z$. \Cref{prop:translating_true_stages_between_oracles} implies:

\begin{lemma} \label{lem:name_translation}
    Let $\Gamma$ be a class description. If $w\ge_\Tur z\ge_\Tur \cs{y}{\Gamma}$ then $\Gamma(z)\subseteq \Gamma(w)$. Indeed, we can $w$-effectively translate a $\Gamma(z)$-name $N$ to a $\Gamma(w)$-name $N'$ such that $F^N = F^{N'}$, uniformly in~$w$ and~$z$.
\end{lemma}

\begin{remark} \label{rmk:class_depends_on_order_types}
    The class $\bGamma$ depends only on the order-types of the ordinals mentioned by~$\Gamma$, not their realisation as well-orderings of a subset of~$\Nat$. This follows from \cref{prop:translating_true_stages_between_copies_of_the_ordinal}
\end{remark}

\subsubsection{Sub-descriptions} 

Suppose that $P$ is a labelled tree and $s\in P$. The subtree $P_s$ issuing from~$s$ is defined to be the labelled tree defined on 
    $T_{P,s} = \left\{ t \,:\,  s\conc t\in T_P \right\}$
by setting $P_s(t) = P(s\conc t)$. 

\begin{lemma} \label{lem:sub-descriptions}
Suppose that $\Gamma$ is a class description and $s\in T_\Gamma$. Then:
\begin{sublemma}
    \item $\Gamma_s$ (equipped with $\cs{y}{\Gamma_s} = \cs{y}{\Gamma}$) is a class description; 
    \item if $N$ is a $\Gamma$-name, then $N_s$ is a $\Gamma_s$-name (we set $S^{N_s} = S^{N}$ and $\cs{z}{N_s} = \cs{z}{N}$). 
\end{sublemma}    
\end{lemma}

The classes $\Gamma_s$ are those that are used in the inductive construction of~$\Gamma$ from simpler classes. Note that consideration of sub-descriptions is why in \cref{def:description_of_a_class} we asked for $\xi_s\le \xi_t$ whenever $s\preceq t$, rather than just $\xi_{\emptystring} \le \xi_{t}$, which would have sufficed up until now. 

The subclasses $\Gamma_{(n)}$ (for~$n$ such that $(n)\in T_\Gamma$) are the classes that were used in the last step of the construction of~$\Gamma$; so $I = \{ n : (n) \in T_\Gamma\}$, and $i^* = \min I$, to use the notation of the discussion at the beginning of this section. If~$N$ is a $\Gamma$-name, then for all $x\in [S] = [S^N]$, $F^N(x) = F^{N_{(n)}}(x)$, where $(n) = f^N(x)$. That is, the function $f^N = f^N_{\emptystring}$ is the function that chooses, for each $x\in [S]$, which of the functions $F^{N_{(n)}}$ to apply to~$x$ in order to calculate $F^N(x)$. For finite $\s\prec_{o(\Gamma)} x$ in~$S$, $f^N(\s)$ is the ``stage $|\s|_{o(\Gamma)}$-guess'' of this choice~$f^N(x)$.

\subsubsection{Duals} 

\begin{definition} \label{def:dual_class}
     Let~$\Gamma$ be a class description. We let $\dual{\Gamma}$ be the result of switching all the labels of the leaves of~$\Gamma$ (i.e., 0s become 1s and 1s become 0s).
 \end{definition} 

 If $N$ is a~$\Gamma$-name, then in the technical sense, it is also a $\dual{\Gamma}$-name, but when we interpret it as such, we write $\dual{N}$. Then $\dom F^N = \dom F^{\dual{N}} = [S^N]$, and 
 \[
 F^{\dual{N}} = 1-F^N. 
 \]

\subsection{Some examples} 
\label{sub:some_examples}

We show how our class descriptions include a variety of classes from the literature. We use the following terminology. 

\begin{definition} \label{def:total_names}
    Let~$\Gamma$ be a class description. A $\Gamma$-name~$N$ is \emph{total} if $S^N = \w^{<\w}$. A function $F\in \bGamma$ is total if $F = F^N$ for some total $\Gamma$-name~$N$. 
\end{definition}

For the following examples, if~$\bLambda$ is a collection of subsets of~$\Baire$, we say that $\bGamma = \bLambda$ if the total functions in~$\bGamma$ are the characteristic functions of the sets in~$\bLambda$.

\medskip
\noindent{\emph{a. Constants}}: Let~$T_\Gamma$ consist only of a root. Let $i \in \{0,1\}$ be the label of the root. Then~$\bGamma$ is the collection of all constant functions with value~$i$, defined on any Borel set. Using the notation just introduced,  $\bGamma = \{\emptyset\}$ when $i=0$ and $\bGamma = \{\Baire\}$ when $i=1$. 

\medskip
\noindent{\emph{b. $\bSigma^0_\alpha$}}: let~$T_\Gamma$ consist of the root, and two children $(0)$ and~$(1)$. Let the labels of the leaves be~0 and~1, respectively. Set~$o(\Gamma)$ to be some ordinal~$\xi$, and let $\cs{\eta}{\Gamma} = 1$. Then \ref{TSP:Sigma_alpha_sets} says that~$\bGamma$ consists of all functions $1_A\colon X\to \{0,1\}$ where~$X$ is $\bSigma^0_{1+\xi}\wedge \bPi^0_{1+\xi}$, and $A\subseteq X$ is $\bSigma^0_{1+\xi}$ within~$X$. In terms of sets, $\bGamma = \bSigma^0_{1+\xi}$.


\medskip
\noindent{\emph{c. Hausdorff difference hierarchy}}: More generally, let~$T_\Gamma$ be as in the previous example, but allow~$\eta = \cs{\eta}{\Gamma}$ to be larger than~1. Then $\bGamma = D_\eta(\bSigma^0_{1+\xi})$ of the Hausdorff difference hierarchy. This is proved in \cite[Prop.3.8]{BSL_paper}. The main takeaway is the dynamic characterisation of a set~$C$ in the class $D_\eta(\bSigma^0_{1+\xi})$ as one for which membership $x\in C$ is the result of a guessing process which starts at level~$\eta$ with ``no'', and then allows mind-changes as the ordinal decreases. 

\medskip
\noindent{\emph{d. One-sided separated unions}}: Generalising the previous example, let~$\Upsilon$ be a class description. Let~$\Gamma$ be the class description consisting of a root, and two children~$(0)$ and~$(1)$, such that the immediate sub-descriptions are $\Gamma_{(0)} = \dual{\Upsilon}$ and $\Gamma_{(1)} = \Upsilon$. Choose $\xi = o(\Gamma) \le o(\Upsilon)$ and any $\eta = \cs{\eta}{\Gamma}$. Then 
    $\bGamma = \ClassName{Sep}(D_\eta(\bSigma^0_{1+\xi}), \bUpsilon)$
is the collection of all sets of the form $(C\cap A_0)\cup (C^\complement \cap A_1)$, where $C\in D_\eta(\bSigma^0_{1+\xi})$, $A_0\in {\bUpsilon}$ and $A_1\in \dual{\bUpsilon}$.

\medskip
\noindent{\emph{e. Two-sided separated unions}}: Let~$\Upsilon$ and~$\Lambda$ be two class descriptions. Let~$\Gamma$ be the class description consisting of a root, and three children~$(0)$, $(1)$, and~$(2)$, such that the immediate sub-descriptions are $\Gamma_{(0)} = \Lambda$, $\Gamma_{(1)} = {\Upsilon}$ and $\Gamma_{(2)} = \dual{\Upsilon}$. Choose $\xi = o(\Gamma)\le o(\Lambda), o(\Upsilon)$, and $\cs{\eta}{\Gamma}=1$. Then 
    $\bGamma = \ClassName{BiSep}(\bSigma^0_{1+\xi}, \bUpsilon,\bLambda)$
is the collection of all sets of the form $(C_1\cap A_1)\cup (C_2\cap A_2)\cup ((C_1\cup C_2)^\complement \cap B)$, where~$C_1$ and~$C_2$ are disjoint $\bSigma^0_{1+\xi}$ sets, $A_1\in \bUpsilon$, $A_2\in \dual{\bUpsilon}$, and $B\in \bLambda$. For a $\Gamma$-name~$N$, the corresponding sets are $C_i = \left\{ x \,:\,  f^N(x)=i \right\}$ for $i=1,2$, and the characteristic functions of the sets~$A_i$  are $F^{N_{(i)}}$. 

\medskip

Below we extend this example to separated unions. More general two-sided separated unions, of the form $\ClassName{BiSep}(D_\eta(\bSigma^0_{1+\xi}), \bUpsilon,\bLambda)$, as well as more complicated separated unions and separated differences, can also be described, but their descriptions are a little more complicated; a direct construction is given in \cite{sequel}.

\subsection{Pointclasses} 
\label{sub:pointclasses}

Each described class is a pointclass.

\begin{proposition} \label{prop:described_are_pointclass}
   Let~$\Gamma$ be a class description. Then $\bGamma$ is closed under taking continuous pre-images. 
\end{proposition}

\begin{proof}
    Follows from \cref{prop:translating_true_stages_across_computable_functions} (and \cref{prop:translating_true_stages_between_oracles}). Let~$N$ be a total $\Gamma$-name; let $\Phi\colon \Baire\to\Baire$  be continuous. By \cref{lem:name_translation}, we may assume that~$\Phi$ is $\cs{z}{N}$-computable. For each $s\in T_N$ let~$h_s$ be given by \cref{prop:translating_true_stages_across_computable_functions} for $\preceq^{\cs{z}{N}}_{\xi_s}$; define a $\Gamma$-name~$M$ by letting $f_s^M(\s) = f_s^N(h_s(\s))$ and $\beta_s^M(\s) = \beta_s^N(h_s(\s))$. Then $F^M = F^N\circ \Phi$. 
\end{proof}

Next, we build a universal set. The following lemma implies that we can effectively list all $(z,\xi)$-approximations with convergence witnesses of the appropriate form. The key is the existence of the default value at the beginning of the approximation. 

\begin{lemma} \label{lem:preparation_for_universal_set}
    Uniformly, given an oracle~$z$, $z$-computable ordinals~$\xi$ and~$\eta$, and an index for  partial $z$-computable functions $f\colon \w^{<\w}\to \Nat$ and $\beta\colon \w^{<\w}\to \eta+1$, and a ``default value'' $n^*\in \Nat$, we can compute an index for total $z$-computable functions $g\colon \w^{<\w}\to \Nat$ and $\gamma\colon \w^{<\w}\to \eta+1$ satisfying:
    \begin{orderedlist}
        \item $g$ is a $(z,\xi)$-approximation of a function $G\colon \Baire\to \Nat$; 
        \item $\gamma$ is a witness for the convergence of~$g$;
        \item For all~$\s$, if $\gamma(\s) = \eta$ then $g(\s) = n^*$; 
        \item If $f$ and~$\beta$ are total, $f$ is a $(z,\xi)$-approximation of a function $F\colon \Baire \to \Nat$, $\beta$ is a witness for the convergence of~$f$, and for all~$\s$, if $\beta(\s) = \eta$ then $f(\s) = n^*$, then $G=F$. 
    \end{orderedlist}
\end{lemma}

\begin{proof}
    This is fairly standard.  We proceed recursively.  Define $f(\emptystring) = n^*$ and $\gamma(\emptystring) = \eta$.
    
    Given $\s\in \w^{<\w}$ with $\s \neq \emptystring$, let $k = |\s|^z_\xi$, and let $\tau$ be $\s$'s $\preceq_\xi^z$-predecessor (i.e., $\tau \preceq_\xi^z \s$ and $|\tau|^z_\xi = k-1$.)  Let $\rho \preceq_\xi^z \s$ be longest with $f(\rho)$ and $\beta(\rho)$ both converging within $k$ steps, and $\beta(\rho) < \gamma(\tau)$.  If there is no such $\rho$, define $g(\s) = g(\tau)$ and $\gamma(\s) = \gamma(\tau)$.  Otherwise, define $g(\s) = f(\s)$ and $\gamma(\s) = \beta(\s)$.
\end{proof}

As a result, for a class description~$\Gamma$, we can list all total~$\Gamma$-names. For simplicity, suppose that~$\Gamma$ is computable. Then there are, for each oracle~$z$, a list $N_0^z,N_1^z,\dots$ of total $\Gamma(z)$-names such that $N_e^z$ is $z$-computable, uniformly in~$e$ and~$z$, and such that $\Gamma(z) = \left\{ F^{N_e^z} \,:\,  e<\w \right\}$. Using \cref{prop:translating_true_stages_between_oracles} for $w = e\conc z$ we get:

\begin{lemma} \label{lem:the_universal_Gamma_set}
    Let~$\Gamma$ be a computable class description. There is, for each oracle~$z$, a total $\Gamma(z)$-name $N^z$, uniformly $z$-computable, such that $\left\{ F^{N^z} \,:\,  z\in \Baire \right\}$ lists all the total functions in~$\bGamma$. 
\end{lemma}

The join of these is a universal set for~$\Gamma$. Fix a pairing function $\seq{z,x}$ as discussed before \cref{prop:translating_true_stage_relations_by_unpairing}. 

\begin{lemma} \label{lem:universal_Gamma_set}
    Let $\Gamma$ be a computable class description and let $\seq{N^z}$ be as in \cref{lem:the_universal_Gamma_set}. Then $F(\seq{z,x}) = F^{N^z}(x)$ is in $\Gamma$.
\end{lemma}

\begin{proof}
    For each non-leaf $s\in T_\Gamma$, let~$h_s$ be given by \cref{prop:translating_true_stage_relations_by_unpairing} for the ordinal~$\xi_s$. By slowing approximations (as in the proof of \cref{lem:preparation_for_universal_set}), we may assume that for all~$s$, $\s$ and~$z$, the values $f^z_s(\s) = f^{N^z}_s(\s)$ and $\beta^z_s(\s) = \beta^{N^z}_s(\s)$ can be determined by consulting only $z\rest{|\s|}$; we write $f^\rho_s(\s)$ and $\beta^\rho_s(\s)$ when $|\rho|\ge |\s|$. Now define a computable $\Gamma$-name~$M$ by letting $f^M_s(\s) = f^{(\s)_0}_s(h_s(\s))$ and $\beta^M_s(\s) = \beta^{(\s)_0}_s(h_s(\s))$; $F^M$ is as required. 
\end{proof}

Relativising to~$\cs{y}{\Gamma}$ we get:

\begin{proposition} \label{prop:universal_and_so_nonselfdual}
    For any class description~$\Gamma$, there is a function in $\Gamma(\cs{y}{\Gamma})$ which is universal for total~$\bGamma$ functions. Hence, $\bGamma$ is a non-self-dual Wadge class. 
\end{proposition}

\begin{proof}
    The relativisation of \Cref{lem:the_universal_Gamma_set,lem:universal_Gamma_set} is that there is a total $\Gamma(\smallseq{\cs{y}{\Gamma}, z})$-name $N^{\smallseq{\cs{y}{\Gamma}, z}}$, uniformly $\smallseq{\cs{y}{\Gamma}, z}$-computable, such that $\{F^{N^{\smallseq{\cs{y}{\Gamma}, z}}} : z \in \Baire\}$ lists all the total functions in $\bGamma$, and further that $F(\seq{z,x}) = F^{N^{\smallseq{\cs{y}{\Gamma}, z}}}(x)$ is in $\bGamma$.  This is the universal function.
    
    Now suppose towards a contradiction that $1-F \in \bGamma$.  By \Cref{prop:described_are_pointclass}, $x \mapsto 1-F(\seq{x,x}) \in \bGamma$, so fix $z$ with $F^{N^{\smallseq{\cs{y}{\Gamma}, z}}}(x) = 1-F(\seq{x,x})$.  Then
    \[
    F(\seq{z,z}) = F^{N^{\smallseq{\cs{y}{\Gamma}, z}}}(z) = 1-F(\seq{z,z}),
    \]
    a contradiction.
\end{proof}

\subsection{Some closure properties} 
\label{sub:some_closure_properties}

Let~$\Gamma$ be a class description, and let $z\ge_\Tur \cs{y}{\Gamma}$. 

\begin{proposition} \label{prop:closure:subdomains}
    If $F$ is in $\Gamma(z)$ and $X\subseteq \dom F$ is $\Sigma^0_{1+o(\Gamma)}(z)\wedge \Pi^0_{1+o(\Gamma)}(x)$, then $F\rest{X}\in \Gamma(z)$. 
\end{proposition}

\begin{proof}
    Let~$N$ be a $\Gamma(z)$-name of~$F$. By \cref{lem:relatively_effectively_alpha_closed_and_open_by_forests}(c), there is a $(z,o(\Gamma))$-forest $T\subseteq S^N$ with $[T]=X$. Define the required $\Gamma(z)$-name~$M$ by restricting each approximation $(f_s^N,\beta_s^N)$ to $\s\in T$. 
\end{proof}

\begin{proposition} \label{prop:closure:totalisation}
    Every $F\in \Gamma(z)$ can be extended to a total function in~$\Gamma(z)$.
\end{proposition}

\begin{proof}
    Let $N$ be a $\Gamma(z)$-name; let $S = S^N$. For $s\in T_\Gamma$ a non-leaf, extend $(f_s,\beta_s) = (f_s^N,\beta_s^N)$ to functions $(f_s,\beta_s)$ on all of $\w^{<\w}$ as follows: for $\s\notin S$,
    \begin{itemize}
        \item If $\s$ has no $\prec^{z}_{\xi_s}$-predecessor in~$S$, let $f_s(\s)$ be the leftmost child of~$s$ on~$T_\Gamma$, and let $\beta_s(\s)= \eta_s$. 
        \item If $\s$ has a $\prec^{z}_{\xi_s}$-predecessor in~$S$, let~$\tau$ be the longest such; set $f_s(\s) = f_s(\tau)$ and $\beta_s(\s) = \beta_s(\tau)$. \qedhere
    \end{itemize}
\end{proof}

\begin{proposition} \label{prop:closure:mergeing_paritioned_functions}
    Let $X$ be $\Sigma^0_{1+o(\Gamma)}(z)\wedge \Pi^0_{1+o(\Gamma)}(x)$. Suppose that $(X_n)$ is a partition of~$X$ into sets which are $\Sigma^0_{1+o(\Gamma)}(z)$ within~$X$, uniformly in~$n$. Let $F\colon X\to \{0,1\}$ and suppose that for all~$n$, $F\rest{X_n}\in \Gamma(z)$, uniformly in~$n$. Then $F\in \Gamma(z)$. 
\end{proposition}

\begin{proof}
    Let $S$ be a $z$-computable $(z,o(\Gamma))$-forest such that $[S]^z_\alpha = X$. By \cref{lem:producing_pairwise_orthogonal}, let $(S_n)$ be a uniformly $z$-computable sequence of sets, $o(\Gamma)$-open in~$S$, pairwise $(z,o(\Gamma))$-orthogonal, such that $X_n = [S_n]^z_\alpha$, and such that $\bigcup_n S_n$ is also $z$-computable. 

    Let $N_n$ be uniformly $z$-computable $\Gamma(z)$-names such that $F\rest{X_n} = F^{N_n}$. Define a $\Gamma(z)$-name~$N$ by taking the ``disjoint union'' of the $N_n$'s: set $S^N = S$. For non-leaf $s\in T_\Gamma$ we define $(f_s,\beta_s) = (f^N_s,\beta^N_s)$ as follows: for $\s\in S$,
    \begin{itemize}
        \item If there is a $\rho \preceq_{o(\Gamma)}^z \s$ and an $n$ with $\s \in S_n \cap S^{N_n}$, fix the longest such $\rho$ and define $f_s(\s) = f^{N_n}_s(\rho)$ and $\beta_s(\s) = \beta^{N_n}_s(\rho)$.
        \item If there is no such $\rho$, define $f_s(\s)$ to be the leftmost child of~$s$ in~$T_\Gamma$, and $\beta_s(\s) = \eta_s$.\qedhere
    \end{itemize}
\end{proof}

We state a corollary, that will be used in conjunction with \cref{prop:approximations_and_alpha_plus_one_computability}. It implies that $\Gamma(z)$ is closed under taking approximations at levels lower than $o(\Gamma)$. 

\begin{corollary} \label{cor:closure:lower-level_approximations}
    Let $\xi<o(\Gamma)$. Suppose that~$X$ is $\Sigma^0_{1+\xi}(z)\wedge \Pi^0_{1+\xi}(z)$, and that $g\colon X\to \Nat$ is $\Delta^0_{1+\xi+1}(z)$-measurable. Suppose that $F_n\colon X\to \{0,1\}$ are uniformly in $\Gamma(z)$. Define $G\colon X\to \{0,1\}$ by $G(x) = F_{g(x)} (x)$. Then $G\in \Gamma(z)$. 
\end{corollary}

\begin{proof}
    For $n\in \Nat$ let $X_n = g^{-1}\{n\}$. Since $\xi+1 \le o(\Gamma)$, the sets~$X_n$ are uniformly $\Sigma^0_{1+o(\Gamma)}$ within~$X$. By \cref{prop:closure:subdomains}, the functions $G\rest{X_n} = F_n\rest{X_n}$ are uniformly in $\Gamma(z)$. By \cref{prop:closure:mergeing_paritioned_functions}, $G$ is in~$\Gamma(z)$. 
\end{proof}

The propositions above were stated for functions on reals. As usual, their proofs reveal uniformity in terms of manipulating names, and this uniformity will be used. For example, in \cref{prop:closure:mergeing_paritioned_functions}, observe that from the sequence $\seq{N_n}$ of names for $F\rest{X_n}$, we can $z$-computably construct the name~$N$ for~$F$.

\subsection{Effective containment} 
\label{sub:effective_containment}

\begin{definition} \label{def:equivalence_of_names}
    Let~$\Gamma$ and~$\Lambda$ be two class descriptions; let~$N$ be a $\Gamma$-name and let~$M$ be a $\Lambda$-name. We say that~$N$ and~$M$ are \emph{equivalent} (and write $N\equiv M$) if $\cs{z}{N} = \cs{z}{M}$, $S^N = S^M$, and $F^N= F^M$. 
\end{definition}

In that definition, we implicitly assume that $S = S^N = S^M$ is a $(z,\xi)$-forest for some $\xi \le o(\Gamma), o(\Lambda)$.

\begin{definition} \label{def:effective_containment}
    Let $\Lambda$ and~$\Gamma$ be class descriptions. 
    \begin{sublemma}
     \item   We write
        \[
            \Lambda \subseteq \Gamma 
        \]  
        if $\cs{y}{\Lambda}\ge_\Tur \cs{y}{\Gamma}$, and $\bLambda\subseteq \bGamma$, uniformly above~$y^\Lambda$: there is a $\cs{y}{\Lambda}$-computable function which, given a total $\Lambda$-name~$N$, outputs a total $\Gamma$-name~$M$ equivalent to~$N$.
    
    \item We write 
     \[
     \Lambda \equiv \Gamma
     \]
     if $\Lambda \subseteq \Gamma$ and $\Gamma\subseteq \Lambda$. 

    \item  We write
        \[
            \Lambda < \Gamma 
        \]
        if $\Lambda \subseteq \Gamma$ and $\check\Lambda \subseteq \Gamma$. 
    \end{sublemma}
\end{definition}

Note that $\Lambda < \Gamma$ iff $\check \Lambda < \Gamma$ iff $\Lambda < \check \Gamma$. 

\begin{lemma} \label{lem:<_is_transitive}
    The relations $\subseteq$ and $<$ on class descriptions are transitive. The relation~$\equiv$ is an equivalence relation. 
\end{lemma}

The requirement in \cref{def:effective_containment} that the names be total comes from the fact that if $o(\Lambda)\ne o(\Gamma)$ then the partial $\Lambda$-functions and the $\Gamma$-functions are not defined on the same domains. However, this is not a serious restriction. If $\Lambda \subseteq \Gamma$ and $\xi \le o(\Lambda),o(\Gamma)$ then every $F\in \bLambda$ whose domain is $\bSigma^0_{1+\xi}\wedge \bPi^0_{1+\xi}$ is also in~$\bGamma$, effectively; this follows from  \cref{prop:closure:subdomains,prop:closure:totalisation}. 

In practice, we will always have $o(\Gamma)$ and~$o(\Lambda)$ comparable (but both $o(\Lambda)>o(\Gamma)$ and $o(\Gamma)< o(\Lambda)$ will occur in nature). So we can take $\xi = \min \{o(\Lambda),o(\Gamma)\}$. However, incorporating this requirement into the definition would make the relation~$\subseteq$ not transitive.

\begin{definition} \label{def:oracle_constant_sequence}
    A sequence $\bar \Theta = \Theta_0,\Theta_1,\dots$ of class descriptions is \emph{uniform} if $\cs{y}{\Theta_0}, \cs{y}{\Theta_1},\dots$ is constant~$y$, and the class descriptions are uniformly $y$-computable. We write $\cs{y}{\bar\Theta}$ for~$y$. 
\end{definition}

Note that if~$\Gamma$ is a class description and for all~$n$, $(n)\in T_\Gamma$, then the sequence $\Gamma_{(0)},\Gamma_{(1)},\dots$ is uniform.

\begin{definition} \label{def:effective_containment:sequences}
    Let $\Gamma$ be a class description and let $\bar \Theta = \Theta_0,\Theta_1,\dots$ be a uniform sequence of class descriptions. We write
    \[
        \bar \Theta \subseteq \Gamma
    \]
    if $\Theta_n\subseteq \Gamma$ for all~$n$, uniformly in~$n$. Similarly, we write $\bar \Theta < \Gamma$ if $\Theta_n< \Gamma$, uniformly. 
\end{definition}

\subsection{Notation for constructed classes} 
\label{sub:notation_for_constructed_classes}

\begin{notation} \label{not:SU_defined}
    If $\bar\Theta$ is a uniform sequence of class descriptions and~$\xi$ is an ordinal, then we write $\xi \le o(\bar \Theta)$ if for all~$n$, $\xi \le o(\Theta_n)$. 
\end{notation}

Note that we are not requiring that $o(\Theta_n)$ is the same for all~$n$. Also note that if $\xi \le o(\bar\Theta)$ then as $o(\Theta_n)$ is $\cs{y}{\bar\Theta}$-computable, so is~$\xi$.

\begin{definition} \label{def:SU_operator}
    Suppose that:
    \begin{orderedlist}
        \item $\bar \Theta = \Theta_0,\Theta_1,\dots$ is a uniform sequence of class descriptions; 
        \item $\xi$ is an ordinal and $\xi \le o(\bar\Theta)$; and
        \item $\eta$ is a nonzero $\cs{y}{\bar\Theta}$-computable ordinal. 
    \end{orderedlist}
    Then we let
    \[
        \SU{\xi}{\eta}{\Theta_0,\Theta_1,\dots}
    \]
    be the class description~$\Gamma$ determined by setting:
    \begin{itemize}
        \item $\cs{y}{\Gamma} = \cs{y}{\bar\Theta}$;
        \item $o(\Gamma) = \xi$; 
        \item $\cs{\eta}{\Gamma} = \eta$; and
        \item for all~$n$, $\Gamma_{(n)}= \Theta_n$. 
    \end{itemize}
    If $\eta = 1$ then we write $\SU{\xi}{\Theta_0,\Theta_1,\dots}$. 
\end{definition}

The notation is derived from the notion of \emph{separated unions}. If $\bTheta_0 = \{\emptyset\}$ and $\eta = 1$, then the sets in $\bGamma$ are the ones of the form $\bigcup_{n\ge 1} (C_n\cap A_n)$, where $A_n\in \bTheta_n$, and $C_1,C_2,\dots$ are pairwise disjoint $\bSigma^0_{1+\xi}$ sets. In the more general case, we add $B\setminus \bigcup_n C_n$, where $B\in \bTheta_0$. When $\eta>1$, under some restrictions, $\bGamma$ is the class of separated unions based on $D_\eta(\bSigma^0_{1+\xi})$ sets; see \cite{sequel}.  We observe that if~$\Gamma$ is a class description and for all~$n$, $(n)\in T_\Gamma$, then $\Gamma = \SU{\xi}{\eta}{\Gamma_{(0)}, \Gamma_{(1)},\dots}$, where $\xi = o(\Gamma)$ and $\eta = \cs{\eta}{\Gamma}$. 

For the following lemma and below, we write $\lnot\SU{\xi}{\eta}{\bar\Theta}$ for the dual of the class.

\begin{lemma} \label{lem:SU:very_basic_stuff}
Suppose that $\bar\Theta$ is a uniform sequence, $\xi \le o(\bar\Theta)$, and $\eta>0$ is $\cs{y}{\bar\Theta}$-computable. 
   \begin{sublemma}
    \item \label{item:SUsimple:dual}
    $\lnot \SU{\xi}{\eta}{\bar \Theta} \equiv \SU{\xi}{\eta}{\check\Theta_0,\check\Theta_1,\dots}$.

    \item \label{item:SUsimple:containment}
     $\bar \Theta \subseteq \SU{\xi}{\eta}{\bar \Theta}$. 
                
    \item \label{item:SUsimple:sequence_containment} 
    If $\Lambda_n\subseteq \Theta_n$, uniformly, and $\xi \le o(\bar\Lambda)$, then $\SU{\xi}{\eta}{\bar \Lambda}\subseteq \SU{\xi}{\eta}{\bar \Theta}$. 
    \end{sublemma}     
\end{lemma}

\begin{lemma} \label{lem:SU_remains_bounded}
    Suppose that $o(\Gamma)>\xi$, $\bar \Theta<\Gamma$, $\xi \le o(\bar\Theta)$, and $\eta$ is $\cs{y}{\bar\Theta}$-computable. Then $\SU{\xi}{\eta}{\bar \Theta}< \Gamma$. 
\end{lemma}

\begin{proof}
    Let~$N$ be a total $\SU{\xi}{\eta}{\bar \Theta}$-name; let $z = \cs{z}{N}$. Let $F = F^N$ and $F_n = F^{N_{(n)}}$. So $F_n\in \Theta_n(z)$, uniformly, and so, as $\bar\Theta<\Gamma$, $F_n\in \Gamma(z)$ and ${F}_n\in \check\Gamma(z)$, uniformly. By \cref{prop:approximations_and_alpha_plus_one_computability,cor:closure:lower-level_approximations}, as $\xi < o(\Gamma)$, $F\in \Gamma(z)$ and $F\in \check\Gamma(z)$. This is uniform in everything.
\end{proof}

\subsection{The case $\eta = 1$} 
\label{sub:the_case_eta_1_}

Let~$\Gamma$ be a class description and let~$N$ be a $\Gamma$-name. For $s\in T_\Gamma$ let~$X_s^N$ be the collection of $x\in [S^N]$ such that $s\preceq \ell(x)$, where we use the notation from \cref{def:function_defined_by_name}. That is, $X_s^N$ consists of those~$x$ for which the procedure computing $F^N(x)$ ``passes through'' the node~$s$.

In general,~$X_s^N$ may be complicated. Suppose, however, that for all non-leaf~$s\in T_\Gamma$, $\eta_s = 1$. In this case, if~$t$ is a child of~$s$ on~$T_\Gamma$ then~$X_t^N$ is $\Sigma^0_{1+\xi_s}(z)\wedge \Pi^0_{1+\xi_s}(z)$ (where $z = \cs{z}{N}$). In fact, $X_t^N$ is either $\Sigma^0_{1+\xi_s}(z)$ or $\Pi^0_{1+\xi_s}(z)$ within $X_s^N$; the latter holds for the default (leftmost) child.

\begin{definition} \label{def:sub-forests_when_eta_is_1}
    Let~$\Gamma$ be a class description satisfying $\eta_s = 1$ for all non-leaf $s\in T_\Gamma$; let $N$ be a $\Gamma$-name. By induction on~$|s|$, for $s\in T_\Gamma$ we define $S_s^N\subseteq S^N$ as follows:
    \begin{itemize}
        \item $S_{\emptystring}^N = S^N$; 
        \item If~$t$ is a child of~$s$ on $T_\Gamma$, then $S_t^N = \left\{ \s\in S_s^N \,:\,  f^N_s(\s)=t \right\}$. 
    \end{itemize}
\end{definition}

If~$t$ is the default child of~$s$ then $S_t^N$ is $(z,\xi_s)$-closed in~$S_s^N$; otherwise, it is $(z,\xi_s)$-open in~$S_s^N$. Since $\xi_s\le\xi_t$, by induction on~$|s|$ we see that for non-leaf~$s$, $S_s^N$ is a $(z,\xi_s)$-forest; and $X_s^N = [S_s^N]^z_{\xi_s}$. 

Thus, in the case that each $\eta_s$ is~1, for each~$s$, we do not need to define the approximation $f^N_s$ on all of~$S^N$, but rather, only on~$S_s^N$. In the general case, since~$X_s^N$ may fail to be $\Sigma^0_{1+\xi_s}(z)\wedge \Pi^0_{1+\xi_s}(z)$, we need to define our approximations on all of~$S^N$ at each node~$s$.


\subsection{Montone sequences} 
\label{sub:montone_sequences}

\begin{definition} \label{def:monotone}
    A uniform sequence $\Theta_0,\Theta_1,\dots$ is \emph{monotone} if $\Theta_n\subseteq \dual{\Theta}_{n+1}$, uniformly in~$n$.   
\end{definition}

There are two main cases: either $\Theta_{n+1} = \dual{\Theta}_n$ for all~$n$, or $\Theta_{n}< \Theta_{n+1}$ for all~$n$. It will be a consequence of our analysis that by passing to infinite subsequences, these are the only two possibilities.

\begin{proposition} \label{prop:main_SU_prop}
    Suppose that $\bar \Theta = \seq{\Theta_0,\Theta_1,\dots}$ is monotone, $\xi \le o(\bar\Theta)$, and $\eta$ is $\cs{y}{\bar\Theta}$-computable. 
    \begin{sublemma}
        \item \label{item:SU:extra_containment}
        $\bar\Theta < \SU{\xi}{\eta}{\bar \Theta}$.

        \item \label{item:SU:behead}
        $\SU{\xi}{\eta}{\Theta_0,\Theta_1,\dots} \subseteq \lnot \SU{\xi}{\eta}{\Theta_1,\Theta_2,\dots}$. 

        \item \label{item:SU:even_worse_behead}
        If $\Theta_0<\Theta_1$ then $\SU{\xi}{\eta}{\Theta_0,\Theta_1,\dots} < \SU{\xi}{\eta}{\Theta_1,\Theta_2,\dots}$.

        \item \label{item:SU:increase_eta}
        If $\eta' < \eta $ then $\SU{\xi}{\eta'}{\bar \Theta} < \SU{\xi}{\eta}{\bar\Theta}$. 

        \item \label{item:SU:increase_xi}
        If $\xi' < \xi $ then $\SU{\xi'}{\bar \Theta} < \SU{\xi}{\bar\Theta}$. 
    \end{sublemma}
\end{proposition}

\begin{proof}
    \ref{item:SU:extra_containment} follows from \cref{lem:SU:very_basic_stuff}\ref{item:SUsimple:containment} and $\bar\Theta$ being monotone. 
    
    \ref{item:SU:behead} follows from \cref{lem:SU:very_basic_stuff}\ref{item:SUsimple:dual},\ref{item:SUsimple:sequence_containment}. 

    For \ref{item:SU:even_worse_behead}, in light of \ref{item:SU:behead}, it suffices to show that $\SU{\xi}{\eta}{\Theta_0,\Theta_1,\dots} \subseteq \SU{\xi}{\eta}{\Theta_1,\Theta_2,\dots}$. Let~$N$ be a $\SU{\xi}{\eta}{\Theta_0,\Theta_1,\dots}$-name. Define a $\SU{\xi}{\eta}{\Theta_1,\Theta_2,\dots}$-name~$M$ equivalent to~$N$ by letting $M_{(0)}$ be a $\Theta_1$-name equivalent to~$N_{(0)}$, and for $n\ge 2$, let $M_{(n)}$ be a $\Theta_{n+1}$-name equivalent to $N_{(n-1)}$. It doesn't matter how we define $M_{(1)}$. We define $\beta^M = \beta^N$ and $f^M(\s)=0$ if $f^N(\s)=0$, otherwise $f^M(\s) = f^N(\s)+1$.

    For \ref{item:SU:increase_eta}, $\SU{\xi}{\eta'}{\bar \Theta} \subseteq \SU{\xi}{\eta}{\bar\Theta}$ is clear, as every $\SU{\xi}{\eta'}{\bar \Theta}$-name is also a $\SU{\xi}{\eta}{\bar \Theta}$-name, and has the same interpretation. To see that $\SU{\xi}{\eta'}{\bar \Theta} \subseteq \lnot\SU{\xi}{\eta}{\bar\Theta}$, let~$N$ be a $\SU{\xi}{\eta'}{\bar \Theta}$-name. We define a $\SU{\xi}{\eta}{\Theta_0,\Theta_1,\dots}$ name~$M$ equivalent to~$N$ by letting $M_{(n+1)}\equiv \dual{N}_{(n)}$ and setting $\beta^M = \beta^N$ and $f^M = f^N+1$. The point is that $\eta>\eta'$ so even if $\beta^N(\s)=\eta'$, in~$M$, we are allowed to take a non-default outcome. 

    \ref{item:SU:increase_xi} follows from \ref{item:SU:extra_containment} and \cref{lem:SU_remains_bounded}. 
\end{proof}
Notice that for any class description $\Gamma$ and any $\xi < o(\Gamma)$, $\SU{\xi}{\eta}{\Gamma, \Gamma, \dots} \equiv \Gamma$ (one direction relies on \Cref{cor:closure:lower-level_approximations}).  So the ordinal $o(\Gamma)$ may vary between descriptions~$\Gamma$ of the same class.  However, we can identify a largest ordinal among all such descriptions. 

\begin{lemma}\label{lem:max_ordinal_for_class}
    If $\bar\Theta$ is monotone, $\xi \le o(\bar\Theta)$, and $\Gamma$ is a description with $\bGamma = \bSU{\xi}{\eta}{\bar\Theta}$, then $\otp(o(\Gamma)) \le \otp(\xi)$.
\end{lemma}
Recall that when we write equality of boldface classes, we mean that they contain the same total functions.

\begin{proof}
    For a contradiction, suppose $\otp(o(\Gamma)) > \otp(\xi)$, and fix $F \in \bSU{\xi}{\eta}{\bar\Theta}$ universal, by \Cref{prop:universal_and_so_nonselfdual}.  It suffices to show that $1-F \in \bGamma$. 
    
    Fix $z = y^\Gamma \oplus y^{\bar{\Theta}}$ and $N$ an $\SU{\xi}{\eta}{\bar\Theta}$-name of $F$.  Then $X_n = \{x \in \Baire : f_{\emptystring}^N(x) = n\}$ is $\Delta^0_{1+\xi+1}(z)$.  By \Cref{prop:main_SU_prop}\ref{item:SU:extra_containment} and monotonicity, $1 - F^{N_{(n)}} \in \bSU{\xi}{\eta}{\bar\Theta} = \bGamma$, and so by \Cref{cor:closure:lower-level_approximations}, $1-F \in \bGamma$.
\end{proof}

\subsection{Acceptable descriptions} 
\label{sub:acceptable_descriptions}

\begin{definition} \label{def:acceptable} \
    \begin{sublemma}
        \item A class description~$\Gamma$ is \emph{acceptable} if for all non-leaf $s\in T_\Gamma$, 
        \begin{orderedlist}
              \item $\cs{\eta_s}{\Gamma} = 1$; and
              \item for all~$n$, $s\conc n\in T_\Gamma$, and the sequence $\Gamma_{s\conc 0}, \Gamma_{s\conc 1},\dots$ is monotone, uniformly in~$s$. 
          \end{orderedlist}  
          \item A uniform sequence $\Theta_0,\Theta_1,\dots$ of class descriptions is acceptable if it is monotone, and the class descriptions $\Theta_n$ are uniformly acceptable. 
    \end{sublemma}
\end{definition}

Uniform acceptability implies that for all~$n$ and non-leaf $s\in T_{\Theta_n}$, the sequence $(\Theta_n)_{s\conc 0}, (\Theta_n)_{s\conc 1},\dots$ is monotone, uniformly in~$s$ and~$n$. 

Note that some of the class descriptions given in \Cref{sub:some_examples} are not acceptable. Nevertheless, we can give acceptable descriptions for these classes.

\medskip
\noindent{\emph{b${}'$. $\bSigma^0_\alpha$}}: let~$T_\Gamma$ consist of the root, and children $(0)$, $(1)$, $(2)$, and so on. Let the labels of the leaves alternate between $0$ and $1$, with even leaves labeled $0$ and odd leaves labeled $1$. Set~$o(\Gamma)$ to be some ordinal~$\xi$, and let $\cs{\eta}{\Gamma} = 1$. Then $\bGamma = \bSigma^0_{1+\xi}$.


\medskip
\noindent{\emph{c${}'$. Hausdorff difference hierarchy}}: The class $D_\eta(\bSigma^0_{1+\xi})$ has an acceptable description defined inductively on $\eta$ as follows. First, suppose that $\eta$ is a successor ordinal, $\eta = \gamma+1$. Set~$o(\Gamma)$ to be~$\xi$, and let $\cs{\eta}{\Gamma} = 1$. Let $\Gamma_{(0)}$ be a leaf labeled 0, and  let $\Gamma_{(i)}$ be an acceptable description for either $D_\gamma(\bSigma^0_{1+\xi})$ (if $i$ is even) or its dual (if $i$ is odd). If $\eta$ is a limit ordinal, let $\langle \eta_k \rangle$ be a computable increasing and cofinal sequence in $\eta$. We still set~$o(\Gamma)$ to be~$\xi$, and let $\cs{\eta}{\Gamma} = 1$. Let $\Gamma_{(0)}$ be a leaf labeled 0, and  let $\Gamma_{(i)}$ be an acceptable description for $D_{\eta_i}(\bSigma^0_{1+\xi})$.

\medskip

\begin{lemma} \label{lem:acceptability_of_SU_and_subclasses} \
    \begin{sublemma}
        \item If~$\Gamma$ is acceptable, then the sequence $\Gamma_{(0)},\Gamma_{(1)},\dots$ is acceptable,
        \item If $\bar\Theta$ is acceptable and $\xi\le o(\bar\Theta)$ then $\SU{\xi}{\bar\Theta}$  is acceptable. 
    \end{sublemma}
\end{lemma}

\begin{proposition} \label{prop:replacing_eta_by_1}
    Suppose that $\bar\Theta$ is acceptable; $\xi \le o(\bar\Theta)$, and $\eta$ is $\cs{y}{\bar\Theta}$-computable. Then there is some acceptable $\Gamma \equiv \SU{\xi}{\eta}{\bar\Theta}$ with $o(\Gamma) = \xi$. 
\end{proposition}

\begin{proof}
    This is done by induction on~$\eta$. Since everything has to be uniformly computable, this is, in fact, $\cs{y}{\bar\Theta}$-effective transfinite recursion on~$\eta$. 

    \medskip

    Suppose that $\eta>1$ is a successor ordinal. By induction, for each~$n$, let $\Lambda_n\equiv \SU{\xi}{\eta-1}{\Theta_n,\Theta_{n+1},\dots}$ be acceptable. Then $\Theta_0,\Lambda_0,\Lambda_1,\Lambda_2,\dots$ is monotone: $\Theta_0 \subseteq \dual{\Lambda}_1$ follows from \cref{prop:main_SU_prop}\ref{item:SU:extra_containment}; $\Lambda_n\subseteq \dual{\Lambda}_{n+1}$ follows from \cref{prop:main_SU_prop}\ref{item:SU:behead}. By \cref{lem:acceptability_of_SU_and_subclasses}, $\Gamma = \SU{\xi}{\Theta_0,\Lambda_0,\Lambda_1,\Lambda_2,\dots}$ is acceptable.
    
    \smallskip
    
     We check that $\Gamma \equiv \SU{\xi}{\eta}{\bar\Theta}$. In one direction, suppose that~$N$ is a $\SU{\xi}{\eta}{\bar\Theta}$-name; we build a $\Gamma$-name~$M$ equivalent to~$N$. Let $z = \cs{z}{N}$ and $S = S^N$. To define~$M$, we use the notation of \cref{def:sub-forests_when_eta_is_1}. We let $S^M_{(0)} = \left\{ \s\in S \,:\,  \beta^N(\s)=\eta \right\}$; we let $M_{(0)}$ be the restriction of $N_{(0)}$ to $S^M_{(0)}$ (\cref{prop:closure:subdomains}). 

     For simplicity, we may assume that for all $\s\in S$, if $\beta^N(\s)<\eta$ then $\beta^N(\tau) = \eta-1$ for some $\tau\preceq^z_\xi \s$ in~$S$. For all~$n$ we let~$W_n$ be the $(z,\xi)$-open subset of~$S$ generated by the strings~$\s$ with $\beta^N(\s) = \eta-1$ and $f^N(\s)=n$. We will define a $\SU{\xi}{\eta-1}{\Theta_n,\Theta_{n+1},\dots}$-name $G_n$ with $S^{G_n}= W_n$; then, we let $S^M_{(n+1)} = W_n$ and let $M_{(n+1)}$ be a $\Lambda_n$-name equivalent to~$G_n$. To define~$G_n$, we let $\beta^{G_n} = \beta^N\rest{W_n}$. We let $(G_n)_{(0)} = N_{(n)}$. Since~$\bar\Theta$ is monotone, we can find a computable function $g\colon \Nat\to \Nat$ and for each~$m$, let $(G_n)_{(m)}$ be a $\Theta_{n+m}$-name, so that:
     \begin{itemize}
         \item $g(n)=0$; and
         \item for all~$k$, $(G_n)_{g(k)}$ is equivalent to $N_{(k)}$. 
     \end{itemize}
    For $\s\in W_n$ we let $f^{G_n}(\s) = g(f^N(\s))$.

    \smallskip
     
    In the other direction, suppose that~$N$ is a~$\Gamma$-name. We define an $\SU{\xi}{\eta}{\bar\Theta}$-name~$M$ equivalent to~$N$. Again let $z = \cs{z}{N}$ and $S = S^N$. For each~$n$, let $K_n$ be an $\SU{\xi}{\eta-1}{\Theta_n,\Theta_{n+1},\dots}$-name equivalent to $N_{(n+1)}$ (so $S^{K_n} = S^N_{(n+1)}$). Now, as $\bar\Theta$ is monotone, find a computable function~$g$ and for each~$m$, let $M_{(m)}$ be a $\Theta_m$-name such that:
    \begin{itemize}
         \item $M_{(0)} = N_{(0)}$; and
         \item for all~$n$ and~$k$, $(K_n)_{(k)}$ is equivalent to $M_{(g(n,k))}$. 
     \end{itemize}
     For $\s\in S^N_{(0)}$ let $f^M(\s)=0$ and $\beta^N(\s)=\eta$. For each $n\ge 1$, for each $\s\in S^N_{(n)}$ let $\beta^M(\s) = \beta^{K_n}(\s)$ and $f^M(\s) = g(n,f^{K_n}(\s))$. 

     \medskip
     
    Suppose that~$\eta$ is a limit ordinal. Fix a computable increasing and cofinal sequence $\seq{\eta_n}$ in~$\eta$. By induction, find acceptable $\Lambda_n \equiv \SU{\xi}{\eta_n}{\bar\Theta}$. The sequence $\Theta_0,\Lambda_1,\Lambda_2,\Lambda_3,\dots$ is monotone: $\Theta_0\subseteq \dual{\Lambda}_1$ follows from \cref{prop:main_SU_prop}\ref{item:SU:extra_containment}; $\Lambda_n\subseteq \dual{\Lambda}_{n+1}$ follows from \cref{prop:main_SU_prop}\ref{item:SU:increase_eta}. Hence $\Gamma = \SU{\xi}{\Theta_0,\Lambda_1,\Lambda_2,\dots}$ is acceptable. 

    \smallskip
    
    The argument that $\Gamma \equiv \SU{\xi}{\eta}{\bar\Theta}$ is a little simpler than the successor case. To translate a $\SU{\xi}{\eta}{\bar\Theta}$-name~$N$ to a $\Gamma$-name~$M$, we of course keep the default outcome $M_{(0)} = N_{(0)}$. Suppose that~$\s\in S$ is minimal with $\beta^N(\s)<\eta$. Then $\beta^N(\s)<\eta_n$ for some~$n\ge 1$. Take least such, and set $\s\in S^M_{(n)}$ (putting all of its successors in $(S,\preceq^z_\xi)$ into $S^M_{(n)}$ as well). We then let $M_{(n)}$ be a $\Lambda_n$-name equivalent to a $\SU{\xi}{\eta_n}{\bar\Theta}$-name~$G_n$ defined by $\beta^{G_n} = \beta^N \rest{S^M_{(n)}}$, $f^{G_n} = \beta^N\rest{S^M_{(n)}}$, and $(G_n)_{(k)} = N_{(k)}$. The point is that we don't have to worry about the default outcome of~$G_n$, since we already have $\beta^N(\s)<\eta_n$. 

    The other direction is exactly as in the successor case. 
\end{proof}

By \Cref{lem:max_ordinal_for_class}, it follows that amongst the class descriptions of a given class~$\bLambda$, the acceptable descriptions have the largest von-Neumann ordinal (and we will eventually show that every Borel Wadge class has an acceptable description).  We can take this to be the ordinal of the class. For equivalent definitions of the ordinal level of a Borel Wadge class see \cite{LSR:reduction}. 

In light of \Cref{prop:replacing_eta_by_1}, one might ask why we allow $\eta > 1$ in our general class descriptions in the first place.  The answer is that classes $\SU{\xi}{\eta}{\bar \Theta}$ with $\eta>1$ show up naturally in our classification of simpler class descriptions, in particular, in the proof of \cref{prop:main:cofinality_omega1}; so we have designed our descriptions to easily accommodate them. Another reason is that natural descriptions of more complicated classes, such as $\ClassName{BiSep}(D_\eta(\bSigma^0_{1+\xi}), \bGamma)$, and separated differences require $\eta>1$. Descriptions with $\eta>1$ are also used in \cite{sequel} to characterise the effective reduction property.

\section{Classification of acceptable classes} 
\label{sec:classification_of_acceptable_classes}

We now give our classification of acceptable classes, by analysing their ambiguous classes. This resembles the work in \cite{Louveau:83}, but is much simplified; it also relies on our effective methods. 

\begin{definition} \label{def:Delta}
    For a class description~$\Gamma$ and $z\ge_\Tur \cs{y}{\Gamma}$, we let 
    \[
        \Delta(\Gamma(z)) = \Gamma(z)\cap \dual{\Gamma}(z) 
    \]
    and
    \[
    \Delta(\bGamma) = \bGamma \cap \dual{\bGamma}.
    \]
\end{definition}

Note that the first includes partial functions, while for the second we will usually only consider  total functions.

\begin{definition} \label{def:type} 
Let $\Gamma$ be an acceptable class description.
\begin{enumerate}
    \item $\Gamma$ has \emph{zero type} if $o(\Gamma)=\w_1$, i.e., if $T_\Gamma = \{\emptystring\}$. 
    
    \item $\Gamma$ has \emph{countable type} if there is an acceptable sequence $\bar\Theta$ such that:
    \begin{itemize}
        \item $\bar\Theta < \Gamma$;
        \item $\cs{y}{\bar\Theta} = \cs{y}{\Gamma}$;
        \item $o(\Theta_n)\ge o(\Gamma)$ for all~$n$; and
        \item For any $z\ge_\Tur \cs{y}{\Gamma}$, for any $F\in \Delta(\Gamma(z))$, there is a partition of $X= \dom F$ into sets $Y_n$ that are $\Sigma^0_{1+o(\Gamma)}(z)$ within~$X$ (uniformly), such that $F\rest{Y_n}\in \Theta_n(z)$, uniformly.
      \end{itemize}  
    \item $\Gamma$ has \emph{uncountable type} if for every $z\ge_\Tur \cs{y}{\Gamma}$, for every $F\in \Delta(\Gamma(z))$ whose domain is $\Pi^0_{1+o(\Gamma)}(z)$, there is some acceptable~$\Lambda$ such that:
    \begin{itemize}
        \item $\Lambda < \Gamma$; 
        \item $\cs{y}{\Lambda}=z$; 
        \item $o(\Lambda)\ge o(\Gamma)$; and
        \item $F\in \Lambda(z)$. 
    \end{itemize}
\end{enumerate}
\end{definition}

\begin{remark} \label{rmk:type:other_directions}
    Suppose that~$\Gamma$ has countable type, witnessed by $\bar\Theta$. Let $z\ge_\Tur y^\Gamma$, $X$ be $\Sigma^0_{1+o(\Gamma)}(z)\wedge \Pi^0_{1+o(\Gamma)}(z)$, $\seq{Y_n}$ a partition of~$X$ into sets that are uniformly $\Sigma^0_{1+o(\Gamma)}(z)$ within~$X$, and $F\colon X\to \{0,1\}$ be such that $F\rest{Y_n}\in \Theta_n(z)$, uniformly. Then by \cref{prop:closure:mergeing_paritioned_functions},  $F\in \Delta(\Gamma(z))$. 

    Hence, for both the countable and uncountable type, the condition described is actually a characterisation of $\Delta(\Gamma(z))$. 
\end{remark}

We say that an acceptable class description is \emph{classified} if it has one of the three types.   
The main theorem of this section is:

\begin{theorem} \label{thm:classification:main}
    Every acceptable class description is classified. 
\end{theorem}

Let us provide some examples.  Fix $\bar{\Theta}$ the sequence $0,1,0,1,\dots$.  For any $\xi$, let $\Gamma = \SU{\xi}{\bar{\Theta}}$, so that $\bGamma =  \ClassName{BiSep}(\bSigma^0_{1+\xi},1,0) = \bSigma^0_{1+\xi}$. Then $\Gamma$ has countable type, where the witnessing sequence is this same $\bar{\Theta}$. This is immediate: let $F\in \Delta(\Gamma(z))$; for $i<2$, let $X_i = F^{-1}[i]$. Then $(X_0,X_1)$ partition $\dom F$ into sets that are relatively $\Delta^0_{1+\xi}(z)$, and~$F$ is constant on both these sets. 

For a slightly more complicated example, consider $\bLambda = D_2(\bSigma^0_{1+\xi})$, which is described by an acceptable~$\Lambda$ constructed in example $c'$ after \cref{def:acceptable}: if $\Gamma$ is the standard acceptable description for $\bSigma^0_\xi$ then $\Gamma = \SU{\xi}{0,\Gamma,\dual{\Gamma},\Gamma,\dual{\Gamma},\dots}$. In this case we let $\Theta_n = \Gamma$ and $\Theta_{n+1} = \dual{\Gamma}$. To see that this satisfies the definition, let $F \in \Delta(\Lambda(z))$.  We can partition $X = \dom F$ into four sets, each $\Sigma^0_{1+\xi}(z)$ within $X$: those $z \in X$ for which the $\Lambda$-name for $F$ moves off the default outcome before the $\dual{\Lambda}$-name does, and it moves to an even outcome~$(n)$; those $z \in X$ for which the $\Lambda$-name moves off the default outcome first and moves to and odd outcome; and two similar sets for when the $\dual{\Lambda}$-name is first to move. On each of these four sets, $F$ is either in $\Gamma(z)$ or in $\dual{\Gamma}(z)$, depending on the parity of the outcome chosen. 

On the other hand, let $\Gamma = \SU{1}{0,1,0,\dots}$ be the standard acceptable name for~$\bSigma^0_2$; and let $\Upsilon = \SU{0}{\Gamma,\dual{\Gamma},\Gamma,\dual{\Gamma},\dots}$, so that $\bUpsilon = \ClassName{BiSep}(\bSigma^0_1, \bPi^0_2,\bSigma^0_2)$.  It will be a consequence of \cref{thm:class_type_by_leftmost_path_labels} that~$\Upsilon$ has uncountable type.  Indeed, by examining the proofs of \cref{prop:the_inductive_proof_of_classification} and \cref{prop:main:cofinality_omega1}, we can see that for any $F \in \Delta(\bUpsilon)$, $F \in \ClassName{BiSep}(\bSigma^0_1,\bPi^0_2,D_\eta(\bSigma^0_1))$ for some sufficiently large $\eta < \omega_1$, and all of these classes are $<\Upsilon$; and so are as required for uncountable type.

\subsection{How classification helps} 
\label{sub:how_this_helps}

The purpose of the notion of type is to tell us about the ambiguous classes, and the cofinality of the class in the Wadge degrees. If $\Gamma$ is an acceptable description, there will be three cases:
\begin{orderedlist}
    \item $o(\Gamma) = \omega_1$;

    \item $\Delta(\bGamma)$ is a principal pointclass, the least upper bound in the Wadge degrees of a countable sequence of degrees; 

    \item $\Delta(\bGamma)$ is the (necessarily uncountable) union of those $\bLambda \subset \Delta(\bGamma)$.
\end{orderedlist}

These three cases will {\em not} exactly correspond to the three sorts of types we have just defined.  It will turn out that case (b) happens only when $\Gamma$ has countable type and $o(\Gamma) = 0$; if $\Gamma$ has uncountable type or $0 < o(\Gamma) < \omega_1$, then case (c) will hold. We summarise this in the following two propositions.

\begin{proposition} \label{prop:countable_type:level_0}
    Let~$\Gamma$ be an acceptable class description. If~$\Gamma$ has countable type, witnessed by $\bar \Theta$, and $o(\Gamma) = 0$, then $\Delta(\bGamma)$ is principal, and is the least principal  pointclass containing~$\bigcup_n \bTheta_n$. 
\end{proposition}

\begin{proposition} \label{prop:classification:uncountable_cofinality}
    Let~$\Gamma$ be an acceptable class description. If~$\Gamma$ has uncountable type, or $0 < o(\Gamma) < \omega_1$, then $\Delta(\bGamma) = \bigcup \left\{ \bLambda \,:\, \Lambda\text{ is acceptable and } \Lambda < \Gamma  \right\}$, and is non-principal. 
\end{proposition}

One can think of \cref{prop:countable_type:level_0} as covering two cases depending on the character of the sequence $\bar\Theta$. As mentioned above, the separation theorem will imply the semi-linear-ordering property for described classes: for any two class descriptions~$\Gamma$ and~$\Lambda$, either $\bLambda\subseteq \bGamma$ or $\bGamma\subseteq \dual{\bLambda}$. This will imply that if~$\bar\Theta$ is monotone then it has a subsequence which is $<$-increasing, or it eventually alternates between a class and its dual. In the first case, $\Delta(\bGamma)$ above will have cofinality~$\w$ in the Wadge degrees; in the latter, it will be a successor of the two dual classes which form a tail of $\bar\Theta$.

\smallskip

The discrepancy between the type of a class and its cofinality in the Wadge degrees would lead one to ask why we have chosen to define countable and uncountable type as we have. What is countable about a class of uncountable cofinality? The heart of the matter is in our general plan for proving \cref{thm:classification:main} (and in particular, the proof of \cref{prop:main:cofinality_omega1} below). Indeed, given an acceptable description, we can easily tell its type.

\begin{theorem} \label{thm:class_type_by_leftmost_path_labels}
    Let~$\Gamma$ be an acceptable class description, and suppose that $o(\Gamma)<\w_1$. Let $s^*$ be the leftmost leaf of~$T_\Gamma$ (necessarily of the form $0^n$ for some~$n$). 
    \begin{sublemma}
        \item If for all $s\prec s^*$, $\cs{\xi_s}{\Gamma} = o(\Gamma)$, then~$\Gamma$ has countable type. 
        \item Otherwise, $\Gamma$ has uncountable type. 
    \end{sublemma}
\end{theorem}

That is,~$\Gamma$ has countable type if climbing from the root to the leftmost leaf~$s^*$, the ordinal labels are constant; it has uncountable type if the ordinal labels increase at some point. The class~$\bGamma$ has countable cofinality (among non-self dual pairs) if the ordinal labels along~$s^*$ are all~0. 

The real answer to the question ``what is countable about a countable type with uncountable cofinality'' relies on an alternative understanding of the ordinal level of a class. Let~$\Gamma$ be a class description. Ignoring computability issues, we can define a class description $\Lambda = \Gamma^{-o(\Gamma)}$ by subtracting $o(\Gamma)$ from all $\xi$-ordinal labels:  $o(\Gamma) + \cs{\xi_s}{\Lambda} = \cs{\xi_s}{\Gamma}$ for all non-leaf $s\in T_\Gamma$. The class~$\bLambda$ is the ``$o(\Gamma)$-jump inversion'' of~$\bGamma$: $F\in \bGamma$ if and only if $F = H\circ g$ where $H\in \bLambda$ and~$g$ is $\bSigma^0_{1+o(\Gamma)}$-measurable. Equivalently, $g$ can be taken to be the $o(\Gamma)$-iterated jump function, relative to some oracle (see \cref{lem:closed_graph_of_x_alpha} below). Indeed, $o(\Gamma)$ can be characterised as the greatest ordinal~$\alpha$ such that $\bGamma$ is the ``$\alpha$-jump'' of a class~$\bLambda$ (\cite{LSR:reduction}). We remark that taking the $\alpha$-jumps of classes is one of the main staples of all analyses in print of Borel Wadge classes, from Wadge's \cite{Wadge:phd} to \cite{KiharaMontalban:BQO}. In contrast, our analysis does not make use of this, because the true stage relations allow us to argue about classes directly, without needing to transform them using iterated jumps. 

For our matter, \cref{thm:class_type_by_leftmost_path_labels} implies that jump and jump inversion preserve the type of a class. An acceptable~$\Gamma$ has countable type if and only if its $o(\Gamma)$-jump inversion~$\bLambda$ has countable cofinality in the Wadge degrees. Equivalently, if it has countable cofinality \emph{among all classes of level $\ge o(\Gamma)$}. For example, $\bSigma^0_{1+\xi}$ has countable type because among classes of level $\ge \xi$, it is the successor of the pair $\{\emptyset\}, \{\Baire\}$; and $D_2(\bSigma^0_{1+\xi})$ is the successor of the pair $\bSigma^0_{1+\xi}, \bPi^0_{1+\xi}$ among such classes.

\subsection{Proving the classification theorem} 
\label{sub:proving_the_classification_theorem}

We turn to the proof of \cref{thm:class_type_by_leftmost_path_labels}, which clearly implies \cref{thm:classification:main}. As one would guess, the proof is by induction on the length of the leftmost leaf~$s^*$ of~$\Gamma$. Note that this is traditional induction on~$\Nat$, not transfinite induction on the rank of~$T_\Gamma$. The induction has three parts, depending on the type and ordinal level of the default sub-class $\Gamma_{(0)}$:

\begin{proposition} \label{prop:the_inductive_proof_of_classification}
    Let~$\Gamma$ be an acceptable class description, and suppose that $o(\Gamma)<\w_1$. 
    \begin{sublemma}
        \item \label{item:classify:zero}
        If $o(\Gamma_{(0)})=\w_1$ then~$\Gamma$ has countable type. 

        \item \label{item:classify:uncountable}
        If $o(\Gamma) < o(\Gamma_{(0)}) < \omega_1$, or if $\Gamma_{(0)}$ has uncountable type, 
        then~$\Gamma$ has uncountable type. 

        \item \label{item:classify:countable}
         If $o(\Gamma_{(0)}) = o(\Gamma)$ and~$\Gamma_{(0)}$ has countable type then~$\Gamma$ has countable type. 
    \end{sublemma}
\end{proposition}

To prove \cref{prop:the_inductive_proof_of_classification} we will use:

\begin{proposition} \label{prop:main:cofinality_omega1}
    Suppose that~$\Gamma$ is classified and that $o(\Gamma)<\omega_1$. Then for all $\xi<o(\Gamma)$, for all $z\ge_\Tur y^\Gamma$, if $F\in \Delta(\Gamma(z))$ and $\dom F$ is $\Pi^0_{1+\xi}(z)$, then there is some acceptable~$\Lambda$ such that:
    \begin{itemize}
        \item $\Lambda < \Gamma$; 
        \item $\cs{y}{\Lambda}=z$; 
        \item $o(\Lambda)\ge \xi$; and
        \item $F\in \Lambda(z)$. 
    \end{itemize}
\end{proposition}

We note that for any~$\Gamma$, if the conclusion holds for some $\xi\le o(\Gamma)$ then it holds for all $\xi'\le \xi$. Hence the proposition holds for all~$\Gamma$ which have uncountable type. Indeed, what it says, in some sense, is that if $o(\Gamma)>0$ then~$\Gamma$ has ``weak uncountable type'', even if it has countable type. This will imply \cref{prop:classification:uncountable_cofinality} (apart from being used for \cref{prop:the_inductive_proof_of_classification}). Due to its length, we delay the proof of \cref{prop:main:cofinality_omega1} to the end of this section.

\subsubsection{The starred class} 
Toward proving \cref{prop:the_inductive_proof_of_classification}, we define a sub-class~$\bGamma^*$ of $\bGamma$, consisting of functions that never take the default outcome at the root. This is a subclass of $\Delta(\bGamma)$ which behaves nicely. For example, it will always have countable type (\cref{prop:Gamma_star_has_countable_type}). The main step of the argument is \cref{prop:main_classification_lemma}, which says that we can understand a set in $\Delta(\bGamma)$ by separating into two parts: on a relative closed set, we stay on the default outcome, and so in $\Delta(\bGamma_{(0)})$; on the rest, we're in $\bGamma^*$.

\begin{definition} \label{def:GammaStar}
    Let $\Gamma$ be a class description. A \emph{$\Gamma^*$-name} is a $\Gamma$-name~$N$ for which for all $\s\in S^N$, $\beta^N(\s) < \cs{\eta}{\Gamma}$. 
\end{definition}

In particular, if~$\cs{\eta}{\Gamma} = 1$ this means that $f^N = f^N_{\emptystring}$ never changes: for all $\s\preceq^z_{o(\Gamma)} \tau$ in~$S^N$, $f^N(\tau) = f^N(\s)$. 

We let $\Gamma^*(z)$ denote the collection of all functions~$F^N$, where~$N$ is a $\Gamma^*$-name and $\cs{z}{N} = z$. We extend the relations of \cref{def:effective_containment} to~$\Gamma^*$ as expected. 

\begin{lemma} \label{lem:Gamma_star_is_in_Delta}
    If~$\Gamma$ is acceptable, then for all $z\ge_\Tur \cs{y}{\Gamma}$, $\dual{\Gamma}^*(z) \equiv \Gamma^*(z)$. 
\end{lemma}

As usual, everything is as uniform as possible. Since $\Gamma^*(z)\subseteq \Gamma(z)$, it follows that $\Gamma^*(z)\subseteq \Delta(\Gamma(z))$. 

\begin{proof}
    Suppose that~$N$ is a $\dual\Gamma^*$-name. For each~$n$ let $M_{n+1}$ be a $\Gamma_{(n+1)}$-name equivalent to $N_{(n)}$; set $f^M(\s)= f^N(\s)+1$ and $\beta^M(\s) = \beta^N(\s)$, which is permitted since $\beta^N(\s) < \eta^\Gamma$. 
\end{proof}

As a result, when~$\Gamma$ is acceptable, $\Lambda\subseteq \Gamma^*$ implies $\Lambda < \Gamma^*$. 
Informally, the following proposition says that~$\Gamma^*$ has countable type.

\begin{proposition} \label{prop:Gamma_star_has_countable_type}
    For any acceptable~$\Gamma$ with $o(\Gamma)<\w_1$ there is an acceptable sequence $\bar\Theta$ satisfying:
    \begin{itemize}
        \item $\bar\Theta < \Gamma^*$;
        \item $\cs{y}{\bar\Theta} = \cs{y}{\Gamma}$;
        \item $o(\Theta_n)\ge o(\Gamma)$ for all~$n$; and
        \item For any $z\ge_\Tur \cs{y}{\Gamma}$, for any $F\in \Gamma^*(z)$, there is a partition of $X= \dom F$ into sets~$Y_n$ that are $\Sigma^0_{1+o(\Gamma)}(z)$ within~$X$ (uniformly), such that $F\rest{Y_n}\in \Theta_n(z)$, uniformly.
    \end{itemize}
\end{proposition}

\begin{proof}
    Let $\Theta_n = \Gamma_{(n)}$. \Cref{lem:acceptability_of_SU_and_subclasses} says that $\bar\Theta$ is acceptable. The (easy) proof of \cref{lem:SU:very_basic_stuff}\ref{item:SUsimple:containment} shows that $\bar\Theta < \Gamma^*$. Let~$N$ be a $\Gamma^*$-name; recall that $S^N_{(n)} = \left\{ \s\in S^N \,:\, f^N(\s)=n \right\}$. The sets $S^N_{(n)}$ are pairwise $(z,o(\Gamma))$-orthogonal, and $F^N\rest{[S_n]}\in \Theta_n(z)$. 
\end{proof}

\begin{proposition} \label{prop:main_classification_lemma}
    Suppose that $\Gamma$ is acceptable, $z\ge_\Tur \cs{y}{\Gamma}$, and $F\in \Delta(\Gamma(z))$. Then there is partition of $X = \dom F$ into a set $Y$, $\Pi^0_{1+o(\Gamma)}(z)$ within~$X$, and a set~$Z$, $\Sigma^0_{1+o(\Gamma)}(z)$ within~$X$, such that:
    \begin{itemize}
         \item  $F\rest{Y}\in \Delta(\Gamma_{(0)}(z))$; and
         \item $F\rest{Z}\in \Gamma^*(z)$.
     \end{itemize} 
\end{proposition}

\begin{proof}
    This is a ``stage comparison'' argument. Let~$N$ be a $\Gamma(z)$-name and $M$ be a $\dual{\Gamma}(z)$-name such that $F = F^N = F^M$. We may assume that $S^N= S^M$, call it~$S$. 

    Define 
    \[
        T = \left\{ \s\in S \,:\, \beta^N(\s) = \beta^M(\s) = 1  \right\};
    \]
    then~$T$ is $(z,o(\Gamma))$-closed in~$S$, so $Y = [T]$ is $\Pi^0_{1+o(\Gamma)}(z)$ within~$X$. Let $W = S\setminus T$, which is $(z,o(\Gamma))$-open in~$S$; we let $Z = [W]$. We decompose~$W$ into two parts:
    \[
        W^+ = \left\{ \s\in W \,:\,  (\exists \tau\preceq^z_{o(\Gamma)}\s)\,
        \,\,\,\beta^N(\tau) = 0 \andd \beta^M(\tau) = 1
         \right\},
    \]
    and $W^- = W\setminus W^+$. Let $Z^+ = [W^+]$ and $Z^- = [W^-]$. The sets~$Z^-$ and~$Z^+$ partition~$Z$, and are both $\Sigma^0_{1+o(\Gamma)}(z)$ in~$Z$, and so in~$X$. Note that $W^+$ and $W^-$ are $(z,o(\Gamma))$-orthogonal. 

    Since $f^N(\s) = f^M(\s)=0$ for all $\s\in T$, $N_{(0)}$ and $M_{(0)}$ with $S^{N_{(0)}} = S^{M_{(0)}} = T$ are $\Gamma_{(0)}$- and $\check{\Gamma}_{(0)}$-names for $F\rest{Y}$. 

    The restriction of~$N$ to $W^+$ shows that $F\rest{Z^+}\in \Gamma^*(z)$. Similarly, the restriction of~$M$ to $W^-$ shows that $F\rest{Z^-}\in \check\Gamma^*(z)$; by \cref{lem:Gamma_star_is_in_Delta}, $F\rest{Z^-}\in \Gamma^*(z)$. Since $W^+$ and $W^-$ are orthogonal, the proof of \cref{prop:closure:mergeing_paritioned_functions} shows that $F\rest{Z}\in \Gamma^*(z)$. 
\end{proof}

Ignoring our major debt (\cref{prop:main:cofinality_omega1}), we can now finish the proof.

\begin{proof}[Proof of \cref{prop:the_inductive_proof_of_classification}]
    \ref{item:classify:zero}: The main point in this case is that for any $z$, $\Delta(\Gamma_{(0)}(z))$ contains only the empty partial function, whence by \cref{prop:main_classification_lemma}, $\Delta(\Gamma(z)) = \Gamma^*(z)$. Let $\bar\Theta$ be as provided by \cref{prop:Gamma_star_has_countable_type} for~$\Gamma$; it is as required. 

    \medskip
    
    \ref{item:classify:uncountable}: Let $z\ge_\Tur \cs{y}{\Gamma}$ and let $F\in \Delta(\Gamma(z))$, with $X = \dom F$ being $\Pi^0_{1+o(\Gamma)}(z)$. 

    Obtain a decomposition $X = Y\cup Z$ from \cref{prop:main_classification_lemma}; then $Y$ is also $\Pi^0_{1+o(\Gamma)}(z)$.

    In either case, we can find some acceptable $\Lambda <\Gamma_{(0)}$ with $\cs{y}{\Lambda} = z$, $o(\Lambda)\ge o(\Gamma)$ and $F\rest Y\in \Lambda(z)$. If $\Gamma_{(0)}$ has uncountable type this is by definition; if $o(\Gamma) < o(\Gamma_{(0)}) < \omega_1$ this is by \cref{prop:main:cofinality_omega1}.

    Now let
    \[
        \Upsilon = \SU{o(\Gamma)}{\Lambda,\Gamma_{(0)},\Gamma_{(1)}, \Gamma_{(2)},\dots};
    \]
    except that we update the oracle of all classes to be~$z$ (as $\cs{y}{\Lambda} = z$). This is a class description since $o(\Lambda)\ge o(\Gamma)$. The sequence $\Lambda,\Gamma_{(0)},\Gamma_{(1)}, \Gamma_{(2)},\dots$ is acceptable since $\Lambda < \Gamma_{(0)}$. By \cref{lem:acceptability_of_SU_and_subclasses}, $\Upsilon$ is acceptable. By \cref{prop:main_SU_prop}\ref{item:SU:even_worse_behead}, $\Upsilon< \Gamma$. 

    The fact that $F\rest{Y}\in \Lambda(z)$ and $F\rest{Z}\in \Gamma^*(z)$ implies $F\in \Upsilon(z)$. 

    \medskip
    
    \ref{item:classify:countable}: 
    Let $\bar \Theta$ witness that $\Gamma_{(0)}$ has countable type. Let
    \[
        \Lambda_n = \SU{o(\Gamma)}{\Theta_n,\Gamma_{(0)},\Gamma_{(1)}, \Gamma_{(2)},\dots}.
    \]
    Note that $o(\Theta_n)\ge o(\Gamma_{(0)}) = o(\Gamma)$. Since $\Theta_n < \Gamma_{(0)}$, each~$\Lambda_n$ is acceptable. By \cref{prop:main_SU_prop}\ref{item:SU:even_worse_behead}, $\bar\Lambda< \Gamma$. Monotonicity of~$\bar\Theta$ implies that of~$\bar\Lambda$: the argument is similar to that of \cref{prop:main_SU_prop}\ref{item:SU:even_worse_behead}. 

    Let $z\ge_\Tur \cs{y}{\Gamma}$ and $F\in \Delta(\Gamma(z))$. Obtain a partition $\dom F = Y\cup Z$ given by \cref{prop:main_classification_lemma}; say $Y = [T]$ and $Z = [W]$ where $T$ is $(z,o(\Gamma))$-closed in~$S$ and~$W$ is $(z,o(\Gamma))$-open in~$S$ (where $\dom F = [S]$). By assumption (and because $o(\Gamma_{(0)}) = o(\Gamma)$), there are pairwise $(z,o(\Gamma))$-orthogonal $T_n\subseteq T$, $(z,o(\Gamma))$-open in~$T$, such that $Y_n = [T_n]$ partition~$Y$ and $F\rest{Y_n}\in \Theta_n(z)$. Let~$S_n$ be the $(z,o(\Gamma))$-open subset of~$S$ generated by~$T_n$. Add to~$S_0$ those $\s\in W$ with no $\preceq^z_{o(\Gamma)}$-predecessor in~$T$. Then the~$S_n$ are pairwise $(z,o(\Gamma))$-orthogonal and $X_n = [S_n]$ partition~$\dom F$. It is not difficult to see that $F\rest{[S_n]}\in \Lambda_n(z)$. 
\end{proof}

This completes the proof of \cref{thm:class_type_by_leftmost_path_labels}, and thus of \cref{thm:classification:main}.

\subsection{Consequences of classification} 
\label{sub:consequences_of_classification}

Before we prove \cref{prop:main:cofinality_omega1}, we verify the desired consequences of classification.

\begin{proof}[Proof of \cref{prop:countable_type:level_0}]
    That $\Delta(\bGamma)$ contains $\bigcup_n \bTheta_n$ follows from \Cref{rmk:type:other_directions}.  $\Delta(\bGamma)$ is closed under continuous pre-images because both $\bGamma$ and $\dual{\bGamma}$ are.
    
    \smallskip

    Suppose $\bLambda \supseteq \bigcup_n \bTheta_n$ is a principal Wadge pointclass, i.e., there is some total $G \in 2^\Baire$ with
    \[
    \bLambda = \{ G\circ f : \text{$f$ continuous}\}.
    \]
    (Again, we are identifying subsets of $\Baire$ with their characteristic functions.)
    
    Fix $F \in \Delta(\bGamma)$ total. By countable type for $o(\Gamma) = 0$, and because $F$ is total, there is actually a {\em clopen} partition of $\Baire$ into sets $Y_n$ such that $F\rest{Y_n} \in \bTheta_n$.  By \Cref{prop:closure:totalisation}, we may extend each $F\rest{Y_n}$ to a total $F_n \in \bTheta_n$.
    
    For each $n$, fix a continuous $g_n$ such that $F_n = G\circ g_n$.  Then define $g = \bigcup_n g_n\rest{Y_n}$; $g$ is continuous, with $F = G\circ g$, and so $\Delta(\bGamma) \subseteq \bLambda$.
    
    \smallskip
    
    It remains to show that $\Delta(\bGamma)$ is principal.  For each $n$, fix $H_n$ a universal $\Theta_n$ function.  Define $H$ by $H(n\conc x) = H_n(x)$.  $H \in \Delta(\bGamma)$ by \Cref{rmk:type:other_directions}.  Every total function in $\bTheta_n$ is continuously reducible to $H_n$, and so to $H$ by preppending $n$.  By the previous argument, the continuous pre-images of $H$ contains $\Delta(\bGamma)$.  Thus $\Delta(\bGamma)$ is principal.
\end{proof}

\Cref{prop:countable_type:level_0} allows us to identify the successor of a dual pair of classes, or of an increasing $\w$-sequence of classes.

\begin{proposition} \label{rmk:the_successor_class}
    Let $\bar \Theta$ be an acceptable sequence. The least principal pointclass containing~$\bigcup_n \bTheta_n$ is $\Delta(\bGamma)$, where $\Gamma = \SU{0}{0,\bar \Theta}$. 
\end{proposition}

Here again~$0$ is the standard description of $\{\emptyset\}$.  In particular, when $\bar \Theta$ is an alternating sequence $\Theta,\dual{\Theta},\Theta,\dots$, the next non-self-dual pair above the pair $\bTheta, \dual{\bTheta}$ is $\Gamma = \ClassName{BiSep}(\bSigma^0_1, \bTheta,\{\emptyset\})$ and its dual. 

\begin{proof}
By the proof of \cref{prop:Gamma_star_has_countable_type}, $\bar \Theta$ witnesses that~$\Gamma^*$ has countable type. Since $\bGamma_0 = \{\emptyset\}$, $\bar \Theta$ witnesses that~$\Gamma$ has countable type (see the proof of \cref{prop:the_inductive_proof_of_classification}\ref{item:classify:zero}). The result now follows from \cref{prop:countable_type:level_0}. 
\end{proof}

Next, to the uncountable case.

\begin{proof}[Proof of \cref{prop:classification:uncountable_cofinality}]
    The equality $\Delta(\bGamma) = \bigcup \left\{ \bLambda \,:\, \Lambda\text{ is acceptable and } \Lambda < \Gamma  \right\}$ follows from the definition of uncountable type, if~$\Gamma$ has uncountable type, and from \cref{prop:main:cofinality_omega1}, if $o(\Gamma)>0$. Since each acceptably described $\bLambda$ is non-self-dual, and $\Delta(\bGamma)$ is self-dual, it follows that $\Delta(\bGamma)$ is non-principal. 
\end{proof}

\begin{remark} \label{rmk:really_uncountable_cofinality}
    \Cref{prop:classification:uncountable_cofinality} did not actually state that in these cases, the cofinality of $\Delta(\bGamma)$ is uncountable (under the Wadge ordering). This is because this fact requires Wadge's semi-linear-ordering principle, at least for acceptably described classes, which we will only be able to deduce later (it follows from \cref{lem:SLO_withLipschitz} below; for a simpler proof, see \cite{sequel}). Using this principle, suppose, for a contradiction, that there is a countable set~$Y$ of acceptable $\Lambda < \Gamma$ such that $\Delta(\bGamma) = \bigcup \left\{ \bLambda \,:\,  \Lambda\in Y \right\}$. We can then produce a sequence $\bTheta_0,\bTheta_1,\dots$ with $\bTheta_n\subseteq \dual{\bTheta}_{n+1}$ such that $\Delta(\bGamma) = \bigcup_n \bTheta_n$. By \cref{lem:comparing_classes_bold_to_light} below, if we update to a sufficiently strong oracle, $\bar \Theta$ is an acceptable sequence. By \cref{prop:countable_type:level_0}, the sequence $\bar\bTheta$ has a least upper bound $\Delta(\bLambda)$, which is principal. Then $\Delta(\bGamma) = \bigcup_n \bTheta_n \subsetneq \Delta(\bLambda)$. But the minimality of $\Delta(\bLambda)$ implies that $\Delta(\bLambda)\subseteq \bGamma$, and since $\Delta(\bLambda)$ is self-dual, it follows that $\Delta(\bLambda)\subseteq \Delta(\bGamma)$. 
\end{remark}

The following is not needed at any point in our analysis, but it is included for the sake of aesthetics.

\begin{proposition}
    There is no acceptable description $\Gamma$ which has both countable and uncountable type.
\end{proposition}

\begin{proof}
    Note that \cref{prop:countable_type:level_0,prop:classification:uncountable_cofinality} together imply that no acceptable~$\Gamma$ with $o(\Gamma)=0$ can have both countable and uncountable type. We can then use jump inversion to conclude the general case. Alternatively, we argue directly as follows. 

    Suppose that~$\Gamma$ has both countable and uncountable type. Let~$\bar \Theta$ witness that~$\Gamma$ has countable type. By \cref{rmk:the_successor_class}, let~$\Upsilon$ be acceptable such that $\Delta(\bUpsilon)$ is the least principal pointclass continaing $\bigcup_n \bTheta_n$. By minimality, $\Delta(\bUpsilon)\subseteq \Delta(\bGamma)$. Since $\Delta(\bUpsilon)$ is principal and~$\Gamma$ has uncountable type, there is some acceptable $\Lambda<\Gamma$ with $o(\Lambda)\ge o(\Gamma)$ and $\Delta(\bUpsilon)\subseteq \bLambda$, whence $\bigcup_n \bTheta_n\subseteq \Delta(\bLambda)$. Since $o(\Lambda)\ge o(\Gamma)$, the fact that $\bar\Theta$ witnesses that~$\Gamma$ has countable type, together with \cref{prop:closure:mergeing_paritioned_functions}, shows that $\Delta(\bGamma)\subseteq \Delta(\bLambda)$. This contradicts $\Lambda<\Gamma$. 
\end{proof}

It follows that \cref{thm:class_type_by_leftmost_path_labels} gives an exact characterization of when an acceptable description has countable or uncountable type.

\subsection{Paying our debts} 
\label{sub:paying_our_debts}

It remains to give the proof of \cref{prop:main:cofinality_omega1}. 

\begin{proof}[Proof of \cref{prop:main:cofinality_omega1}] 
    We remark that the stipulation $o(\Gamma) < \omega_1$ is necessary, because if $o(\Gamma) = \omega_1$, then $\Delta(\Gamma(z)) = \{F\}$ where $F$ is the empty partial function, and there can be no $\Lambda < \Gamma$.

    Now as noted above, if the conclusion holds for some $\xi\le o(\Gamma)$ then it holds for all $\xi'\le \xi$. Hence the proposition holds for all~$\Gamma$ which have uncountable type. So it remains to show that the proposition holds for classes~$\Gamma$ which have countable type. Let~$\Gamma$ have countable type, witnessed by $\bar\Theta$; and let $\xi < o(\Gamma)$. 

    There are two main cases, depending on whether~$o(\Gamma)$ is a successor or a limit. 

    \medskip
    
    The successor case generalises the Hausdorff-Kuratowski theorem. Suppose that $o(\Gamma)$ is a successor; by our observation above, we may assume that $o(\Gamma) = \xi+1$. 

    Fix some $z\ge_\Tur y^\Gamma$, and let $F\colon X\to \{0,1\}$ be in $\Delta(\Gamma(z))$, where $X=[S]$ and~$S$ is a $(z,\xi)$-forest (in this case of the proof, we do not need to assume that~$S$ is $(z,\xi)$-closed in $\w^{<\w}$). Let $(Y_n)$ be a partition of~$X$ into uniformly $\Sigma^0_{1+o(\Gamma)}(z)$ sets, such that $F\rest{Y_n}\in \Theta_n(z)$, uniformly.
    
    As $X$ is $\Sigma^0_{1+\xi}(z) \wedge \Pi^0_{1+\xi}(z)$, and $\xi < o(\Gamma)$, $X$ is $\Delta^0_{1+o(\Gamma)}(z)$. Thus the $Y_n$ are uniformly $\Delta^0_{1+o(\Gamma)}(z)$ sets, and the function $x\mapsto n$ if $x\in Y_n$ is $\Delta^0_{1+o(\Gamma)}(z)$-measurable. By the ``limit lemma'' \cref{prop:approximations_and_alpha_plus_one_computability}, since $o(\Gamma)= \xi+1$, $f$ has a $z$-computable $(z,\xi)$-approximation on~$S$, which we also denote by~$f$. Note that here we use that~$S$ is a $(z,\xi)$-forest rather than only a $(z,o(\Gamma))$-forest. 

    By \cref{lem:existence_of_a_computable_ordinal_witness}, there is a $z$-computable ordinal~$\eta$ and a $z$-computable witness $\beta\colon S\to \eta$ for the convergence of~$f$. 

    Let
    \[
        \Upsilon = \SU{\xi}{\eta}{\bar\Theta},
    \]
    except that we set $\cs{y}{\Upsilon} = z$ rather than $\cs{y}{\bar\Theta} = \cs{y}{\Gamma}$ (this is necessary since~$\eta$ is only $z$-computable). This class is well-defined since $\xi \le o(\bar \Theta)$ (as $o(\Gamma)\le o(\bar\Theta)$). By \cref{lem:SU_remains_bounded}, $\Upsilon <\Gamma$. By \cref{prop:replacing_eta_by_1}, $\Upsilon$ is equivalent to some acceptable~$\Lambda$ with $o(\Lambda)=\xi$. So it remains to show that $F\in \Upsilon(z)$. 

    Define an $\Upsilon(z)$-name~$N$ with $S^N = S$, and at the root let $(f^N,\beta^N) = (f,\beta)$. For each~$n$, let $N_{(n)}$ be a $\Theta_n(z)$-name for an extension of $F\rest{Y_n}$ to all of~$X$ (\cref{prop:closure:totalisation,prop:closure:subdomains}). Then $F^N = F$, as required.

    \medskip

    Next, we consider the case that~$o(\Gamma) = \lambda$ is a limit ordinal. We perform an analysis similar to Wadge's analysis of $\bDelta^0_\gamma$ sets. Before we consider a particular $F\in \Delta(\bGamma)$, we first describe the possible classes we will be using for~$\Lambda$. Fix some $z\ge_\Tur \cs{y}{\Gamma}$, and recall that we fixed some $\xi<\lambda$. By \ref{TSP:limit}, let 
    \[
        \xi < \lambda_0 < \lambda_ 1< \lambda_2 < \dots,
    \]
    be increasing, $\cs{y}{\Gamma}$-computable and cofinal in~$\lambda$, such that for every $\s\in \w^{<\w}$, if $k = |\s|^z_\lambda$ is the height of~$\s$ in the tree $(\w^{<\w}, \preceq^z_\lambda)$, then $k = |\s|^z_{\lambda_k}$ and $[\s]^z_\lambda$ is $\Delta^0_{1+\lambda_k}(z)$, as for all~$\tau$, $\s\preceq^z_{\lambda} \tau \Iff\s\preceq^z_{\lambda_k}\tau$. 

    Fix a $z$-computable ordinal~$\delta$. By recursion on $\alpha\le \delta$ we define sequences $\bar \Lambda_\alpha = \Lambda_{\alpha,0},\Lambda_{\alpha,1}\cdots$ of classes, with $\cs{y}{\bar \Lambda_\alpha} = z$ and $o(\Lambda_{\alpha,k}) = \lambda_k$, as follows:
    \begin{itemize}
        \item We let 
         \[
         \Lambda_{0,k} = \SU{\lambda_k}{\bar\Theta}
         \]
         (except that as mentioned, we set the oracle to be~$z$); this is well-defined since $\lambda_k \le \lambda \le o(\bar\Theta)$. 
        \item For any~$\alpha<\delta$, we let 
        \[
            \Lambda_{\alpha+1,k} =  \SU{\lambda_k}{\Lambda_{\alpha,k},\Lambda_{\alpha,k+1},\Lambda_{\alpha,k+2},\dots}
        \]
        and again note that this is well-defined since $o(\Lambda_{\alpha,n}) = \lambda_n \ge \lambda_k$ when $n\ge k$. 
        \item For limit $\alpha\le \delta$, let $\alpha_0,\alpha_1,\dots$ be a $z$-computable, cofinal increasing sequence in~$\alpha$; we let
        \[
            \Lambda_{\alpha,k} =  \SU{\lambda_k}{\Lambda_{\alpha_k,k},\Lambda_{\alpha_{k+1},k+1},\Lambda_{\alpha_{k+1},k+2},\dots}
        \]
    \end{itemize}
    (As usual, technically, these class descriptions are defined by $z$-effective transfinite recursion.) We verify that:
    \begin{sublemma}
        \item For all~$\alpha\le \delta$, $\bar \Lambda_\alpha$ is acceptable and strictly increasing:  $\Lambda_{\alpha,k}<\Lambda_{\alpha,k+1}$ for all~$k$; and
        \item For all $\beta<\alpha\le \delta$, $\bar\Lambda_{\beta}<\Lambda_{\alpha,0}$. 
    \end{sublemma}
    (And again as usual, to make this work, everything is uniform; so for example, the computable functions that witness the monotonicity of $\bar\Lambda_\alpha$ are constructed by $z$-effective transfinite recursion.)

    Both~(a) and~(b) are proved simultaneously by induction on $\alpha\le \delta$. We start with $\alpha=0$, where the classes $\Lambda_{0,k}$ are acceptable since $\bar\Theta$ is acceptable (\cref{lem:acceptability_of_SU_and_subclasses}). We have $\Lambda_{0,k}< \Lambda_{0,k+1}$ by \cref{prop:main_SU_prop}\ref{item:SU:increase_xi}. 

    For the successor case, let $\alpha<\delta$. Each $\Lambda_{\alpha+1,k}$ is acceptable by~(a) and \cref{lem:acceptability_of_SU_and_subclasses}. To see that $\Lambda_{\alpha+1,k}<\Lambda_{\alpha+1,k+1}$ use \cref{prop:main_SU_prop}\ref{item:SU:behead},\ref{item:SU:increase_xi}. For~(b), $\bar\Lambda_\alpha < \Lambda_{\alpha+1,0}$ follows from \cref{prop:main_SU_prop}\ref{item:SU:extra_containment}. 

    For limit $\alpha\le \delta$, the argument for~(a) is the same as the successor case, noting that by induction, $\Lambda_{\alpha_0,0} < \Lambda_{\alpha_1,1}<\Lambda_{\alpha_2,2}<\cdots$. For~(b), we again note that $\Lambda_{\alpha_k,k}< \Lambda_{\alpha,0}$ and use induction: for any~$\beta<\alpha$, for sufficiently large~$k$, $\Lambda_{\beta,m} < \Lambda_{\alpha_k,k}$.

    Next, by induction on~$\alpha$, we show:
    \begin{sublemma}[resume]
        \item For all $\alpha\le k$, $\bar\Lambda_\alpha <\Gamma$. 
    \end{sublemma}

    To see this, we repeatedly use \cref{lem:SU_remains_bounded}, using the fact that $o(\Gamma)  > \lambda_k$. 

    \smallskip
    
    Now again fix some $z\ge_\Tur \cs{y}{\Gamma}$, and let $F\colon X\to \{0,1\}$ be in $\Delta(\Gamma(z))$, where $X=[S]$ and~$S$ is a $(z,\xi)$-tree, i.e., is $(z,\xi)$-closed in $\w^{<\w}$. Let $(Y_n)$ be a partition of~$X$ into sets which are uniformly $\Sigma^0_{1+\lambda}(z)$, such that $F\rest{Y_n}\in \Theta_n(z)$, uniformly.

    By \cref{lem:producing_pairwise_orthogonal}, let $S_n$ be pairwise $(z,\lambda)$-orthogonal subsets of~$S$, $(z,\lambda)$-open in~$S$, such that $Y_n = [S_n]^z_\lambda$, and such that the sets~$S_n$ are uniformly $z$-computable, and $\bigcup_n S_n$ also $z$-computable. 
    
    Let $T = S\setminus \bigcup_n S_n$. Then~$T$ is $(z,\lambda)$-closed in $\w^{<\w}$, and is well-founded, as~$X$ is covered by the $Y_n$'s. By \cref{lem:computable_rank_function}, let $\delta$ be a $z$-computable ordinal and let $r\colon T\to \delta$ be a rank function for $(T,\preceq^z_\lambda)$. We make the following modification: we assume that $r(\s)>0$ for all $\s\in T$; we then extend~$r$ to all of~$S$ by letting $r(\s)=0$ for all $\s\in S\setminus T$.

    We will show that $F\in \Lambda_{\delta,1}(z)$. For each~$n$, let~$K_n$ be a $\Theta_n(z)$-name for $F\rest{Y_n}$, with $S^{K_n}= S_n$.

    \smallskip
    
    For brevity, for $\s\in S$ let $k(\s)= |\s|^z_\lambda$; let $S_\s = \left\{ \tau\in S \,:\,  \tau\succeq^z_\lambda \s \right\}$ and $X_\s = [S_\s]= [S]\cap [\s]^z_\lambda$. Note that $S_\s$ is $(z,\lambda_{k(\s)})$-open in~$S$ (recalling that~$S$ is a $(z,\xi)$-tree and $\xi<\lambda_0$). Let
    \[
        \Lambda_\s = \Lambda_{r(\s),k(\s)+1}.
    \]
    We will show that for all $\s\in S$, $F_\s = F\rest{X_\s} \in \Lambda_\s(z)$. Which means that uniformly in $\s\in S$ we $z$-compute a $\Lambda_\s(z)$-name~$N_\s$ for~$F_\s$, with $S^{N_\s} = S_\s$. 

    \smallskip

    For $\s\in S\setminus T$, i.e., when $r(\s)=0$, this is easy, since for such~$\s$ we can compute the~$n$ such that $\s\in S_n$; from $K_n$, we can obtain a $\Theta_n(z)$-name for~$F_\s$ (i.e., we restrict~$K_n$ to~$S_\s$); since $\bar\Theta < \Lambda_{0,k(\s)+1}$, we obtain a name~$N_\s$ as required. 

    For $\s\in T$, we compute~$N_\s$ by effective transfinite recursion. For such~$\s$, we have $r(\s)>0$ and for every $\tau\in S_\s\setminus \{\s\}$, $r(\tau)< r(\s)$. So by recursion, for each such~$\tau$ we have already computed~$N_\tau$.

    We define $N_\s$ by (using the notation of \cref{def:sub-forests_when_eta_is_1}) defining $(S^{N_\s})_{(n)}$ for each~$n$, and defining $(N_\s)_{(n)}$ for all~$n$. 
    
    Let $R(\s)$ be the collection of immediate successors of~$\s$ on~$S$: those $\tau\in S_\s$ such that $k(\tau) = k(\s)+1$. 

    With oracle~$z$, compute an injective $f\colon R(\s)\to \Nat\setminus \{0\}$ such that:
    \begin{orderedlist}
        \item For all $\tau\in R(\s)$, $\Lambda_\tau \subseteq (\Lambda_\s)_{(f(\tau))}$; and
        \item The range of~$f$ is $z$-computable. 
    \end{orderedlist}
    (i) can be achieved since $r(\tau)<r(\s)$; (ii) can be done by padding. The point is that~$R(\s)$ is merely $z$-c.e., and we cannot even tell if it is empty or not. But in all cases, we will ensure that~$N_\s$ is a well-defined $\Lambda_\s$-name. 

    Now for each $m>0$, if $m = f(\tau)$ then let $(S^{N_\s})_{(m)} = S_\tau$ and let $(N_\s)_{(m)}$ be a $(\Lambda_\s)_{(m)}$-name equivalent to~$N_\tau$. If $m\notin \range f$ let $(S^{N_\s})_{(m)} = \emptyset$; we set $(S^{N_\s})_{(0)} = \{\s\}$. In either case, $(N_\s)_{(m)}$ is trivial. This defines~$N_\s$. What's important to note is that each $(N_\s)_{(m)}$ is indeed $(z,o(\Lambda_\s))$-open in~$S_\s$, as $o(\Lambda_\s) = \lambda_{k(\s)+1} = \lambda_{k(\tau)}$ for $\tau\in R(\s)$. Since $X_\s = \bigcup_{\tau\in R(\s)}X_\tau$, we have $F^{N_\s} = F\rest{X_\s}$. 

    \smallskip
    
    Since $S$ is $(z,\xi)$-closed, we have $\emptystring\in S$. Since $r(\emptystring)\le \delta$ and $k(\emptystring)=0$ we conclude, as promised, that $F = F_{\emptystring}\in \Lambda_{\delta,1}(z)$. 
\end{proof}

\section{The separation theorem(s)} 
\label{sec:the_separation_theorem}

In this section we give a new, simplified proof of the separation theorem of Louveau and Saint-Raymond~\cite{LouveauSR:WH,LouveauSR:Strength}, which will be essential for Borel Wadge determinacy. 

\begin{theorem}[Louveau \& Saint-Raymond] \label{thm:LSR_separation}
Let~$r$ be an oracle. Let~$\Gamma$ be an $r$-computable, acceptable class description; let $F \in \Gamma(r)$ be total. Let $B_0, B_1 \in \Sigma^1_1(r)$ be disjoint.  Then at least one of the following holds:
\begin{itemize}
    \item There is a Lipschitz$_1$ function $g\colon \Baire \to \Baire$  such that for all $x\in \Baire$, $g(x) \in B_{F(x)}$.
    \item There is a $G \in \Gamma(\Delta^1_1(r))$  such that for both $j < 2$, for all $x \in B_j$, $G(x) = 1-j$.
\end{itemize}
\end{theorem}

Here Lipschitz$_1$ means that $g(x)\rest{n}$ is determined by $x\rest{n}$. Recall that $\Gamma(\Delta^1_1(r)) = \bigcup \left\{ \Gamma(w) \,:\,  w\in \Delta^1_1(r) \right\}$. In terms of sets, if $F = 1_A$, the first option says that~$g$ reduces the pair $(A^\complement, A)$ to $(B_0,B_1)$. The second option, if $G = 1_C$, says that~$C$ separates~$B_0$ from~$B_1$: $B_0\subseteq C$ and $B_1\subseteq C^\complement$. 

As in \cite{BSL_paper}, and as mentioned in \cite{LouveauSR:Strength}, this theorem implies a strong version of Louveau's separation theorem (\cite{Louveau:80:Separation}), which was proved by different means in \cite{Louveau:83}.

\begin{theorem}[Louveau] \label{thm:Louveau_strong}
    Let~$r$ be an oracle; let~$\Gamma$ be an $r$-computable acceptable class description. Let~$B_0$ and~$B_1$ be disjoint $\Sigma^1_1(r)$ sets. If there is a~$\bGamma$ separator between~$B_0$ and~$B_1$, then there is a $\Gamma(\Delta^1_1(r))$ separator between~$B_0$ and~$B_1$. 
\end{theorem}

\begin{proof}
    Apply \cref{thm:LSR_separation} to~$F$, a univeral $\Gamma$-function (see \cref{lem:universal_Gamma_set}). There cannot be a continuous reduction~$g$ of~$F$ to $(B_0,B_1)$: if $G\in \bGamma$ is a separator between~$B_0$ and~$B_1$, then $1-F = G\circ g$ would be in~$\bGamma$ (\cref{prop:described_are_pointclass}), contradicting \cref{prop:universal_and_so_nonselfdual}. 
\end{proof}

\begin{corollary}[Louveau]
    \label{cor:hyp_functions_have_hyp_names}
    Let~$r$ be an oracle, and let~$\Gamma$ be an~$r$-computable, acceptable description. Suppose that $F \in 2^\Baire$ is in~$\bGamma$ and is~$\Delta^1_1(r)$. Then $F \in \Gamma(\Delta^1_1(r))$. 
\end{corollary}

We now prove \cref{thm:LSR_separation}. We assume that $o(\Gamma) < \omega_1$, as otherwise the result is immediate. The argument generalises the proof we gave in \cite{BSL_paper}, where $\bGamma = \bSigma^0_{\xi}$. Hence, we mainly give the details of how to adapt that argument to this more general setting. 

Fix a $\Gamma(r)$-name~$N$ for~$F$ and $r$-computable trees~$T_0$ and~$T_1$ such that for $j=0,1$, $B_j = \{ y : (y,z) \in [T_j]\}$.  For simplicity of notation, we assume that $r = \emptyset$. 

\begin{remark} \label{rmk:game:simplifying_assumption}
    Without loss of generality we may make the following assumptions about~$N$ by modifying the functions $f_t^N$. Recall the notation of \cref{def:function_defined_by_name}: the leaf $\ell^N(\s)$ of~$T_\Gamma$ is defined for finite strings~$\s$ as well as infinite ones. 
    \begin{orderedlist}
        \item If $|\s|\le 1$ then $\ell^N(\s)$ is the leftmost leaf of~$T_\Gamma$, reached by always taking the default outcome~0.
        \item For any non-leaf $t\in T_\Gamma$, $m>0$, and $\s\in \w^{<\w}$, if $\ell^N(\s)\succeq t\conc m$ and~$\s$ is $\preceq_{\xi_t}$-minimal such that $f_t^N(\s) = t\conc m$, then $\ell^N(\s)$ is the leftmost leaf of~$T_\Gamma$ extending~$t\conc m$. 
    \end{orderedlist}
    Making such changes to the functions $f_t^N$ does not change the leaf~$\ell^N(x)$ for infinite~$x$, and so does not change the function $F = F^N$ on~$\Baire$.
\end{remark}

\subsection{The game}

The game $G_\Gamma(N, B_0, B_1)$ is essentially the same as the one used in \cite{BSL_paper}. Two players take turns; at each step, player~I plays a single natural number, and player~II plays a pair of natural numbers. As in \cite{BSL_paper}, the idea is that player~I plays a real $x\in \Baire$, and player~II plays $y, z\in \Baire$; $y$ is player~II's proposed value for $g(x)$, and a subsequence of $z$ should witnesses that $y$ is in the correct~$B_j$, depending on the value~$F(x)$. The appropriate subsequence is determined by the true stage machinery. 

As in \cite{BSL_paper}, the game is open for player~I. We index the plays so that after~$n$ steps, player~I has played $(x_0,\dots, x_{n-1})$ and player~II has played $(y_1,z_1)$, $(y_2,z_2)$, \dots, $(y_n,z_n)$. We write $\bar x = (x_0,\dots, x_{n-1})$. We need to determine whether player~I has already won the game. We do this by using~$N$ and~$\bar x$ to approximate $F(x)$; we choose a subsequence of the $z_i$'s and check whether player~II is still on the tree~$T_j$ corresponding to the guess about~$F(x)$. The new aspect of this calculation is that the level~$\delta$ which we use to choose the $z_i$'s is determined by our approximation of~$F(x)$. 

Specifically, recall that $F^N(\bar x)$ is defined to be the label $\Gamma(s)$ that~$\Gamma$ gives to the leaf $s=\ell^N(\bar x)$, and is defined for finite sequences~$\bar x$, as well as elements of Baire space. Our assumption that $o(\Gamma)<\w_1$ means that~$s$ is not the root of~$T_\Gamma$, so has a parent (immediate predecessor)~$s^-$ on~$T_\Gamma$; we let $\delta = \delta(\bar x) = \cs{\xi_{s^-}}{\Gamma}$, the ordinal level at which the last choice $f_{s^-}^N(\bar x) = s$ was made. 

We then let
\[
    D = \left\{ i\in \{1,\dots, n\} \,:\,  \bar x\rest{i} \preceq_{\delta(\bar x)} \bar x\,\,\andd\,\, \ell^N(\bar x\rest{i}) = \ell^N(\bar x) \right\};
\]
note that $n\in D$, but that $0\notin D$. Enumerate~$D$ as $D = \{a_0 < a_1 < \cdots < a_{k-1}\}$, so $k\ge 1$; let $\bar y = (y_1,\dots, y_k)$ and $\bar z = (z_{a_0}, z_{a_1}, \dots, z_{a_{k-1}})$. We declare that player~II has not yet lost the game if $(\bar y,\bar z)\in T_j$, where $j = F^N(\bar x)$. 

As this is an open game for player I, it is determined. Note that the game is computable. 

\subsection{If player II has a winning strategy}

Suppose that player~II has a winning strategy. For $x \in \Baire$, let $(y,z)$ be the sequences generated by player II playing according to this strategy when player I plays~$x$; we set $g(x) = y$. Certainly~$g$ is continuous, indeed Lipschitz$_1$. We need to show that~$g$ reduces~$F$ to $(B_0,B_1)$, i.e., that for all $x$, $g(x) \in B_{F(x)}$.

To do this, fix $x\in \Baire$; let $s = \ell^N(x)$ and $j = \Gamma(s) = F(x)$; let $\delta = \delta(x) = \cs{\xi_{s^-}}{\Gamma}$ for~$s^-$ the parent of~$s$ on~$T_\Gamma$. 

\begin{lemma} \label{clm:infinitely_many_true_stages_with_correct_leaf}
   For all but finitely many $\rho\prec_\delta x$, $\ell^N(\rho) = s$. 
\end{lemma}

\begin{proof}
    By induction on the length of $t\preceq s$ (on $T_\Gamma$) we show that for all but finitely many  $\rho \prec_\delta x$, $\ell^N(\rho)\succeq t$. This is immediate for $t = \seq{}$. Suppose this holds for $t\prec s$, and let $t^+\preceq s$ be the child of~$t$ on~$T_\Gamma$ in the direction of~$s$. Since $t \preceq s^-$, $\xi_t  = \cs{\xi_t}{\Gamma} \le \delta$. For all but finitely many $\rho\prec_{\xi_t} x$ we have $f_t(\rho) = f_t(x) = t^+$; and $\rho \prec_\delta x$ implies $\rho\prec_{\xi_t} x$ (\ref{TSP:nested}). 
\end{proof}

As there are infinitely many $\delta$-true initial segments of~$x$ (\ref{TSP:unique_path}), let $a_0 < a_1 < \cdots$ enumerate those $i>0$ for which $x\rest{i} \prec_{\delta} x$ and $\ell(x\rest{i}) = s$. Let $v = (z_{a_0}, z_{a_1}, \dots)$. For each $k<\w$, observe the step at the end of which player  player~I has played $\bar x = x\rest{a_k}$. Since $\bar x\prec_\delta x$, the set~$D$ computed at this step is $\{a_0,\dots, a_{k}\}$ (essentially, \ref{TSP:trees}). Since player~II did not lose, $(y\rest k,v\rest{k})\in T_j$. Hence, $(y,v)\in [T_j]$, so $y\in B_j$, as required.

\subsection{If player I has a winning strategy}

Suppose now that player I has a winning strategy~$Q$.  Since the game is effectively open for player I, we may take $Q\in \Delta^1_1$. We use~$Q$ to compute a $\Gamma$-name~$M$ for a separator between~$B_0$ and~$B_1$. For each $y\in \Baire$, we need to find a leaf $\ell^M(y)$ which will of course determine $F^M(y)$. As in \cite{BSL_paper}, the idea is to examine possible second-coordinate sequences~$\s$ for player~II to play along~$y$. We use~$Q$ to pull back the true stage relations among such strings~$\s$, and so identify ``correct'' strings that will indicate where we should go on~$T_\Gamma$. 

In \cite{BSL_paper}, we used $\xi$-correct strings to put~$y$ inside a $\Sigma^0_{1+\xi}(Q)$ separator. In the current proof, we need to apply the same idea at every non-leaf $t\in T_\Gamma$; correct strings will guide us towards non-default outcomes $t\conc m$ for $m>0$. This is an inductive process (on the length of~$t$). The key difficulty is that as we go up the tree, the ordinals~$\xi_t$ may increase; but to define~$f^M_t$, we are only allowed to check correctness up to level~$\xi_t$, and not higher ordinals. That is, we cannot require a string~$\rho$, that guides us to a non-default outcome $t\conc m$, to be more than $\xi_t$-correct. But if $\xi_{t\conc m}>\xi_t$, such a string~$\rho$ may fail to be $\xi_{t\conc m}$-correct, and the inducive procedure of producing longer and longer correct extensions may stall. For this reason, once we get to $t\conc m$, we may need to ignore~$\rho$ and its predecessors, and use the more elaborate notion of $\xi$-correctness \emph{above~$\rho$}. 

To the details. For $y \in \Baire$ and $\s \in \omega^{<\omega}$, let $Q(y, \s)$ be the sequence $\seq{x_0, \dots, x_n}$ which results from player I playing according to~$Q$ and player II playing $(y\rest{|\s|}, \s)$.  Note that $|Q(y,\s)|= |\s|+1$.\footnote{In \cite{BSL_paper}, we used a ``string"~$\pi$ of length~$-1$ so that $Q(y,\pi)= \seq{}$. This was needed as a base case for generating a sequence of $\xi$-correct strings; the string~$\pi$ was $\xi$-correct, whereas the string~$\seq{}$ was not necessarily $\xi$-correct. In the current construction, as we use the more elaborate notion of correctness above a string~$\rho$, the extension \cref{lem:alpha_correct_extensions_exist_1} below does not require such a base case, so we do not need the ``string'' $\pi$.}  

\begin{definition}
Let $\alpha$ be a computable ordinal, and let $y \in \Baire$. For $\sigma, \tau \in \omega^{<\omega}$, we write $\sigma \tleq_\alpha^y \tau$ when $\sigma \preceq \tau$ and $Q(y, \sigma) \preceq_\alpha Q(y, \tau)$.
\end{definition}

Observe that $\tleq_\alpha^y$ is $Q$-computable, uniformly in~$y$ and $\alpha$.

\begin{definition} \label{def:legal}
    Let $y \in \Baire$. A string $\s\in \w^{<\w}$ is \emph{legal for~$y$} if in the play of the game in which player~I played $Q(y,\s)\rest{|\s|}$ and player~II played $(y\rest{|\s|},\s)$, player~I has not yet won.\footnote{Note that it is still possible that after player~I plays the last entry of $Q(y,\s)$, player~II has no response that would prevent~I from winning; we still consider~$\s$ to be legal for~$y$.}
\end{definition}

Note that if $\s$ is legal for~$y$ then every $\eta\preceq \s$ is also legal for~$y$. 

\begin{definition}\label{def:alpha-correct}
For $y \in \Baire$, a computable ordinal~$\alpha$, and $\rho \prec\sigma \in \omega^{<\omega}$, we define what it means for~$\sigma$ to be \emph{$\alpha$-correct} or \emph{strongly $\alpha$-correct for~$y$ above~$\rho$}. 
\begin{itemize}
\item $\sigma$ is 0-correct for~$y$ above~$\rho$ if it is legal for~$y$. 
\item $\sigma$ is $(\alpha+1)$-correct for $y$ above~$\rho$ if it is strongly $\alpha$-correct for $y$ above~$\rho$, and for every $\tau \treq{y}{\alpha} \sigma$ which is strongly $\alpha$-correct for $y$ above~$\rho$, we have $\tau \treq{y}{\alpha+1}  \sigma$.
\item For $\lambda$ a limit, $\sigma$ is $\lambda$-correct for~$y$ above~$\rho$ if it is strongly $\beta$-correct for~$y$ above~$\rho$ for all $\beta < \lambda$.
\item $\sigma$ is strongly $\alpha$-correct for~$y$ above~$\rho$ if for every~$\tau$ with $\rho \prec \tau \tleq^y_\alpha \sigma$, $\tau$ is $\alpha$-correct for $y$ above~$\rho$. 
\end{itemize}
\end{definition}

Note that only the definition of strong correctness involves~$\rho$; the first three items are identical to \cite[Def. 4.2.3]{BSL_paper}. In the definition of strong correctness, note that the extension $\rho\prec \tau$ is proper. Indeed, we do not define the notion of a string being correct above itself. 

\smallskip

This definition of correctness may seem odd.  Let us outline the plan of the construction, so that we may give a bit of motivation for it.  After establishing some results about how correctness behaves, we will use this definition to construct a set $G$, which is our intended separator.  If $G$ fails to be a separator, witnessed by some $y$, then we will use $y$ to construct a series of plays for player II which defeats $Q$, contrary to assumption.  Fix $v$ with $(y,v) \in [T_j]$ for the appropriate $j$.

Suppose $\delta(Q(y,\sigma)) = \alpha$, meaning that $\alpha$ is the largest ordinal used by $N$ to send $Q(y,\sigma)$ to $\ell^N(Q(y,\sigma))$.  If player II has played $(y\rest{|\sigma|}, \sigma)$, and as their next move they play $(y(|\sigma|), b)$, then the sequence $\bar{z}$ which is constructed to determine if they have lost the game is precisely $(\sigma(|\rho_0|), \sigma(|\rho_1|), \dots, b)$, where $(\rho_i)$ enumerates those $\rho \tle_\alpha^y \sigma$, in increasing order.  We would like to arrange that $\bar{z} \prec v$, so $b$ should be $v(n)$ for the correct $n$.  Then, by some sequence of legal plays, we would hope to return to $\ell^N(Q(y,\sigma))$.  That is, we would like to find a $\tau \supseteq \sigma\cat v(n)$ with $\ell^N(Q(y,\tau)) = \ell^N(Q(y,\sigma))$.

$\alpha$-correctness turns out to be sufficient to ensure this.  If $\sigma$ is $\alpha$-correct above some $\rho$, there will be a $\tau \supseteq \sigma\cat v(n)$ with $\sigma \tleq_\alpha^y \tau$ and $\tau$ is strongly $\alpha$-correct above $\sigma$.  So we may have player II play according to $(y,\tau)$, and then play $(y(|\tau|), v(n+1))$ after that.  Then we can repeat.  In this fashion we build an infinite sequence of plays such that player II never loses, always playing elements of $v$ at the appropriate moment.

One concern is that if we have just found $\tau$ as described above, there may be $\nu$ with $\sigma \tle_\alpha^y \nu \tle_\alpha^y \tau$.  Then when we play $v(n+1)$, the sequence $\bar{z}$ which is constructed to check if our position is legal will contain more than just the elements of $v$; it will also contain the extraneous element $\tau(|\nu|)$.  This is why it is important that $\tau$ is {\em strongly} $\alpha$-correct above $\sigma$: this implies that $\nu$ is $\alpha$-correct above $\sigma$, and so we could have taken $\nu$ instead of $\tau$.

\smallskip

We now continue with the details.

\begin{lemma} \label{clm:correctness_basics}
Let $y\in \Baire$, $\alpha$ a computable ordinal, and $\rho\prec \s\in \w^{<\w}$.
\begin{sublemma}
    \item \label{item:correctness_basics:strong_implies_weak} 
    If $\sigma$ is strongly $\alpha$-correct for~$y$ above~$\rho$ then it is $\alpha$-correct for~$y$ above~$\rho$.
    
    \item \label{item:correctness_basics:strong-0-correct}
     $\s$ is strongly 0-correct for~$y$ above~$\rho$ if and only if it is legal for~$y$.

    \item \label{item:correctness_basics:smaller_beta}
     If $\sigma$ is $\alpha$-correct for~$y$ above~$\rho$, then it is strongly $\beta$-correct for $y$ above~$\rho$, for every $\beta < \alpha$.

    \item \label{item:correctness_basics:initial_segment1}
     If $\sigma$ is strongly $\alpha$-correct for~$y$ above~$\rho$, then any~$\nu$ with $\rho \prec \nu\tleq^y_\alpha \s$  is also strongly $\alpha$-correct for~$y$ above~$\rho$. 

    \item \label{item:correctness_basics:initial_segment2}
     If $\sigma\preceq \tau$ and both~$\s$ and~$\tau$ are $\alpha$-correct for~$y$ above~$\rho$, then  $\s \tleq^y_\alpha \tau$. 
\end{sublemma}
\end{lemma}

\begin{proof}
    The same as \cite[Claim 4.2.4]{BSL_paper}. 
\end{proof}

The following claim has the same proof as \cite[Claim 4.2.5]{BSL_paper}.

\begin{lemma} \label{clm:correctness_complexity}
Let $\alpha$ be a computable ordinal. The relations ``$\s$ is $\alpha$-correct for~$y$ above~$\rho$'' and ``$\s$ is strongly $\alpha$-correct for~$y$ above~$\rho$'' are:
\begin{orderedlist}
\item $\Delta^0_1(Q)$ if $\alpha=0$;
\item $\Pi^0_\alpha(Q)$ if $\alpha\ne 0$ is finite; 
\item $\Delta^0_\alpha(Q)$ for limit $\alpha$; and
\item $\Pi^0_{\alpha-1}(Q)$ if~$\alpha$ is a successor and infinite. 
\end{orderedlist}
\end{lemma}

The following is new.

\begin{lemma} \label{lem:chaining_corrects}
    Let~$\alpha$ be a computable ordinal and let $y\in \Baire$. Let $\rho_0\prec \rho_1$ and suppose that~$\rho_1$ is strongly $\alpha$-correct for~$y$ above~$\rho_0$. 
    Then the following are equivalent for all $\s\succ \rho_1$:
    \begin{equivalent}
        \item $\s$ is strongly $\alpha$-correct for~$y$ above~$\rho_0$. 
        \item $\s$ is strongly $\alpha$-correct for~$y$ above~$\rho_1$. 
    \end{equivalent}
\end{lemma}

\begin{proof}
    We prove the claim by induction on~$\alpha$. For $\alpha=0$ the claim is immediate (see \cref{clm:correctness_basics}\ref{item:correctness_basics:strong-0-correct}). Suppose that $\alpha>0$ and that the claim holds for all $\beta<\alpha$; and suppose that~$\rho_1$ is strongly $\alpha$-correct for~$y$ above~$\rho_0$. 

    First, we observe that any $\s\succ \rho_1$ is $\alpha$-correct for~$y$ above~$\rho_0$ if and only if it is $\alpha$-correct for~$y$ above~$\rho_1$. This is not difficult to check using \cref{def:alpha-correct}, whether~$\alpha$ is limit or a successor. 

    Given this equivalence, (1)$\Then$(2) follows easily. For (2)$\Then$(1), we need to consider strings~$\nu$ such that $\rho_0 \prec \nu \preceq \rho_1$ and $\nu\tleq^y_\alpha \s$; we need to show that such~$\nu$ are $\alpha$-correct for~$y$ above~$\rho_0$. We have just observed that~$\s$ is $\alpha$-correct for~$y$ above~$\rho_0$; and we are assuming that~$\rho_1$ is (strongly) $\alpha$-correct for~$y$ above~$\rho_0$. By \cref{clm:correctness_basics}\ref{item:correctness_basics:initial_segment2}, $\rho_1\tleq^y_\alpha \s$. By \ref{TSP:trees}, $\nu\tleq^y_\alpha \rho_1$. Since~$\rho_1$ is strongly $\alpha$-correct for~$y$ above~$\rho_0$, $\nu$ is $\alpha$-correct for~$y$ above~$\rho_0$, as required. 
\end{proof}

\begin{lemma}\label{lem:alpha_correct_extensions_exist_1}
Let $y \in \Baire$ and $\rho \in \omega^{<\omega}$.  If~$\sigma$ is a one element extension of~$\rho$ that is legal for~$y$, then for all computable~$\alpha$ there is a $\tau \succeq \sigma$ that is strongly $\alpha$-correct for~$y$ above $\rho$.
\end{lemma}

\begin{proof}
    The proof of \cite[Claim 4.2.6]{BSL_paper} gives this lemma, and in fact, is slightly simplified, since we do not need to consider $\nu\preceq \rho$. In the limit case, since we may have $\rho\ntleq^y_\alpha \tau$, we take $k = |\rho|+2$; this (and the minimality of~$\tau$) ensures that $|Q(y,\tau)|_\alpha\le k$. 
\end{proof}

We now work towards defining the $\Gamma$-name~$M$. By induction on the length of $t\in T_\Gamma$ we define a set $U_t \subseteq \Baire\times \w^{<\w}$. We start by letting, for $t = \seq{}$ the root of~$T_\Gamma$,  
\begin{itemize}
    \item    $U_{\seq{}} = \left\{ (y,\seq{}) \,:\, y\in \Baire   \right\}$. 
\end{itemize}
Suppose that~$U_t$ has been defined, and that~$t$ is not a leaf of~$T_\Gamma$. We let 
\begin{itemize}
      \item  $U_{t\conc 0} = U_t$;
      \item For $m>0$, $(y,\tau)\in U_{t\conc m}$ if there is some $\s\prec \tau$ such that:
      \begin{orderedlist}
            \item $(y,\s)\in U_t$;
          \item $\tau$ is strongly $\xi_t$-correct for~$y$ above~$\s$;
          \item $\ell^N(Q(y,\tau))\succeq t\conc m$; and
          \item $\tau$ is $\tleq_{\xi_t}^y$-minimal with respect to (ii)\&(iii). 
      \end{orderedlist}
\end{itemize}



By induction on the length of~$t$, using \cref{clm:correctness_complexity}, we observe that for non-leaf~$t$, for all $m>0$, the set 
\[
    Y_{t\conc m} = \left\{ y\in \Baire \,:\,  (\exists \s)\,\,(y,\s)\in U_{t\conc m} \right\}
\]
is $\Sigma^0_{1+\xi_t}(Q)$. By the effective reduction property for the class $\Sigma^0_{1+\xi_t}(Q)$, we can find, for $m>0$, $\Sigma^0_{1+\xi_t}(Q)$ sets $X_{t\conc m}$ that reduce the sets $Y_{t\conc m}$: $X_{t\conc m}\subseteq Y_{t\conc m}$ for all~$m$, $\bigcup_{m>0} X_{t\conc m} = \bigcup_{m>0} Y_{t\conc m}$, but the sets $X_{t\conc m}$ are pairwise disjoint. 

Since $\eta_t = 1$ for all~$t$ (as the class description~$\Gamma$ is acceptable), we can use the sets $X_{t\conc m}$ to define a $\Gamma$-name~$M$ so that for all~$t$,~$y$ and $m>0$, 
\[
  \ell^M(y)\succ t \,\,\,\Then\,\,\,  [y\in X_{t\conc m} \Iff \ell^M(y)\succeq t\conc m].
\]
See \cref{sub:the_case_eta_1_}. Technically, to define~$M$ (with $\cs{z}{M} = Q$ and $S^M = \w^{<\w}$), for each non-leaf~$t$, by \cref{lem:producing_pairwise_orthogonal}, we find, for $m>0$, pairwise $(Q,\xi_t)$-orthogonal sets $W_{t\conc m}\subseteq \w^{<\w}$, uniformly~$Q$-computable, such that $\bigcup_{m>0} W_{t\conc m}$ is~$Q$-computable as well (and all of this is uniform in~$t$), and such that $[W_{t\conc m}]^Q_{\xi_t} = X_{t\conc m}$; and we define~$f_t^M$ by letting $f_t^M(\nu) = m$ if $\nu\in W_{t\conc m}$, $f_t^M(\nu)=0$ if $\nu\notin \bigcup_{m>0} W_{t\conc m}$. This defines~$M$.\footnote{Note that~$N$ is a \emph{computable} name, and so the game, the relations $\tleq^y_\alpha$, and the notions of $\alpha$-correctness, all used the unrelativised true stage relations $\preceq_{\xi_t}$. In contrast, $M$ is only~$Q$-computable, and so needs to use the relativised relations~$\preceq^Q_{\xi_t}$.}

We let $G = F^M$, so $G\in \Gamma(Q)$. The following lemma concludes the proof of \cref{thm:LSR_separation}.

\begin{lemma} \label{lem:game:final}
   $G$ defines a separator between~$B_0$ and~$B_1$: For both $j =0,1$, for all $y \in B_j$, $G(y) = 1-j$.
\end{lemma}

\begin{proof}
    Suppose not; fix $j\in \{0,1\}$ and $y\in B_j$ such that $G(y)=j$. Let $s = \ell^M(y)$ (so $j = \Gamma(s)$). Also, fix $v\in \Baire$ such that $(y,v)\in [T_j]$; we write $v = (v_0,v_1,\dots)$. Let~$\delta = \xi_{s^-}$, where~$s^-$ is the parent of~$s$ on~$T_\Gamma$. 

    Since~$y$ is fixed, for simplicity of notation:
    \begin{itemize}
        \item We omit ``for~$y$'' when we talk about~$\alpha$-correctness and legality;
        \item We write $\tleq_\alpha$ for $\tleq^y_\alpha$; 
        \item For $\s\in \w^{<\w}$,  we let $\ell(\s) = \ell^N(Q(y,\s))$. 
    \end{itemize}

    For legal $\s\in \w^{<\w}$ let 
     \[
     E = \left\{ i<|\s| \,:\,  (\s\rest{i})\tle_\delta \s\,\,\andd\,\, \ell(\s\rest{i}) = \ell(\s)  \right\},  
     \]
    and let $\bar z(\s) = (\s(a_0),\s(a_1),\dots, \s(a_{k-1}))$ where $E = \{a_0 < a_1 < \dots < a_{k-1}\}$; we can have $E = \emptyset$, in which case $\bar z(\s)$ is the empty string. Let
    \[
        K = \left\{ \s\in \w^{<\w} \,:\, \s \text{ is legal, }\,\ell(\s)=s,  \andd  \bar z(\s) \prec v  \right\}. 
    \]
    The point is that $\{i+1\,:\, i\in E\}\cup \{|\s|+1\}$ is the set~$D$ of lengths which is used to determine whether~I has won the game after~I plays $Q(y,\s)$ and II plays $(y\rest{(|\s|+1)},\s\conc b)$ for any $b\in \w$. So if $\s$ is legal and $\Gamma(\ell(\s)) = j$ (in particular, if  $\ell(\s)=s$), then for any $b\in \w$, $\s\conc b$ is legal if and only if $(y\rest{(k+1)}, \bar z(\s)\conc b)\in T_j$. Hence:

    \begin{claim} \label{clm:K_is_good}
        Let $\s\in K$; let $k = |\bar z(\s)|$. Then $\s\conc v_k$ is legal. Further, if $\tau\succeq \s\conc v_k$ is legal, $\s\tle_\delta \tau$, $\ell(\tau)=s$, and there is no~$\nu$ with $\s\tle_\delta \nu \tle_\delta \tau$, then $\tau\in K$. 
    \end{claim}

    The plan is to find some ``nice'' string in~$K$ from which we can keep extending to longer strings in~$K$, as described by the claim, and thus build a winning play for~II against~$Q$, contradicting~$Q$ being a winning strategy for~I. Indeed, we will show:

    \begin{claim} \label{clm:there_is_a_nice_string_in_K}
        There are $\s\prec \theta$ such that:
        \begin{orderedlist}
            \item $\theta\in K$;
            \item $\theta$ is strongly $\delta$-correct over~$\s$; and
            \item for every $\eta\succ \s$ that is strongly $\delta$-correct over~$\s$, $\ell(\eta)= s$. 
        \end{orderedlist}
    \end{claim}

    \Cref{clm:there_is_a_nice_string_in_K} suffices to complete the proof of \cref{lem:game:final}. Indeed, starting with $\theta_0 = \theta$ and~$\s$ given by the claim, we define an increasing sequence $\theta_0 \prec\theta_1\prec \theta_2\prec \cdots$ of strings $\theta_e\in K$ that are each strongly $\delta$-correct over~$\theta$. Given~$\theta_e$, let $k = |\bar z(\theta_e)|$. As $\theta_e\conc v_k$ is legal (\cref{clm:K_is_good}), \cref{lem:alpha_correct_extensions_exist_1} gives us some~$\theta_{e+1}$, an extension of $\theta_e\conc v_k$, that is strongly $\delta$-correct over~$\theta_e$; we choose a $\preceq$-minimal such string. By \cref{lem:chaining_corrects}, $\theta_{e+1}$ is strongly $\delta$-correct above~$\s$, and so by choice of~$\theta$, $\ell(\theta_{e+1}) = s$. By \Cref{clm:correctness_basics}\ref{item:correctness_basics:initial_segment2}, $\theta_e \tle_\delta \theta_{e+1}$. By minimality of~$\theta_{e+1}$ (and \cref{clm:correctness_basics}\ref{item:correctness_basics:initial_segment1}), there is no~$\nu$ with $\theta_e\tle_\delta \nu \tle_\delta \theta_{e+1}$. Hence, by \cref{clm:K_is_good}, $\theta_{e+1}\in K$. 

    \medskip
    
    It suffices, then, to prove \cref{clm:there_is_a_nice_string_in_K}. 

    Let $t_1\prec t_2 \prec \cdots \prec t_{n}$ list those nodes $t\prec s$ at which~$s$ takes a non-default outcome: $t\conc m\preceq s$ for some $m>0$. Let $t_i^+$ denote the child of~$t_i$ extended by~$s$. Note that $n=0$ is possible: when~$s$ is the leftmost node on~$T_\Gamma$. 

    If $n>0$, then for all $i\in \{1,\dots, n\}$, $y\in Y_{t_i^+}$, and so, starting at~$t_n$ and working our way down, we can choose strings~$\s_i$ such that $(y,\s_i)\in U_{t_i^+}$ and such that~$\s_{i-1}$ witnesses this fact; here $\s_0 = \seq{}$. We let $\xi_i = \xi_{t_i}$. If $n=0$ then we just choose $\s_0  = \seq{}$. 

    \begin{claim} \label{clm:tleq_between_sigma_is}
        For all $1\le i \le j \le n$, $\s_i \tleq_{\xi_i} \s_j$. 
    \end{claim}
    
    \begin{proof}
        We may assume that $i<j$. Repeated applications of \cref{lem:chaining_corrects} give that~$\s_j$ is $\xi_{i}$-correct above~$\s_{i-1}$. The claim then follows from \cref{clm:correctness_basics}\ref{item:correctness_basics:initial_segment2}. 
    \end{proof}
    
    \begin{claim} \label{clm:game:none_in_between}
        For all $i = 1,\dots, n$, for all $\tau \tle_{\xi_i}\s_i$, $\ell(\tau)\nsucceq t_i^+$. 
    \end{claim}

    \begin{proof}
        We prove the claim by induction on~$i$. Suppose that it has been proved for all $j<i$. Let $\tau \tle_{\xi_i} \s_i$. Let $j\le i$ be least such that $\tau \prec \s_j$ (so $j\ge 1$). By \cref{clm:tleq_between_sigma_is} and \ref{TSP:trees}, $\tau \tle_{\xi_j} \s_j$. Since $t_i^+\succeq t_j^+$, it suffices to show that $\ell(\tau)\nsucceq t_j^+$. 

        Suppose that $\tau = \s_{j-1}$. If $j=1$, then by~(i) of \cref{rmk:game:simplifying_assumption}, $\ell(\tau)$ is the leftmost leaf of~$T_\Gamma$ and so does not extend $t_1^+$. If $j>1$ then by the induction assumption for~$j-1$, and by~(ii) of \cref{rmk:game:simplifying_assumption}, $\ell(\tau)$ is the leftmost leaf of~$T_\Gamma$ extending $t_{j-1}^+$, and so does not extend $t_j^+$. 

        Suppose that $\s_{j-1}\prec \tau$. Then by \cref{clm:correctness_basics}\ref{item:correctness_basics:initial_segment1}, $\tau$ is $\xi_j$-correct above $\s_{j-1}$. If $\ell(\tau)\succeq t_j^+$ then~$\tau$ would contradict the minimality of~$\s_j$ (part~(iv) of the definition of $U_{t_j^+}$). 
    \end{proof}
    
    \begin{claim} \label{clm:game:s_is_chosen}
        For all $\eta\succ \s_n$ that is strongly $\delta$-correct above~$\s_n$, $\ell(\eta)=s$. 
    \end{claim}
    
    \begin{proof}
        Let $\eta\succ \s_n$ be strongly $\delta$-correct above~$\s_n$. By induction on the length of $r\preceq s$ we show that $r\preceq \ell(\eta)$. Suppose that $r\preceq s$ and~$r\preceq \ell(\eta)$. There are two cases. 

        First, suppose that $r = t_i$ for some $i\ge 1$. By \cref{lem:chaining_corrects}, $\eta$ is strongly $\xi_i$-correct above $\s_{i-1}$. By \cref{clm:correctness_basics}\ref{item:correctness_basics:initial_segment2}, $\s_i\tle_{\xi_i} \eta$. Since $f^N_{t_i}(Q(y,\s_i))= t_i^+$ is not the default outcome of~$t_i$ and $Q(y,\s_i)\prec_{\xi_i} Q(y,\eta)$, we must have $f^N_{t_i}(Q(y,\eta)) = t_i^+$ as well, i.e., $t_i^+\preceq \ell(\eta)$. 

        Otherwise, $r\conc 0\preceq s$. Let $i\in \{1,\dots, n\}$ be greatest such that $t_i \prec r$; let $i=0$ if there is no such~$i$. As $\delta \ge \xi_r$, by \cref{lem:chaining_corrects}, $\eta$ is $\xi_r$-correct above~$\s_i$. If $\ell(\eta)\succeq r\conc m$ for some $m>0$, then~$\eta$ would show that there is some~$\tau\succ \s_i$ such that $(y,\tau)\in U_{r\conc m}$, so $y\in Y_{r\conc m}$, in which case, we would not have $r\conc 0\preceq \ell^M(y) = s$. 
    \end{proof}

    \Cref{clm:game:none_in_between} and \cref{rmk:game:simplifying_assumption} imply that~$\ell(\s_n)=s$, and further, that $\bar z(\s_n) = \seq{}$. By \cref{clm:K_is_good}, $\s_n\conc v_0$ is legal. By \cref{lem:alpha_correct_extensions_exist_1}, let~$\theta\succeq \s_n\conc v_0$ be strongly $\delta$-correct above~$\s_n$, and choose minimal such. 
    
    \begin{claim}\label{clm:tleq_between_sigman_theta}
    $\s_n \tle_{\xi_n} \theta$.
    \end{claim}
   
    \begin{proof}
    Same as the proof of \Cref{clm:tleq_between_sigma_is}.
    \end{proof}

    \begin{claim} \label{clm:game:last}
        If $\tau \tle_\delta \theta$ and $\ell(\tau) = s$ then $\tau = \s_n$. 
    \end{claim}
    
    \begin{proof}
      Let $\tau \tle_\delta \theta$. By minimality of~$\theta$, we have $\tau \preceq \s_n$. Suppose that $\tau \prec \s_n$ (and in particular, $\sigma_n \neq \seq{}$, so $n > 0$). By \cref{clm:tleq_between_sigman_theta} and \ref{TSP:trees}, $\tau \tle_{\xi_n} \s_n$, so the claim follows by \cref{clm:game:none_in_between}.
    \end{proof}

    By \cref{clm:game:s_is_chosen}, $\ell(\theta) = s$. Now there are two possibilities. If $\s_n \tle_\delta \theta$ then $\bar z(\theta) = \seq{v_0}$. If not, then $\bar z(\theta) = \seq{}$. In either case, $\bar z(\theta)\prec v$. Hence $\theta\in K$. This completes the proof of \cref{clm:there_is_a_nice_string_in_K}, and so of \cref{lem:game:final}, and so of \cref{thm:LSR_separation}. 
\end{proof}

\begin{remark} \label{rmk:where_did_we_use_acceptability_here}
    The proof of \cref{thm:classification:main} relied heavily on the acceptability of the class description~$\Gamma$. In contrast, the proof of \cref{thm:LSR_separation} did not at all use the fact that sequences $\seq{\Gamma_{t\conc m}}$ are monotone. In the construction of the~$\Gamma$-name~$M$ and its analysis, we used the fact that $\cs{\eta_t}{\Gamma}=1$ for all non-leaf $t\in T_\Gamma$. This was mostly for convenience. We could have proved the theorem for all class descriptions~$\Gamma$; in defining~$M$, we would need to define level sets $Y_{t\conc m,\beta}$ for all $\beta < \cs{\eta_t}{\Gamma}$, and then use the reduction property for the classes $D_{\eta_t}(\Sigma^0_{1+\xi_t})$. The current formulation of \cref{thm:LSR_separation} suffices, as we will use both it and \cref{thm:classification:main} in the next section to show that every non-self-dual Borel Wadge class has an acceptable description. 
\end{remark}

\section{Borel Wadge determinacy} 
\label{sec:a_proof_of_borel_wadge_determinacy}

As mentioned in the introduction, Louveau and Saint-Raymond \cite{LouveauSR:Strength} used \cref{thm:LSR_separation} and Louveau's version of \cref{thm:classification:main} from \cite{Louveau:83} to give a proof of Borel Wadge determinacy in second-order arithmetic. In this section we use our work above to give a proof in $\ATR_0+\Pind$, the subsystem of second-order arithmetic consisting, in addition to the base system~$\RCA_0$ of computable mathematics, both the principle of arithmetic transfinite recursion, and induction on~$\Pi^1_1$ subsets of~$\Nat$. 

Working in the restricted setting of~$\ATR_0$, we need to be careful about how exactly we formalise set-theoretic notions in the language second-order arithmetic. In this setting, it is also important to fill in the details of some steps that are missing in \cite{LouveauSR:Strength}.

\subsection{Formalisation} 
\label{sub:formalisation}

So far in this paper, all of our reasoning has occurred within $\ATR_0$, including the construction in \cref{sec:true_stage_machinery}. In particular, the true stage relations are constructed via arithmetic transfinite recursion, and their various properties are provable in~$\ATR_0$. As a result, for a $\Gamma$-name~$N$ (for some class description~$\Gamma$) and $x\in \Baire$, the value $F^N(x)$ is well-defined in~$\ATR_0$ (without the need for evaluation maps). $\ATR_0$ implies open determinacy, and so \cref{thm:LSR_separation} is provable in $\ATR_0$. For the rest of this section, unless otherwise stated, we work in $\ATR_0$. 

For a Borel code~$B$ and a class description~$\Gamma$, we write $B\in \bGamma$ to mean that there is some $\Gamma$-name~$N$ with $F^N= 1_B$. So that we can use familiar notation, we henceforth identify sets and their characteristic functions. We also use ``Borel set'' to mean a Borel code. For class descriptions~$\Gamma$ and~$\Lambda$, we write $\bGamma\subseteq \bLambda$ if $B\in \bGamma\Then B\in \bLambda$ for all Borel~$B$; we write $\bGamma< \bLambda$ if $\bGamma\subseteq \bLambda$ and $\bGamma\subseteq \dual{\bLambda}$. 

\ref{TSP:Sigma_alpha_sets}, together with effective transfinite recursion, shows that we can effectively pass from an acceptable class description~$\Gamma$ and a~$\Gamma$-name~$N$ to a Borel code of~$F^N$. 

Note that the construction of a name~$N_\Gamma$ of a universal function~$F_\Gamma$ for a class description~$\Gamma$ is effective (given~$\Gamma$). 

\begin{lemma} \label{lem:containment_and_game}
    Let~$B$ be a Borel set, and let $\Gamma$ be an acceptable class description. Then $B\in \bGamma$ if and only if player~I has a winning strategy in the game $G_\Gamma(N_\Gamma,B,B^\complement)$. 
\end{lemma}

\begin{proof}
    If player~I has a winning strategy then $B\in \bGamma$. If player~II has a winning strategy then we get a Wadge reduction of~$F_\Gamma$ to~$B^\complement$; by \cref{prop:universal_and_so_nonselfdual} (and \cref{prop:described_are_pointclass}), $B\notin \bGamma$. 
\end{proof}

For a Borel set~$B$, we let
\begin{itemize}
    \item $[B]_W$ be the Wadge class of~$B$: the collection of pullbacks $1_B\circ g$ for continuous functions~$g$; 
    \item $[B]_L$ be the Lipschitz class of~$B$: the collection of pullbacks $1_B\circ g$ for Lipschitz$_1$  functions~$g$. 
\end{itemize}

An immediate corollary of \cref{thm:LSR_separation} and \cref{lem:containment_and_game} is:

\begin{lemma} \label{lem:SLO_withLipschitz}
    Let~$B$ be a Borel set, and let~$\Gamma$ be an acceptable class description. If $B\notin\bGamma$ then $\check{\bGamma}\subseteq [B]_L$. 
\end{lemma}




\subsection{Well-foundedness} 
\label{sub:well_foundedness}

Our next goal is to show that every Borel Wadge class either has an acceptable description, or is the ambiguous class of a class with an acceptable description. The plan is to find a minimal acceptable~$\Gamma$ such that $B\in \bGamma$, and then argue that $[B]_W$ is either $\bGamma$ or $\Delta(\bGamma)$. Thus, we need the relation $<$ on acceptably-described Wadge classes to be well-founded. 

\begin{proposition}[Martin] \label{lem:no_descending_sequence_of_descriptions}
    There is no infinite sequence $\seq{\Gamma_n}$ of acceptable class descriptions such that for all~$n$, $\bGamma_{n+1}< \bGamma_n$. 
\end{proposition}

We can follow the standard argument (see for example \cite[(21.15)]{Kechris:book} or \cite[Thm.~3.5]{LouveauSR:Strength}). Since the map $\Gamma\mapsto N_\Gamma$ is computable, the sequence $\seq{N_{\Gamma_n}}$ exists. Let $F_n = F_{\Gamma_n}$. The arguments in print use the determinacy of the games $G_L(F_{{n}}, F_{n+1})$ and $G_L(F_{{n}}^\complement, F_{n+1})$.  This determinacy shows that there are Lipschitz$_{1/2}$ reductions of $F_{n+1}$ to~$F_n$ and to~$F_{n}^\complement$. Here a function $g\colon \Baire\to \Baire$ being Lipschitz$_{1/2}$ means that $g(x)\rest{k+1}$ is determined by $x\rest{k}$; such reductions are precisely winning strategies for player~I in $G_L(F_{n}, F_{n+1})$. 

We can show the existence of such reductions, indeed that the sequence of these reductions exists, without proving Wadge determinacy in general. We use our construction of the universal functions $F_\Gamma$. We note that $\bGamma\subseteq \bLambda$ does not imply $\Gamma\subseteq \Lambda$ (as defined in \cref{def:effective_containment}). However, this can be obtained by relativising to a not-too-complicated oracle. For a class description~$\Gamma$ and an oracle $z\ge_\Tur \cs{y}{\Gamma}$, let $\Gamma^z$ be the class description which is exactly the same as~$\Gamma$, except that we let $\cs{y}{\Gamma^z} = z$. 

\begin{lemma} \label{lem:comparing_classes_bold_to_light}
    For two acceptable class descriptions~$\Gamma$ and~$\Lambda$, $\bGamma\subseteq \bLambda$ if and only if for some oracle $w\ge_\Tur \cs{y}{\Gamma},\cs{y}{\Lambda}$ with $w\in \Delta^1_1(\seq{\cs{y}{\Gamma},\cs{y}{\Lambda}})$, $\Gamma^w \subseteq \Lambda^w$. 
\end{lemma}

\begin{proof}
    We apply \cref{thm:LSR_separation} with the game $G_\Lambda(N_\Lambda,F_\Gamma,F_\Gamma^\complement)$. Player~II cannot have a winning strategy since $\bLambda \nsubseteq \dual{\bGamma}$, (since $\bGamma \nsubseteq \dual{\bGamma}$). Hence, there is some $w\in \Delta^1_1(\cs{y}{\Gamma},\cs{y}{\Lambda})$ (which we may assume computes both $\cs{y}{\Gamma}$ and $\cs{y}{\Lambda}$) such that $F_\Gamma\in \Lambda(w)$. Now, given some $\Gamma(z)$-name~$M$, where $z\ge_\Tur w$, we can effectively find~$a\le_\Tur z$ such that~$F^M$ is the $a$-section of~$F_\Gamma$; using a $\Lambda(w)$-name of $F_\Gamma$, we find a $\Lambda(z)$-name equivalent to~$M$ (we use the effectivity of \cref{prop:described_are_pointclass}).  
\end{proof}

Recall that the complete function~$F_\Gamma$ was defined using an effective pairing function $(x,y)\mapsto \seq{x,y}$ from $\Baire^2$ to~$\Baire$. The standard pairing function $(x,y)\mapsto x_0y_0x_1y_1\dots$ has the property that for all~$a$, $y\mapsto \seq{a,y}$ is Lipschitz$_{1/2}$. Using this pairing function, we get:

\begin{lemma} \label{lem:1/2-reduction_to_complete_set}
    Let~$\Gamma$ be a class description. For all~$B\in \bGamma$, there is a Lipschitz$_{1/2}$ reduction of~$B$ to $F_\Gamma$. 
\end{lemma}

\Cref{lem:comparing_classes_bold_to_light} holds for infinite sequences. In particular, if $\seq{\Gamma_n}$ is a sequence of acceptable class descriptions with $\bGamma_0> \bGamma_1 > \cdots$, then there is a single oracle~$w$ such that $\Gamma_0^w > \Gamma_1^w > \cdots$, uniformly. This is because $\ATR_0$ implies open determinacy for infinite sequences of open games.\footnote{Specifically, a winning strategy for~I in an open game is (uniformly) computed from the rank function on the well-founded tree of positions at which~I has not yet won; we can join infinitely many well-founded trees to get ranking functions for all.} Thus, after relativising to an oracle, in the proof of \cref{lem:no_descending_sequence_of_descriptions}, we may assume that $\Gamma_{n+1}<\Gamma_n$, uniformly. By \cref{lem:1/2-reduction_to_complete_set}, then, we get the desired Lipschitz$_{1/2}$ reductions of~$F_{n+1}$ to~$F_n$ and~$F_n^\complement$. The last step of the proof is the fact that Borel sets have the property of Baire; this is provable in~$\ATR_0$ (the standard construction can be carried out in $\ATR_0$; see, for example, \cite{DzhafarovFloodEtAl}). Hence, \cref{lem:no_descending_sequence_of_descriptions} is provable in~$\ATR_0$.

\subsubsection{Minimal elements} 

The relation $\left\{ (\Gamma,\Lambda) \,:\, \bGamma<\bLambda  \right\}$ is a relation on reals. Thus, from \cref{lem:no_descending_sequence_of_descriptions} we cannot immediately deduce stronger formulations of this relation being well-founded. Indeed, very strong forms of dependent choice on reals are not even provable in second-order arithmetic. To find minimal elements, we restrict to a countable set. 

\begin{lemma}[$\ATR_0+\Pind$] \label{lem:minimal_class}
    Let~$r$ be an oracle, and let~$B$ be a Borel set with an $r$-computable code. There is an acceptable description~$\Gamma\in \Delta^1_1(r)$ with $B\in \bGamma$, such that for any acceptable description $\Lambda\in \Delta^1_1(r)$, if $\bLambda < \bGamma$ then $B\notin \bLambda$.
\end{lemma}

We will shortly see, though, that the class description~$\bGamma$ is actually minimal among all acceptable class descriptions: the condition $\Lambda\in \Delta^1_1(r)$ can be omitted. We need the following complexity calculations.

\begin{lemma} \label{lem:complexity_of_acceptable_descriptions_and_membership}
    The following relations and sets are all~$\Pi^1_1$:
    \begin{orderedlist}
        \item The collection of acceptable class descriptions;
        \item The relation $B\in \bGamma$ between a Borel code~$B$ and an acceptable class description~$\Gamma$; 
        \item The relation $\bGamma < \bLambda$ between acceptable class descriptions~$\Gamma$ and~$\Lambda$. 
    \end{orderedlist}
\end{lemma}

\begin{proof}
    First, we observe that the relation $F^N(x)= F^M(x)$ for a real $x\in \Baire$ and $\Gamma$- and $\Lambda$-names~$N$ and~$M$ (for any class descriptions~$\Gamma$ and~$\Lambda$, not necessarily acceptable), is~$\Pi^1_1$. Hence, the relation $N\equiv M$ (\cref{def:equivalence_of_names}) is also~$\Pi^1_1$. Hence, the relation $\Gamma\subseteq \Lambda$ (\cref{def:effective_containment}) is $\Pi^1_1$. This implies that being a monotone sequence of class descriptions (\cref{def:monotone}) is~$\Pi^1_1$. This gives~(i). Throughout, we used that the class $\Pi^1_1$ is closed under number quantifiers (effective countable unions and intersections); this is provable in~$\ATR_0$, as~$\ATR_0$ implies $\Sigma^1_1$-choice (see \cite[VIII.3.21]{Simpson}). 

    (ii) follows from \cref{lem:containment_and_game} and open determinacy: $B\in \bGamma$ if and only if player~II does not have a winning strategy in the game $G(N_\Gamma,B,B^\complement)$.  (Since player~II's side of the game is closed, a given strategy being a winning strategy for them is arithmetic.)

    (iii) follows from the map $\Gamma\mapsto N_\Gamma$ being computable. $\bGamma\subseteq \bLambda$ if and only if $F_\Gamma\in \bLambda$ (and recall that we effectively obtain a Borel name for~$F_\Gamma$).  
\end{proof}

\begin{proof}[Proof of \cref{lem:minimal_class}]
    For simplicity, assume that~$r$ is computable. Let $\+C$ be the collection of all $e\in \Nat$ which are $\Delta^1_1$-codes of an acceptable class descriptions~$\Gamma$ such that $B\in \bGamma$. 

    $\+C$ is not empty: for some computable ordinal~$\xi$, we have $B\in \Sigma^0_\xi$, and there is a computable acceptable description of this class. 

    $\+C$ is $\Pi^1_1$. This follows from \cref{lem:complexity_of_acceptable_descriptions_and_membership}, and the fact that being a $\Delta^1_1$-code is $\Pi^1_1$. Note that we are not assuming that~$\+C$ exists; it is merely definable. 

    For $a\in \+C$, let $\Gamma_a$ be the class description coded by~$a$. For $a,b\in \+C$, let $a\prec b$ if $\bGamma_a < \bGamma_b$. By \cref{lem:complexity_of_acceptable_descriptions_and_membership}, this relation is~$\Pi^1_1$. 

    The principle of $\Pi^1_1$-dependent choice for relations on numbers is provable in $\ATR_0+\Pind$ (\cite[VIII.4.10]{Simpson}). Hence, if~$\+C$ does not have a $\prec$-minimal element, then \cref{lem:no_descending_sequence_of_descriptions} fails. 
\end{proof}

\subsection{All classes are described; SLO; determinacy} 
\label{sub:all_classes_are_described_slo_determinacy}

We can now obtain all the results we are after.

\begin{theorem}[$\ATR_0+\Pind$] \label{thm:every_class_acceptable}
    Every Borel Wadge class is either $\bGamma$ or $\Delta(\bGamma)$ for some accptable~$\Gamma$. 
\end{theorem}

\begin{proof}
    Let~$B$ be a Borel set. Let~$r$ be an oracle that computes a Borel code for~$B$. Let~$\Gamma$ be supplied by \cref{lem:minimal_class}. We show that $[B]_W = \bGamma$ or $[B]_W = \Delta(\bGamma)$. For simplicity, assume that~$r$ is computable. 

    If $B\notin \dual{\bGamma}$ then $[B]_W = \bGamma$ follows from \cref{lem:SLO_withLipschitz}. Indeed, in this case, we get $\bGamma = [B]_W = [B]_L$. 

    Suppose, then, that $B\in \Delta(\bGamma)$. We now use \cref{thm:classification:main}. As $\Delta(\bGamma)\ne \emptyset$, $\Gamma$ does not have zero type. By \cref{cor:hyp_functions_have_hyp_names}, $B\in \Delta(\Gamma(\Delta^1_1))$. If~$\Gamma$ has uncountable type, or $o(\Gamma)>0$, then by definition, or by \cref{prop:main:cofinality_omega1}, $B\in \bLambda$ for some acceptable $\Lambda\in \Delta^1_1$ with $\Lambda < \Gamma$, contrary to the minimality of~$\Gamma$. Hence, $o(\Gamma)=0$ and~$\Gamma$ has countable type; let $\seq{\Theta_n}$ witness this fact. Since $\Theta_n\in \Delta^1_1$ and $\Theta_n<\Gamma$ for all~$n$, the minimality of~$\Gamma$ implies that $B\notin \bigcup_n \bTheta_n$. By \cref{lem:SLO_withLipschitz}, and since the sequence $\seq{\Theta_n}$ is monotone, $\bigcup_n \bTheta_n \subseteq [B]_W$. By \cref{prop:countable_type:level_0}, $\Delta(\bGamma)\subseteq [B]_W$. 
\end{proof}

We can deduce the semi-linear ordering property for Borel Wadge classes:

\begin{theorem}[$\ATR_0+\Pind$]
    For any two Borel sets~$A$ and~$B$, $A\le_W B$ or $B^\complement \le_W A$. 
\end{theorem}

\begin{proof}
    If $[B]_W$ is non-self-dual, then the result follows from \cref{lem:SLO_withLipschitz}. 

    Suppose that $[B]_W = \Delta(\bGamma)$ for some non-self-dual~$\bGamma$. If $A\nle_W B$ then either $A\notin \bGamma$ or $A\notin \dual{\bGamma}$; in either case, the result again follows from \cref{lem:SLO_withLipschitz}. 
\end{proof}

Finally, for Borel Wadge determinacy. The proof in the non-self-dual case is simplified by our construction of complete sets. 

\begin{theorem} \label{thm:BorelWadgeDeterminacy}
    For any two Borel sets~$A$ and~$B$, the game $G_L(A,B)$ is determined. 
\end{theorem}

\begin{proof}
    Recall that this means: either $A\le_W B$ by a Lipschitz$_1$ reduction, or $B^\complement \le_W A$ by a Lipshitz$_{1/2}$ reduction. 

    If $[B]_W = \bGamma$ is non-self dual, then we have already observed that $\bGamma = [B]_L$. Hence, if $A\le_W B$ then $A\le_W B$ by a Lipschitz$_1$ reduction. Otherwise, by \cref{lem:SLO_withLipschitz}, $F_\Gamma^\complement \le_W A$ by a Lipschitz$_1$ reduction. By \cref{lem:1/2-reduction_to_complete_set}, $B\le_W F_\Gamma$ via a Lipschitz$_{1/2}$ reduction. Composing, we get $B^\complement \le_W A$ via a Lipschitz$_{1/2}$ reduction. 

    Suppose that $[B]_W = \Delta(\bGamma)$. Hence,~$\Gamma$ has countable type and $o(\Gamma)= 0$; let $\bar{\Theta}$ witness this. If $A\in \bigcup_n \bTheta_n$ then as observed above, as $\bar\Theta$ is monotone, $\bigcup_n \bTheta_n \subseteq [B]_L$. Thus, the last case we need to consider is $[A]_W = [B]_W = \Delta(\bGamma)$. 

    Let~$T_A$ be the collection of $\s\in \w^{<\w}$ such that $A\cap [\s]\notin \bigcup_n \bTheta_n$; similarly define~$T_B$. By \cref{prop:closure:subdomains}, $T_A$ and~$T_B$ are trees. Since there is a clopen partition of Baire space into sets~$C$ with $A\cap C\in \bigcup_n \bTheta_n$, the tree~$T_A$ is well-founded; similarly, so is~$T_B$. By assumption, the empty string is on both~$T_A$ and~$T_B$. The trees exist by $\Delta^1_1$-comprehension. 

    Now players~I and~II play an auxiliary game, taking turns choosing natural numbers, and trying to stay on their respective tree: player~II wins if there is a step of the game at which player~I played~$\tau\notin T_A$ and player~II played $\s\in T_B$. This is a clopen game and so one of the players has a winning strategy. We claim that the same player has a winning strategy in $G_L(A,B)$. 

    Suppose, for example, that~II has a winning strategy in the auxiliary game. At the winning step of the auxiliary game, say player~I has played $\tau\notin T_A$ and player~II has played $\s\in T_B$; note that $|\tau| = |\s|+1$. There is some~$n$ such that $A\cap [\tau]\in \bTheta_n$. By \cref{prop:closure:mergeing_paritioned_functions}, there is some~$m$ such that $B\cap [\s\conc m]\notin \dual{\bTheta_n}$. Player~II chooses this~$m$. Now by \cref{lem:SLO_withLipschitz}, $A\cap [\tau]\in [B\cap [\s\conc m]]_L$, so player~II can use this reduction to win the rest of $G_L(A,B)$. The same argument works when player~I wins the auxiliary game. 
\end{proof}

\section{Developing the true stage machinery}
\label{sec:true_stage_machinery}

In this section we define the true stage relations and prove they have the various desirable properties we have been using. Our development closely follows the one in \cite{GreenbergTuretsky:Pi11}. In that paper, however, we did not use the machinery to construct subsets of Baire space; rather, we used it to build a computable countable structure. We therefore defined relations $s\le_{\alpha}t$ for $s,t\le \w$; in the notation of the current paper, we used the restriction of $\preceq_\alpha$ to the initial segments of $x= 0^\infty$. The first step to extending that machinery to all of Baire space is uniform relativisation: $s\le_{\alpha,x} t$ means that~$s$ appears~$\alpha$-true to~$t$ relative to~$x$. However, it is soon observed that this is a ``Lipschitz'' relation: the $x$-use is the identity, i.e., $s\le_{\alpha,x} t$ only depends on $x\rest{t}$. We can therefore dispense with the oracle and write $x\rest{s} \preceq_\alpha x\rest{t}$, giving the notation that we used in \cite{BSL_paper} and here.

\subsection{Definition and properties of the true stages relations} 
\label{sub:definition_and_basic_properties_of_the_true_stages_relations}

We fix a computable ordinal~$\delta$ (see \cref{sub:computable_ordinals}). To define the relation $\preceq_\delta$, we will need to define all relations $\preceq_\alpha$ for all $\alpha\le \delta$. For $\alpha < \delta$, we let~$n_\alpha$ be the element in the field of~$\delta$ which is the least upper bound of~$\alpha$ in~$\delta$. When we write $n_\alpha < n_\beta$ we mean the natural ordering $(\Nat,<)$, so this is different from $\alpha<\beta$.

\subsubsection*{Jumps of strings} 

For a string~$\s\in \w^{\le \w}$, let $\s'$ denote the collection of inputs on which a universal Turing machine with oracle sequence~$\s$ halts in fewer than $|\s|$ many steps. Thus the jump of the empty string is empty. If $\s\preceq \tau$ then $\s'\subseteq \tau'$. We assume that if~$\s$ is a one-entry extension of~$\tau$, then $\s'$ contains at most one more element than~$\tau'$. Thus, for every string~$\s$ we get an enumeration of the elements of~$\s'$ in order of which converged earlier (i.e.\ with shorter oracle). For~$\s\in \w^{<\w}$ we let $p(\s)$ be the last element enumerated into~$\s'$ according to this ordering. In other words, $\s'= \tau' \cup \{p(\s)\}$ for some $\tau\prec \s$. If $\s' = \emptyset$ then we let $p(\s)=-1$. If $x\in \w^{\w}$ is infinite then we let $p(x) = \infty$.

\subsubsection*{Definition} 

By induction on $\alpha\le \delta$ we define the relation $\s\preceq_\alpha \tau$. Simultaneously, for $\alpha<\delta$ we also define strings $\tau^{(\alpha)}$. Suppose that $\preceq_\alpha$ is already defined. We will show shortly that $\s\preceq_\alpha \tau$ implies $\s\preceq \tau$, so that for each~$\tau$, $\{\s\,:\, \s\prec_\alpha \tau\}$ is naturally linearly ordered. 
\begin{itemize}
    \item The string~$\tau^{(\alpha)}$ is defined to be the increasing enumeration of all strings $\s\prec_\alpha \tau$, \emph{excluding the first~$n_\alpha$ many such strings}.
\end{itemize}
We will see that this exclusion is what makes the machinery work at limit levels. If~$\tau$ has at most $n_\alpha$-many $\prec_\alpha$-predecessors, then $\tau^{(\alpha)}$ is the empty sequence. Note that $\s\preceq_\alpha \tau$ implying $\s\preceq \tau$ also ensures that $\tau^{(\alpha)}$ is finite whenever~$\tau$ is finite; we will work to show that $x^{(\alpha)}$ is infinite when~$x$ is infinite. 

Each $\tau^{(\alpha)}$ is a sequence of finite sequences, i.e., an element of $(\w^{<\w})^{\le \w}$. However, after using some computable coding of finite sequences by natural numbers (a computable bijection between $\w^{<\w}$ and~$\w$), we can consider each $\tau^{(\alpha)}$ as a string in $\w^{\le \w}$, and so use it as an oracle in computations. 

For infinite $x\in \Baire$, $x^{(\alpha)}$ will be Turing equivalent to the iteration of the Turing jump of length~$\alpha$, starting with~$x$. For finite $\tau\in \w^{<\w}$, $\tau^{(\alpha)}$ is~$\tau$'s guess about $x^{(\alpha)}$ for~$x$ extending~$\alpha$. The guess will be correct when $\tau\prec_\alpha x$. 

The definition of $\preceq_{\alpha}$ is as follows. Let $\s, \tau \in \w^{\le \omega}$.

\begin{itemize}
    \item  $\s\preceq_{0} \tau \,\,\Iff\,\,\s \preceq \tau$. 
    \item If~$\alpha$ is a limit ordinal, then $\s\preceq_{\alpha} \tau$ if for all $\beta<\alpha$, $\s\preceq_{\beta} \tau$. 
    \item If $\alpha< \delta$ then $\s\preceq_{\alpha+1} \tau$ if $\s\preceq_{\alpha} \tau$, and there is no $e < p(\s^{(\alpha)})$ in $(\tau^{(\alpha)})'\setminus (\s^{(\alpha)})'$.
\end{itemize}

In the successor case, $\alpha+1$ denotes the ordinal $\beta\le \delta$ with order-type $\otp(\alpha)+1$. The motivation for the successor step was discussed in \cite{BSL_paper}. In brief, the definition says that when trying to compute $(\s^{(\alpha)})'$ (which for infinite~$\s$ will be Turing equivalent to $\s^{(\alpha+1)}$), the string~$\s$ only commits to numbers below  $p(\s^{(\alpha)})$, the last number  enumerated into the jump of $\s^{(\alpha)}$. If $\s \npreceq_{\alpha} \tau$ then~$\tau$ thinks that~$\s$ was likely wrong about the oracle $\s^{(\alpha)}$, and so there is no meaningful comparison between their jumps $(\s^{(\alpha)})'$ and $(\tau^{(\alpha)})'$. Suppose, however, that $\s \preceq_{\alpha} \tau$. We will shortly show that $\s^{(\alpha)}\preceq \tau^{(\alpha)}$, and so $(\s^{(\alpha)})'\subseteq (\tau^{(\alpha)})'$. So~$\s$'s commitment about $(\s^{(\alpha)})'$ is discovered to be false by~$\tau$ exactly when some number $e<p(\s^{(\alpha)})$, which~$\s$ claims is not in $(\s^{(\alpha)})'$, is enumerated into $(\tau^{(\alpha)})'$.

\subsubsection*{Basic properties} 

We gradually verify the properties of the relations $\preceq_\alpha$ as listed in \cref{sec:preliminaries}. First, we remark that for each $\alpha <\delta$, $\preceq_\alpha$ as defined above does not depend on~$\delta$, only on~$\alpha$. In other words, if we only started with~$\alpha$ and performed the same construction, we would arrive with the same relation~$\preceq_\alpha$. (Note though that $\s^{(\alpha)}$ depends on~$\alpha+1$, not just~$\alpha$, because its definition involves~$n_\alpha$.) Now, a simple induction shows that $\alpha \le \beta$ and $\s\preceq_\beta\tau$ implies $\s\preceq_\alpha\tau$. This, together with the definition of $\preceq_0$, establishes \ref{TSP:nested}. 

\smallskip

The following lemma includes the property \RefClub. In \cref{sec:preliminaries} we derived this property from \ref{TSP:successor}, but in fact, we need it to establish all the other properties. 
 
\begin{lemma} \label{lem:first_pass_relations:basic_properties} 
    For all $\alpha\le\delta$,
    \begin{enumerate}
        \item[(a)] The relation~$\preceq_\alpha$ is a partial ordering. 
        \item[($\Diamond$)] For all $\s \preceq \rho \preceq \tau \in \w^{\le \omega}$, if $\s,\rho \preceq_\alpha \tau$, then $\s\preceq_{\alpha} \rho$. 
        \item[($\Club$)] 
        For all $\s\preceq_\alpha \rho \preceq_\alpha \tau \in \w^{\le \omega}$, if $\s\preceq_{\alpha+1} \tau$ then $\s\preceq_{\alpha+1} \rho$. 
        \item[(b)] If $\s\preceq_\alpha \tau$ then $\s^{(\alpha)} \preceq \tau^{(\alpha)}$.
        \item[(c)] If $\s\preceq_{\alpha+1}\tau$ then $p(\s^{(\alpha)})\le p(\tau^{(\alpha)})$.
    \end{enumerate}
\end{lemma}

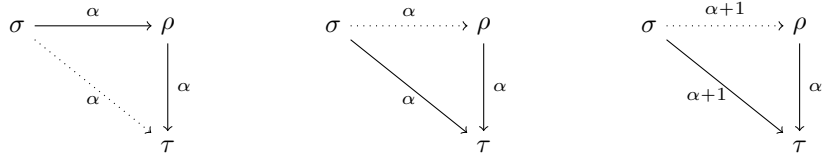
\begin{figure}[h]
    \centering

    \begin{tikzpicture}[scale=1]
        \def\hdistance{2}
        \def\vdistance{-1.6}
        \def\ddistance{4.2}

        \node (spo) at (0,0) {$\s$};
        \node (rpo) at (\hdistance,0) {$\rho$};
        \node (tpo) at (\hdistance,\vdistance) {$\tau$};

        \node (sdi) at (\ddistance,0) {$\s$};
        \node (rdi) at (\ddistance+\hdistance,0) {$\rho$};
        \node (tdi) at (\ddistance+\hdistance,\vdistance) {$\tau$};

        \node (scl) at (2*\ddistance,0) {$\s$};
        \node (rcl) at (2*\ddistance+\hdistance,0) {$\rho$};
        \node (tcl) at (2*\ddistance+\hdistance,\vdistance) {$\tau$};

        \draw[->] (spo) -- (rpo)
            node[pos=0.5,above] {$\scriptstyle\alpha$};
        \draw[->] (rpo) -- (tpo)
            node[pos=0.5,right] {$\scriptstyle\alpha$};
        \draw[->,dotted] (spo) -- (tpo)
            node[pos=0.5,below] {$\scriptstyle\alpha$};

        \draw[->,dotted] (sdi) -- (rdi)
            node[pos=0.5,above] {$\scriptstyle\alpha$};
        \draw[->] (rdi) -- (tdi)
            node[pos=0.5,right] {$\scriptstyle\alpha$};
        \draw[->] (sdi) -- (tdi)
            node[pos=0.5,below] {$\scriptstyle\alpha$};
    
        \draw[->,dotted] (scl) -- (rcl)
            node[pos=0.5,above] {$\scriptstyle\alpha+1$};
        \draw[->] (rcl) -- (tcl)
            node[pos=0.5,right] {$\scriptstyle\alpha$};
        \draw[->] (scl) -- (tcl)
            node[pos=0.6,left] {$\scriptstyle\alpha+1$};
    \end{tikzpicture}

    \caption{From left to right: the transitivity of $\preceq_{\alpha}$, the property $(\Diamond)_{\alpha}$, and the property $(\Club)_{\alpha}$ (given $(\Diamond)_{\alpha}$).}
    \label{fig:diamond_club_diagram}
\end{figure}

The converse of~(b) may fail because of the exclusion of the first $n_\alpha$ many $\prec_\alpha$-predecessors in the definition of $\tau^{(\alpha)}$. Nonetheless, this converse is ``close'' to the truth, and this gives an informal motivation for two of the properties illustrated in \cref{fig:diamond_club_diagram}. The relation $\preceq_{\alpha}$ is transitive because if $\s  \preceq_{\alpha} \rho \preceq_{\alpha} \tau $ then $\s^{(\alpha)}\preceq \rho^{(\alpha)} \preceq \tau^{(\alpha)}$, and so $\s^{(\alpha)}\preceq \tau^{(\alpha)}$. The property $(\Diamond)$ is similar: if $\s^{(\alpha)},\rho^{(\alpha)}\preceq \tau^{(\alpha)}$ then $\s^{(\alpha)}$ and $\rho^{(\alpha)}$ are comparable.

\begin{proof}
    First, we observe that (a)$_{\alpha}$ and $(\Diamond)_{\alpha}$ imply (b)$_{\alpha}$, (c)$_{\alpha}$, and $(\Club)_{\alpha}$. 

    By the definition of~$\tau^{(\alpha)}$, to show (b)${}_{\alpha}$ we need to show that if $\rho \prec \s \preceq_{\alpha} \tau$ then $\rho \preceq_{\alpha} \s \iff \rho \preceq_{\alpha} \tau$. One direction follows from (a)${}_{\alpha}$, the other from $(\Diamond)_{\alpha}$. 

    For (c)$_{\alpha}$, suppose that $\s \preceq_{\alpha+1} \tau$. Then $\s\preceq_{\alpha}\tau$, which by (b)$_{\alpha}$ implies that $\s^{(\alpha)}\preceq \tau^{(\alpha)}$, and so $(\s^{(\alpha)})'\subseteq (\tau^{(\alpha)})'$. Suppose that $p(\tau^{(\alpha)})< p(\s^{(\alpha)})$. Then $\tau^{(\alpha)}$ is finite, and $p(\tau^{(\alpha)})\notin (\s^{(\alpha)})'$, so $e = p(\tau^{(\alpha)})$ violates the definition of $\s\preceq_{\alpha+1} \tau$. 

    For $(\Club)_{\alpha}$, suppose that $\s\preceq_\alpha \rho \preceq_\alpha \tau$ and $\s\preceq_{\alpha+1} \tau$. By (b)${}_{\alpha}$, $\s^{(\alpha)}\preceq \rho^{(\alpha)} \preceq \tau^{(\alpha)}$, and so $(\s^{(\alpha)})'\subseteq (\rho^{(\alpha)})'\subseteq (\tau^{(\alpha)})'$. Suppose that $e\in (\rho^{(\alpha)})'\setminus (\s^{(\alpha)})'$. Then $e\in (\tau^{(\alpha)})'\setminus (\s^{(\alpha)})'$; since $\s\preceq_{\alpha+1} \tau$, $e> p(\s^{(\alpha)})$, so $\s\preceq_{\alpha+1} \rho$. 

    \medskip
    
    By induction on~$\alpha\le \delta$ we prove that both (a)$_{\alpha}$ and $(\Diamond)_{\alpha}$ hold. 

    \medskip
    
    For $\alpha=0$ this is easy. For limit~$\alpha$, both (a)$_{\alpha}$ and $(\Diamond)_{\alpha}$ follow from the inductive assumption and the fact that the relation $\preceq_{\alpha}$ is the intersection of the relations $\preceq_{\beta}$ for $\beta<\alpha$.

    \smallskip
    
    It remains to check the successor case. Let $\alpha<\delta$, and suppose that (a)$_{\alpha}$ and $(\Diamond)_{\alpha}$ hold (and so also (b)$_{\alpha}$, (c)$_{\alpha}$ and $(\Club)_{\alpha}$). We verify that (a)$_{\alpha+1}$ and $(\Diamond)_{\alpha+1}$ hold as well. 

    \smallskip
    
    For (a)$_{\alpha+1}$, first observe that $\tau \preceq_{\alpha+1} \tau$ follows from $\tau \preceq_\alpha \tau$ and $(\tau^{(\alpha)})' = (\tau^{(\alpha)})'$. We check transitivity of~$\preceq_{\alpha+1}$. Suppose that $\s \preceq_{\alpha+1} \rho  \preceq_{\alpha+1} \tau$. Then $\s \preceq_{\alpha} \rho \preceq_{\alpha}\tau$, and so by (a)$_{\alpha}$, $\s \preceq_{\alpha} \tau$. By (b)$_{\alpha}$, $\s^{(\alpha)}\preceq \rho^{(\alpha)} \preceq \tau^{(\alpha)}$, and so $(\s^{(\alpha)})'\subseteq (\rho^{(\alpha)})' \subseteq (\tau^{(\alpha)})'$. Suppose that $e\in (\tau^{(\alpha)})'\setminus (\s^{(\alpha)})'$; we need to show that $e>p(\s^{(\alpha)})$. There are two cases. If $e\in (\rho^{(\alpha)})'$ then $e> p(\s^{(\alpha)})$ by the assumption $\s \preceq_{\alpha+1} \rho$. Otherwise, $e\in (\tau^{(\alpha)})'\setminus (\rho^{(\alpha)})'$ and so $e> p(\rho^{(\alpha)})$ by the assumption $\rho \preceq_{\alpha+1} \tau$. By (c)$_{\alpha}$ and the assumption $\s \preceq_{\alpha+1} \rho$ we get $p(\rho^{(\alpha)})\ge p(\s^{(\alpha)})$.

    \smallskip
    
     $(\Diamond)_{\alpha+1}$ follows from $(\Diamond)_{\alpha}$ and $(\Club)_{\alpha}$. Suppose that $\s\preceq \rho \preceq \tau$ and that $\s,\rho \preceq_{\alpha+1} \tau$. By $(\Diamond)_{\alpha}$, $\s\preceq_{\alpha} \rho$. By $(\Club)_{\alpha}$, $\s\preceq_{\alpha+1} \rho$. 
\end{proof}

We can conclude:

\begin{lemma} \label{lem:almost_TSP_successor}
    For all $\alpha<\delta$, if $\s\preceq_\alpha\tau$ then $\s\preceq_{\alpha+1} \tau$ if and only if $p(\rho^{(\alpha)})\ge p(\s^{(\alpha)})$ for all finite~$\rho$ with $\s\preceq_\alpha \rho \preceq_\alpha \tau$. 
\end{lemma}

Note that this almost gives us \ref{TSP:successor}, by letting $p_{\alpha+1}(\s)= p(\s^{(\alpha)})$; what remains is showing that $\s\mapsto \s^{(\alpha)}$ is computable. 

\begin{proof}
    We have required that convergence of the universal machine is delayed so that when we extend an oracle by one entry, at most one halting computation is added to the jump. If $\rho_0$ is an immediate $\prec_\alpha$-predecessor of~$\rho_1$, then $|\rho_1^{(\alpha)}| \le |\rho_0^{(\alpha)}|+1$. Hence,
  \[
  (\tau^{(\alpha)})'\setminus (\s^{(\alpha)})' = \big\{ p(\rho^{(\alpha)}) \,:\,  \s \prec_{\alpha} \rho \preceq_{\alpha} \tau,\,\rho\text{ is finite, and }    p(\rho^{(\alpha)}) \ne p(\s^{(\alpha)})  \big\}. \qedhere
  \]
\end{proof}

Structurally, the transitivity of $\preceq_{\alpha}$ and the property $(\Diamond)_{\alpha}$ together say that $(\w^{\le \omega},\preceq_{\alpha})$ is a forest; every~$\s$ has height at most~$|\s|$ in that forest. We shall soon verify that this forest is in fact a tree (it has a single root), and that each infinite~$x$ has height~$\w$ in this tree.

\subsubsection*{When we say nothing} 

For a string~$\s$ and $\alpha\le \delta$ we let $|\s|_\alpha$ be the number of (proper) $\prec_\alpha$-predecessors of~$\s$, in other words, its height in the forest $(\w^{\le \w}, \preceq_\alpha)$.

\begin{lemma} \label{lem:triviality_one_step}
    Suppose that $\alpha<\delta$, $\s\preceq_\alpha \tau$, and $|\s|_\alpha \le n_\alpha$. Then $\s\preceq_{\alpha+1} \tau$. 
\end{lemma}

\begin{proof}
    The assumption implies that $\s^{(\alpha)}$ is the empty string, and so that $p(\s^{(\alpha)}) = -1$. 
\end{proof}

By induction on $\alpha \le \delta$, we see that for all $\tau\in \w^{\le \omega}$, $\emptystring\preceq_{\alpha} \tau$, where again~$\emptystring$ denotes the empty string. Together with $(\Diamond)_{\alpha}$ of \cref{lem:first_pass_relations:basic_properties}, this gives \ref{TSP:trees}. We can also deal with the limit case:

\begin{proof}[Proof of \ref{TSP:limit}]
    Let $\lambda \le \delta$ be a limit ordinal. The first part is by definition of $\preceq_\lambda$. We verify the rest. Starting with any $\lambda_0<\lambda$, recursively define~$\lambda_k$ by letting
    \[
        \lambda_k = \max \{ \lambda_{k-1}+1, \alpha\,:\, \alpha< \lambda \andd n_\alpha \le k\}. 
    \]
    The definition ensures that $\seq{\lambda_k}$ is cofinal in~$\lambda$ and computable. Now we show:
    \begin{itemize}
        \item[$(*)$] for all $\s\in \w^{<\w}$, if $|\s|_{\lambda_k} \le k$, then for all~$\tau$,  $\s\preceq_\lambda \tau \,\Iff\,  \s\preceq_{\lambda_k} \tau$. 
    \end{itemize}
    To see this, let $k<\w$ and suppose that $|\s|_{\lambda_k}\le k$ and $\s\preceq_{\lambda_k} \tau$. By induction on $\alpha\in [\lambda_k,\lambda]$, we show that $\s\preceq_\alpha \tau$. If~$\alpha$ is limit we use the definition of~$\preceq_\alpha$. Suppose that $\lambda_k \le \alpha<\lambda$ and $\s\preceq_\alpha \tau$. By choice of~$\lambda_k$, $n_\alpha \ge k$. Since $\alpha \ge \lambda_k$, $|\s|_\alpha \le |\s|_{\lambda_k}\le k$. Hence, by \cref{lem:triviality_one_step}, $\s\preceq_{\alpha+1}\tau$. This establishes~$(*)$. 

    To prove \ref{TSP:limit}, let $k<\w$; we show that if $|\s|_{\lambda_k}>k$ then $|\s|_\lambda>k$. For suppose that $|\s|_{\lambda_k}>k$. Then there are $k+1$ distinct $\rho\prec_{\lambda_k} \s$ with $|\rho|_{\lambda_k} \le k $. For each such~$\rho$, by~$(*)$, we have $\rho\prec_\lambda \s$. Hence $|\s|_{\lambda}>k$. 
\end{proof}

\subsubsection*{Computability} 

Restricted to finite strings, the true stage relations are computable. As discussed above, this is uniform in~$\alpha$. 

\begin{proof}[Proof of \ref{TSP:computable}]
    The relations $\s\preceq_{\alpha}\tau$ for $\alpha\le \delta$ and the functions $\tau\mapsto \tau^{(\alpha)}$ for $\alpha<\delta$ are computed by simultaneous recursion on~$|\tau|$ (note, this is not effective transfinite recursion on~$\alpha$). If we have decided, for all $\s\preceq \tau$, whether $\s\preceq_\alpha \tau$, then we can compute $\tau^{(\alpha)}$. The definition of the relation $\preceq_{\alpha+1}$ shows that we can then decide whether $\s\preceq_{\alpha+1} \tau$ or not. 

    Now given $\s\preceq \tau$, the algorithm for computing whether $\s\preceq_\alpha \tau$ is given with the aid of \cref{lem:triviality_one_step}. Enumerate the set $\{\beta<\delta\,:\, n_\beta \le |\s|\}\cup \{0,\delta\}$ as $0=\beta_{0} < \beta_{1} < \dots <\beta_{k} = \delta$. Then $\s\preceq_{\beta_{0}} \tau$. If $i<k$ and we have decided that $\s \preceq_{\beta_{i}} \tau$, then we check if $\s\preceq_{\beta_{i}+1} \tau$. If so, then we can conclude that $\s\preceq_{\beta_{i+1}} \tau$: this uses $|\s|\ge |\s|_\beta$ for all~$\beta$, so for all $\beta\in (\beta_i,\beta_{i+1})$ we have $|\s|_\beta \le n_\beta$. 
\end{proof}

As mentioned above (after \cref{lem:almost_TSP_successor}), \ref{TSP:successor} follows as well. 

\begin{remark} \label{rmk:maximal_ordinal}
  The continuity of the relations $\preceq_{\alpha}$ implies that for all $\s\preceq \tau \in \w^{\le \omega}$, $\max \{\gamma\le \delta\,:\, \s\preceq_{\gamma} \tau\}$ must exist. The function $(\s,\tau)\mapsto \max \{\gamma\le \delta\,:\, \s\preceq_{\gamma} \tau\}$ is computable for finite~$\s$ and~$\tau$ (by a search over~$\alpha$, or by following the algorithm described in the proof above). 
\end{remark}

\subsubsection*{True stages} 

For $x \in \Baire$ and $\alpha\le \delta$ let 
\[
    D^\alpha_x = \left\{ \s \,:\,  \s\prec_\alpha x \right\}. 
\]
So $x^{(\alpha)}$ is the increasing enumeration of $D^\alpha_x$ after removing the first~$n_\alpha$ many strings. If~$\alpha$ is a limit ordinal then $D^\alpha_x = \bigcap_{\beta<\alpha} D^\beta_x$. If~$D^{\alpha}_x$ is infinite, then $x^{(\alpha)} = \bigcup_{\s\in D^{\alpha}_x} \s^{(\alpha)}$.

\begin{lemma} \label{lem:whats_true}
    Suppose that $D^\alpha_x$ is infinite, and let $\s\in D^\alpha_x$.  The following are equivalent:
    \begin{equivalent}
        \item $\s\in D^{\alpha+1}_x$;
        \item for all $\tau\succeq_\alpha \s$ in~$D^\alpha_x$, $\s\preceq_{\alpha+1}\tau$;
        \item for infinitely many~$\tau\succeq_\alpha \s$ in~$D^\alpha_x$, $\s\preceq_{\alpha+1} \tau$. 
    \end{equivalent}
\end{lemma}

\begin{proof}
    The implication (1)$\then$(2), follows from the property \RefClub. For (2)$\then$(3), note that if $\tau\in D^\alpha_x$ and $\s\preceq\tau$ then $\s\preceq_\alpha \tau$ (\ref{TSP:trees}).   For (3)$\then$(1), we use the fact that $(x^{(\alpha)})' = \bigcup_{\tau\in D^\alpha_x} (\tau^{(\alpha)})'$. 
    So if $\s\nprec_{\alpha+1} x$ then there is some $e\in (x^{(\alpha)})'\setminus (\s^{(\alpha)})'$ with $e< p(\s^{(\alpha)})$; for large enough~$\tau$ in~$D^\alpha_x$, $e\in (\tau^{(\alpha)})'$ and so $\s\npreceq_{\alpha+1}\tau$. 
\end{proof}

\begin{proposition} \label{prop:Dalpha_infinite}
    For every $x \in \Baire$ and $\alpha\le\delta$, $D^\alpha_x$ is infinite. 
\end{proposition}

\begin{proof}
    By induction on~$\alpha$. $D^0_x$ is the set of all finite initial segments of~$x$. At successor levels we use non-deficiency stages (following Dekker in \cite{Dekker:Hypersimple}). Namely, let $m\in \Nat$; let~$\s$ in $D^\alpha_x$ be such that $|\s|\ge m$ and $p(\s^{(\alpha)})$ is minimal among $p(\tau^{(\alpha)})$ for all $\tau \in D^\alpha_x$ with $|\tau|\ge m$. By \cref{lem:almost_TSP_successor}, $\s\in D^{\alpha+1}_x$. 

    \smallskip

    Suppose that~$\alpha$ is a limit ordinal, and suppose that for all $\beta<\alpha$, $D^\beta_x$ is infinite. 
    Given $m\in \Nat$, we find some $\gamma<\alpha$ such that $n_\gamma\ge m$ and for all $\beta\in (\gamma,\alpha)$, $n_\beta>n_\gamma$. This can be done since~$\alpha$ is a limit ordinal (and~$\Nat$ is well-founded). We then let~$\s \prec x$ be the string in~$D^\gamma_x$ with $|\s|_\gamma =  n_\gamma$. Then $|\s|\ge m$ (as $|\s|\ge |\s|_\gamma$). For all $\beta\in [\gamma,\alpha)$, $|\s|_\beta \le n_\beta$ (as $|\s|_\beta \le |\s|_\gamma$ and $n_\gamma\le n_\beta$). So by induction on such~$\beta$, aided by \cref{lem:triviality_one_step}, we see that $\s\in D^\beta_x$. 
\end{proof}

We obtain:

\begin{proof}[Proof of \ref{TSP:unique_path}]
     This is proved by induction on~$\alpha$. At the successor step, if $\s\in D^\alpha_x\setminus D^{\alpha+1}_x$, then any infinite path above~$\s$ in $(\{ \rho : \rho\prec x\}, \preceq_{\alpha+1})$ must also induce an infinite path in $(\{ \s : \s\prec x\}, \preceq_\alpha)$, and thus must be a subset of~$D^\alpha_x$; but by \cref{lem:whats_true}, $\s\preceq_{\alpha+1}\tau$ for only finitely many $\tau\in D^\alpha_x$. 
\end{proof}

\begin{remark} 
    \Cref{prop:Dalpha_infinite} implies that for all $x\in \Baire$ and $\alpha<\delta$ $x^{(\alpha)}$, is an infinite sequence, and so $p(x^{(\alpha)}) = \infty$. As a result, in \ref{TSP:successor}, we do not need to restrict to finite~$\rho$. 
\end{remark}

\subsubsection{The effective Borel hierarchy} 

Recall that the~$\Sigma^0_\alpha$ sets, for $1\le \alpha \le \delta$, are defined by induction on~$\alpha$. The~$\Sigma^0_1$ sets are the effectively open sets. For $\alpha>1$, the $\Sigma^0_\alpha$ sets are those which are effective unions of sets which are~$\Pi^0_\beta$ for some~$\beta<\alpha$.\footnote{To make the notion of an effective union of $\Pi^0_{<\alpha}$ sets precise, we need to define an effective indexing of each class. For each~$\alpha\le \delta$ with $\alpha>0$, let $\seq{W_{e,\alpha}}$ be an effective enumeration of the c.e.\ subsets of $\w^{<\w}\times \alpha$. We let $A_{(e,1)} =  \left\{ x\in \Baire \,:\,  (\exists \s\prec x)\,\,(\s,n_0)\in W_{e,1} \right\}$. For $\alpha>1$ we let $A_{(e,\alpha)} = \bigcup \big\{ \Baire\setminus A_{(i,\beta)} \,:\,  (i,n_\beta)\in W_{(e,\alpha)}  \big\}$. The sets $A_{e,\alpha}$ form an effective enumeration of the $\Sigma^0_\alpha$ sets.}

We can now verify the last of the properties listed in \cref{sec:preliminaries}. We use the notation of that section regarding $\alpha$-forests and the subsets of Baire space they define. For $\s\in \w^{<\w}$ we let $[\s]_\alpha = \left\{ x\in \Baire \,:\,  \s\prec_\alpha x \right\}$; this is $[W]_\alpha$, where~$W$ consists of the $\preceq_\alpha$-extensions of~$\s$. 

\begin{proof}[Proof of \ref{TSP:Sigma_alpha_sets}]
     We prove \ref{TSP:Sigma_alpha_sets} by effective transfinite induction on~$\alpha$. Thus, it will be important that we work uniformly: by recursion on~$\alpha$, we show how to effectively translate between indices of $\Sigma^0_{1+\alpha}$ sets and c.e.\ indices of subsets of~$\w^{<\w}$. 

     \medskip

     For $\alpha=0$, \ref{TSP:Sigma_alpha_sets} is immediate, as $\preceq_0$ is~$\preceq$, and the $\Sigma^0_1$ sets are the effectively open ones. 

     \medskip
     
    Suppose that $\alpha\le \delta$ is a limit ordinal. 

    In one direction, since $\Sigma^0_{1+\alpha}$ is closed under effective unions, it suffices to show that for all $\s\in \w^{<\w}$, $[\s]_\alpha$ is $\Sigma^0_{1+\alpha}$. \ref{TSP:limit} implies that it is in fact $\Sigma^0_{<\alpha}$.

    In the other direction, since $1+\alpha = \alpha$, and since c.e.\ sets are closed under effective unions, it suffices to show that every $\Sigma^0_{<\alpha}$ set is $[U]_\alpha$ for some c.e., $\preceq_\alpha$-upwards closed~$U$. Let~$A$ be $\Sigma^0_{1+\beta}$ for some $\beta<\alpha$. By induction, there is a c.e.\ set~$U$, upwards closed in $\preceq_{\beta}$, such that $A = [U]_{\beta}$. Then~$U$ is also $\preceq_\alpha$-upwards closed, and $[U]_\alpha = [U]_{\beta} = A$ (see \cref{rmk:alpha_and_beta_forests}). 


    \medskip
     
    For the successor case, suppose that~$\alpha<\delta$. In one direction, it suffices to show that for all $\s\in \w^{<\w}$, $[\s]_{\alpha+1}$ is $\Sigma^0_{1+\alpha+1}$. In fact, it is $\Pi^0_{1+\alpha}$. To see this, let~$T$ be the collection of $\tau\in \w^{<\w}$ such that either $\tau\preceq_\alpha \s$ or $\s\preceq_{\alpha+1}\tau$. Then~$T$ is computable; \RefClub{} implies that~$T$ is $\preceq_\alpha$-downwards closed (it is a subtree of $(\w^{<\w},\preceq_\alpha)$), so $[T]_\alpha$ is $\Pi^0_{1+\alpha}$. \Cref{lem:whats_true} implies that $[\s]_{\alpha+1} = [T]_\alpha$.

     In the other direction, it suffices to show that every $\Pi^0_{1+\alpha}$ set~$A$ is $[U]^\prec_{\alpha+1}$ for some c.e.~$U$, upwards closed in $\preceq_{\alpha+1}$. Given such~$A$, by induction, let~$T$ be a computable subtree of $(\w^{<\w},\preceq_\alpha)$ with $A = [T]_\alpha$. Note that for all $x\in \Baire$, $x\in A$ if and only if for all $\s$ listed in $x^{(\alpha)}$, $\s\in T$. This is a $\Pi^0_1$ property of $x^{(\alpha)}$, so there is some $e\in \Nat$ such that for all $x\in \Baire$, $x\in A \Iff e\notin (x^{(\alpha)})'$. Let $W = \left\{ \s\in \w^{<\w} \,:\,  e\notin (\s^{(\alpha)})' \andd p(\s^{(\alpha)})>e   \right\}$. Then~$W$ is computable, and is $\preceq_{\alpha+1}$-upwards closed. We claim that $A = [W]_{\alpha+1}$. For all $\s\in W$, for all $x\in [\s]_{\alpha+1}$, $e\notin (x^{(\alpha)})'$ by definition of~$\preceq_{\alpha+1}$. On the other hand, if $e\notin (x^{(\alpha)})'$, find $\s\prec_{\alpha+1} x$ sufficiently long so that $p(\s^{(\alpha)})>e$; then $\s\in W$. 
 \end{proof}

\subsubsection{Relativising true stages} 

Given an oracle~$z$ and a $z$-computable ordinal~$\alpha$, we define the relation $\preceq^z_\alpha$ precisely as above; however, the universal machine which calculates the jumps of strings is given an extra oracle tape, containing~$z$. Write $(z,\s)^{(\alpha)}$ instead of $\s^{(\alpha)}$. 

For finite~$\s$, the universal machine computing $(z,\s)^{(\alpha)}$ is stopped after $|\s|$ many steps, and so consults only $z\rest{|\s|}$. If~$\alpha$ is computable (rather than~$z$-computable), this shows that for finite~$\tau$, determining whether $\s\preceq^z_\alpha \tau$ depends only on $z\rest{|\tau|}$, as promised in \cref{sec:preliminaries}. 

Further, the easier directions in the proof of \ref{TSP:Sigma_alpha_sets} give:

\begin{lemma} \label{lem:Sigma_alpha_subsets_of_product_space:one_way}
    Let~$\alpha$ be a computable ordinal. For any $\tau\in \w^{<\w}$, the set
    \[
        \left\{ (z,x)\in \Baire^2 \,:\,  \tau \prec^z_\alpha x \right\}
    \]
is $\Sigma^0_{1+\alpha}$.
\end{lemma}

As usual, this is uniform in~$\alpha$, $z$ and~$\tau$.

\subsection{Interlude: the Kuratowski-Sierpi\'nski theorem} 
\label{sub:interlude_the_kuratowski_sierpi}

We still need to prove \cref{prop:approximations_and_alpha_plus_one_computability,lem:existence_of_a_computable_ordinal_witness,prop:translating_true_stages_between_oracles,prop:translating_true_stages_across_computable_functions,prop:translating_true_stages_between_copies_of_the_ordinal,prop:translating_true_stage_relations_by_unpairing}. Before we do so, we briefly comment on iterated Turing jumps and the Kuratowski-Sierpi\'nski characterisation of the classes $\bSigma^0_{\xi}$. 

In \cite{BSL_paper}, we justified \ref{TSP:Sigma_alpha_sets} as follows. We started with the aforesaid characterisation.

\begin{theorem}[Kuratowski \cite{Kuratowski:33}, Sierpi\'nski \cite{Sierpinski:33}] \label{thm:KuratowskiSierpinski}
Let $\xi\ge 1$ be a countable ordinal. A set $A\subseteq \Baire$ is $\bSigma^0_{\xi}$ if and only if there is a closed set $K\subseteq \Baire$, a bijection $h\colon \Baire\to K$, and an open set $O\subseteq \Baire$, such that:
\begin{orderedlist}
    \item $h$ is $\bSigma^0_{\xi}$-measurable;
    \item $h^{-1}$ is continuous; and
    \item $A = h^{-1}[O]$. 
\end{orderedlist}
\end{theorem}

This theorem allows us to prove facts about $\bSigma^0_\xi$ sets by changing the topology using a ``generalised homeomorphism'' such as~$h$. For example, this allowed Kuratowski to extend Hausdorff's analysis of the $\bDelta^0_2$ sets to the classes $\bDelta^0_{\xi+1}$. 

In \cite{BSL_paper}, we stated that \cref{thm:KuratowskiSierpinski} can be effectivised; we can take~$h$ to be an iteration of the Turing jump of length~$\alpha$, where $\xi = 1+\alpha$. We then argue that $D^\alpha_x$, as defined above, is (uniformly) Turing equivalent to this iteration of the jump starting with~$x$. We can then pull back the open set~$O$ to obtain a c.e.\ set~$U$ satisfying $A = [U]_\alpha$. 

Here, rather, we directly derived \ref{TSP:Sigma_alpha_sets} from the other basic properties of the true stage relations, and the definition of the classes $\Sigma^0_{1+\alpha}$. We can now argue in the other direction, deriving an effective version of \cref{thm:KuratowskiSierpinski}, which in turn, by relativisation, implies \cref{thm:KuratowskiSierpinski}.

For a computable ordinal~$\alpha$ and $x\in \Baire$, let $x^{\{\alpha\}}$ be the increasing enumeration of $D^\alpha_x = \left\{ \s \,:\,  \s\prec_\alpha x  \right\}$.\footnote{The reason we use $x^{\{\alpha\}}$ rather than $x^{(\alpha)}$ is that the latter does not depend solely on~$\alpha$, as it requires~$n_\alpha$. That is, it also depends on the upper bound of~$\alpha$ in an ordinal $\alpha+1>\alpha$. Observe, though, that once such~$\alpha+1$ is chosen, we get $x^{(\alpha)}\equiv_{\scriptsize\textrm T} x^{\{\alpha\}}$, uniformly: given the least element of $x^{(\alpha)}$ we can computably list its $n_\alpha$-many $\preceq_\alpha$-predecessors.}

\begin{lemma} \label{lem:modified_TSP7_for_open_sets}
    Let~$\alpha$ be a computable ordinal. A set $A\subseteq \Baire$ is $\Sigma^0_{1+\alpha}$ if and only if there is an effectively open set $O\subseteq (\w^{<\w})^\w$ such that for all~$x$, $x\in A\Iff x^{\{\alpha\}}\in O$. 
\end{lemma}

\begin{proof}
    We of course use \ref{TSP:Sigma_alpha_sets}. Suppose that $A = [U]_\alpha$ for some c.e., $\preceq_\alpha$-upwards-closed~$U$. Let $O = \left\{ y\in (\w^{<\w})^\w \,:\,  \range y\cap U\ne 0 \right\}$. Then $x\in A\Iff x^{\{\alpha\}}\in O$. 

    In the other direction, let~$O$ be an effectively open subset of $(\w^{<\w})^\w$; so the set of $\mu\in (\w^{<\w})^{<\w}$ such that $[\mu]\subseteq O$ is c.e. Let~$U$ be the set of $\s\in \w^{<\w}$ such that the increasing enumeration~$\mu$ of $\left\{ \rho \,:\,  \rho\preceq_\alpha \s \right\}$ satisfies $[\mu]\subseteq O$. Then $x^{\{\alpha\}}\in O \Iff   x\in [U]_\alpha$. 
\end{proof}

\begin{lemma} \label{lem:closed_graph_of_x_alpha}
    Let $\alpha$ be a computable ordinal. The map $x\mapsto x^{\{\alpha\}}$:
    \begin{sublemma}
        \item has $\Pi^0_1$ graph and $\Pi^0_1$ image; 
        \item has computable (effectively continuous) inverse;
        \item is $\Sigma^0_{1+\alpha}$-measurable. 
    \end{sublemma}
    Indeed, the map $x\mapsto x^{\{\alpha\}}$ is universal for $\Sigma^0_{1+\alpha}$-measurable functions: for any computable Polish space~$X$ and any  $\Sigma^0_{1+\alpha}$-measurable function $h\colon \Baire\to X$, there is a partial computable map $g \colon (\w^{<\w})^\w\to X$ such that $h(x) = g(x^{\{\alpha\}})$.
\end{lemma}

\begin{proof}
    (a) follows from \ref{TSP:unique_path}: Let $K = \left\{ x^{\{\alpha\}} \,:\, x\in \Baire   \right\}$. Then $y\in K$ if and only if~$y$ is an infinite path in $(\w^{<\w},\preceq_\alpha)$; $y = x^{\{\alpha\}}$ if and only if $y\in K$ and $x = \bigcup_n y(n)$.\footnote{Note that $y=D^\alpha_x$ is $\Pi^0_2$: we need to say that it is infinite.} (b) is immediate. (c) follows from \cref{lem:modified_TSP7_for_open_sets}. 

    For universality, let $h\colon \Baire\to X$ be $\Sigma^0_{1+\alpha}$-measurable. Let $W_0,W_1,\dots$ be an effective listing of the effectively open subsets of~$X$; so $h^{-1}[W_e]$ are uniformly $\Sigma^0_{1+\alpha}$; by \cref{lem:modified_TSP7_for_open_sets}, there are uniformly effectively open sets $O_e$ with $h(x)\in W_e \Iff x^{\{ \alpha\}}\in O_e$. For $y\in (\w^{<\w})^\w$, let $g(y)$ be the unique element of $\bigcap \left\{ W_e  \,:\, y\in O_e  \right\}$ if there is such an element, undefined otherwise. 
\end{proof}

\Cref{thm:KuratowskiSierpinski} now follows by relativising to an oracle. 

\medskip

The line of argument just given does not mention the concept of an iteration of the Turing jump, but of course, this concept was the original motivation for the entire machinery. For a computable ordinal~$\delta$ and $x\in \Baire$, define $H_\alpha(x)$ in Cantor space by induction on~$\alpha\le \delta$: $H_0(x)$ is the set of (codes of) the finite initial segments of~$x$; $H_{\alpha+1}(x) = (H_\alpha(x))'$; for limit $\alpha\le \delta$, $H_\alpha(x) = \bigoplus_{\beta<\alpha}H_\beta(x)$.\footnote{To be precise, $H_{\alpha}(x) = \left\{ (n_\beta,m) \,:\,  \beta<\alpha\andd m\in H_\beta(x) \right\}$.} 

\begin{proposition} \label{prop:the_sigmas_as_jumps} 
    For all computable $\alpha$ and all $x\in \Baire$, $H_\alpha(x)\equiv_\Tur x^{\{\alpha\}}$.  
\end{proposition}

This is of course uniform in~$\alpha$ and~$x$. 

\begin{proof}
    We use effective transfinite recursion on~$\alpha$. Since $x^{\{\alpha\}}\equiv_\Tur D^\alpha_x$, we show that $D^\alpha_x \equiv_\Tur H_\alpha(x)$. 

    If~$\alpha$ is a limit, we need to show that $D^\alpha_x$ is Turing equivalent to $\bigoplus_{\beta<\alpha} D^\beta_x$. If $\beta<\alpha$ then $ D^\beta_x \le_\Tur D^\alpha_x$: $\s\prec_\beta x$ if $\s\prec_\beta \tau$ for some or all $\tau\succeq \s$ in $D^\alpha_x$. On the other hand, given $D^\beta_x$ for all $\beta<\alpha$, we can tell whether a given~$\s$ is in $D^\alpha_x$ by using \cref{lem:triviality_one_step}: we let $\beta = \max \{\gamma<\alpha\,:\, n_\gamma \le |\s| \}$, and ask whether $\s\in D^\beta_x$. 

    For the successor case, we need to check that $D^{\alpha+1}_x \equiv_\Tur (D^\alpha_x)'$. Note that $(x^{(\alpha)})'\equiv_\Tur (D^\alpha_x)'$ (as $x^{(\alpha)}\equiv_\Tur D^\alpha_x$). 
    Given $(x^{(\alpha)})'$ we first compute~$D^\alpha_x$. For $\s\in D^\alpha_x$ we can check if there is some $e<p(\s^{(\alpha)})$ in $(x^{(\alpha)})'\setminus (\s^{(\alpha)})'$ and so tell if $\s\prec_{\alpha+1} x$. In the other direction, given $D^{\alpha+1}_x$, we can compute $(x^{(\alpha)})'$ by observing $(\s^{(\alpha)})'\rest{p(\s^{(\alpha)})}$ for $\s\in D^{\alpha+1}_x$.
\end{proof}

\subsection{Final debts} 
\label{sub:final_debts}

We give proofs of the remaining propositions of \cref{sec:preliminaries}.

\begin{proof}[Proof of \cref{prop:approximations_and_alpha_plus_one_computability}]
    In one direction, suppose that $f\colon S\to \Nat$ is a computable $\alpha$-approximation of~$F$. For each~$n$ and~$k$, let 
     \[
     A_{n,k} = \bigcup \left\{ [\s]_\alpha \,:\, \s\in S, \,f(\s) = n\andd |\s|_\alpha = k  \right\}.
     \]
     By \ref{TSP:Sigma_alpha_sets}, and since~$f$ and~$S$ are computable, the sets $A_{n,k}$ are uniformly $\Delta^0_{1+\alpha}$. For $x\in [S]_\alpha$, $F(x)=n$ if and only if $x\in A_{n,k}$ for almost all~$k$, hence $F^{-1}\{n\}$ is $\Sigma^0_{1+\alpha+1}$ within $[S]_\alpha$. 

     In the other direction, let $F\colon [S]\to \Nat$ be $\Sigma^0_{1+\alpha+1}$-measurable. By \ref{TSP:Sigma_alpha_sets} and \cref{lem:producing_pairwise_orthogonal}, let $V_0$, $V_1,\dots$ be uniformly computable subsets of~$S$, $\preceq_{\alpha+1}$-upwards closed in~$S$, such that $F^{-1}\{n\} = [V_n]_{\alpha+1}$ and $\bigcup_n V_n$ computable. 

     Define $f\colon S\to \Nat$ by letting $f(\s)=n$ if $\s\in V_n$, $f(\s)=0$ if $\s \notin\bigcup_n V_n$. Then~$f$ is computable. Let $x\in [S]$; let $n = F(x)$. There is some $\s\prec_{\alpha+1} x$ with $\s\in V_n$. Let $\tau$ be such that $\s\prec_\alpha \tau \prec_\alpha x$. By \RefClub, $\s\preceq_{\alpha+1} \tau$. Since~$V_n$ is $\preceq_{\alpha+1}$-upwards closed, $\tau\in V_n$, so $f(\tau)=n$. 
\end{proof}

\begin{proof}[Proof of \cref{lem:computable_rank_function}]
    Let~$<_{\KB}$ denote the Kleene-Brouwer ordering on $\w^{<\w}$. If~$T$ is a well-founded tree, then the restriction of $<_{\KB}$ to~$T$ is a $T$-computable well-ordering, and the identity function on~$T$ is a rank function. This well-ordering may fail to be a $T$-computable ordinal, as~$T$ may fail to compute the collection of limit points and the successor relation. This is overcome by multiplying on the left by~$\w$, that is, replacing every point by a copy of~$\w$. 
\end{proof}

\begin{proof}[Proof of \cref{lem:existence_of_a_computable_ordinal_witness}]
    For nonempty $\s\in \w^{<\w}$, let $\s^-$ denote the immediate predecessor of~$\s$ on the tree $(\w^{<\w},\preceq_\alpha)$. Let~$R$ be the collection of all $\s\in S$ such that either: 
    \begin{itemize}
        \item $\s = \emptystring$ or $\s^{-}\notin S$; or
        \item $\s^-\in S$ and $f(\s^-)\ne f(\s)$. 
    \end{itemize}
    Then $(R,\preceq_\alpha)$ is a computable well-founded tree. By \cref{lem:computable_rank_function}, let~$\eta$ be a computable ordinal and $r\colon R \to \eta $ be a computable rank function. For $\s\in S$, let $\beta(\s)= r(\tau)$ for the longest $\tau\preceq_\alpha \s$ in~$R$. 
\end{proof}

\begin{proof}[Proof of \cref{prop:translating_true_stages_between_oracles}]
    Since $w\ge_\Tur z$, every $z$-c.e.\ set is $w$-c.e., uniformly in their indices. This, together with effective transfinite recursion on~$\alpha$, shows that every $\Sigma^0_{1+\alpha}(z)$ set is $\Sigma^0_{1+\alpha}(w)$, again uniformly in the indices. Using \ref{TSP:Sigma_alpha_sets}, we see that we can obtain uniformly computable, $\preceq^w_\alpha$-upwards closed sets~$U_\s$ for $\s\in \w^{<\w}$, such that $[U_\s]^w_\alpha = [\s]^z_\alpha$ for all~$\s$. 
    We may assume that $U_{\emptystring} = \w^{<\w}$. 

    We define the function $h(\s)$ by recursion on~$\preceq^w_\alpha$. Of course we set $h(\emptystring) = \emptystring$. Suppose that $\s\ne \emptystring$; let $\s^-$ be $\s$'s immediate $\preceq^w_\alpha$-predecessor, and suppose that $h(\s^-)$ has already been defined. 

    If there is some $\rho\preceq \s$ such that $\s\in U_\rho$ and~$\rho$ is a $\preceq^z_\alpha$-immediate successor of~$h(\s^-)$ then we let $h(\s)=\rho$ for some such~$\rho$ (if there is more than one such~$\rho$ then $[\s]^w_\alpha = \emptyset$, so it doesn't matter which one we choose). If there is no such~$\rho$ we let $h(\s) = h(\s^-)$. 

    Note that by our construction, $h(\s)\preceq \s$ and $\s\in U_{h(\s)}$ for all~$\s$, i.e., $[\s]^w_\alpha \subseteq [h(\s)]^z_\alpha$. Let $x\in \Baire$. We need to argue that for all $\tau\prec^z_\alpha x$ there is some $\s\prec^w_\alpha x$ such that $\tau= h(\s)$. This is done by induction on~$\tau$. Let $\tau^-$ be $\tau$'s $\preceq^z_\alpha$-immedate predecessor, and let $\rho\prec^w_\alpha x$ such that $h(\rho) = \tau^-$. Since $x\in [\tau]^z_\alpha$, there is some $\s\prec^w_\alpha x$ with $\s\in U_\tau$. If we take $\s$ minimal such that $\rho \prec^w_\alpha \s\prec^w_\alpha x$ and $\s\in U_\tau$, then $h(\s) = \tau$. 
\end{proof}

\Cref{prop:translating_true_stages_across_computable_functions,prop:translating_true_stages_between_copies_of_the_ordinal} are proved in exactly the same way. For the former, we use the fact that the pull-back of a $\Sigma^0_{1+\alpha}$ set by a computable function is $\Sigma^0_{1+\alpha}$, uniformly. For the latter, we use the fact that if~$\alpha$ and~$\beta$ are computably isomorphic, then $\Sigma^0_{1+\alpha} = \Sigma^0_{1+\beta}$, uniformly in the indices. (We know that if~$\alpha$ and~$\beta$ are isomorphic, then $\Sigma^0_{1+\alpha} = \Sigma^0_{1+\beta}$; but if the isomorphism is not computable, then this will not be uniform.) \Cref{prop:translating_true_stage_relations_by_unpairing} is proved similarly, using \cref{lem:Sigma_alpha_subsets_of_product_space:one_way}.


\bibliography{LouveauSR_Wadge}
\bibliographystyle{alpha}

\end{document}